\documentclass{amsart}
\usepackage{amsmath,amssymb,amsthm,bm,bbm}
\usepackage{amscd}
\usepackage{graphicx,enumerate}
\usepackage{multirow}
\usepackage{layout}
\usepackage[all]{xy}
\usepackage[OT2,T1]{fontenc}
\DeclareSymbolFont{cyrletters}{OT2}{wncyr}{m}{n}
\DeclareMathSymbol{\Sha}{\mathalpha}{cyrletters}{"58}
\usepackage[OT2,OT1]{fontenc}
\newcommand\cyr{\renewcommand\rmdefault{wncyr}
\renewcommand\sfdefault{wncyss}
\renewcommand\encodingdefault{OT2}
\normalfont\selectfont}
\DeclareTextFontCommand{\textcyr}{\cyr}

\usepackage{hyperref}
\usepackage{url}

\theoremstyle{plain}
\newtheorem{theorem}{Theorem}[section]
\newtheorem*{theorem-nn}{Theorem}
\newtheorem{lemma}[theorem]{Lemma}
\newtheorem{proposition}[theorem]{Proposition}
\newtheorem*{proposition-nn}{Proposition}
\newtheorem{corollary}[theorem]{Corollary}

\theoremstyle{definition}

\newtheorem{example}[theorem]{Example}
\newtheorem{remark}[theorem]{Remark}

\theoremstyle{remark}

\newcommand{\bZ}{\mathbbm{Z}}\newcommand{\bQ}{\mathbbm{Q}}
\newcommand{\bC}{\mathbbm{C}}
\newcommand{\bG}{\mathbbm{G}}\newcommand{\bF}{\mathbbm{F}}
\newcommand{\bA}{\mathbbm{A}}

\newcommand{\GL}{{\rm GL}}\newcommand{\SL}{{\rm SL}}
\newcommand{\PGL}{{\rm PGL}}\newcommand{\PSL}{{\rm PSL}}
\newcommand{\PSU}{{\rm PSU}}\newcommand{\SU}{{\rm SU}}
\newcommand{\PGU}{{\rm PGU}}\newcommand{\GU}{{\rm GU}}
\newcommand{\Sp}{{\rm Sp}}\newcommand{\PSp}{{\rm PSp}}

\title{Hasse norm principle for metacyclic extensions with trivial Schur multiplier}

\author[A. Hoshi]{Akinari Hoshi}
\address{Department of Mathematics, Niigata University, Niigata 950-2181, Japan}
\email{hoshi@math.sc.niigata-u.ac.jp}

\author[A. Yamasaki]{Aiichi Yamasaki}
\address{Department of Mathematics, Kyoto University, Kyoto 606-8502, Japan}
\email{aiichi.yamasaki@gmail.com}

\thanks{
This work was partially supported by JSPS KAKENHI Grant Numbers 
19K03418, 20H00115, 20K03511, 24K00519, 24K06647.
}


\begin{document}
\maketitle
\begin{abstract}
Let $k$ be a global field, $K/k$ be a finite separable field extension 
and $L/k$ be the Galois closure of $K/k$ 
with Galois groups $G={\rm Gal}(L/k)$ and $H={\rm Gal}(L/K)\lneq G$. 
In 1931, Hasse proved that if $G$ is cyclic, 
then the Hasse norm principle holds for $K/k$. 
We show that if $G$ is metacyclic with trivial Schur multiplier $M(G)=0$, 
then $H$ is cyclic and the Hasse norm principle holds for $K/k$. 
Some examples of metacyclic, dihedral, quasidihedral, modular, 
generalized quaternion, extraspecial groups 
and $Z$-groups $G$ with trivial Schur multiplier $M(G)=0$ are given. 
These provide new examples which 
the Hasse norm principle hold for non-Galois extensions $K/k$ 
whose Galois closure is $L/k$ with metacyclic $G={\rm Gal}(L/k)$ and $M(G)=0$. 
\end{abstract}
\tableofcontents
%
\section{Introduction: Main theorem (Theorem \ref{thmain0})}\label{S1}

Let $k$ be a global field, 
i.e. a number field (a finite extension of $\bQ$) 
or a function field of an algebraic curve over 
$\bF_q$ (a finite extension of $\bF_q(t))$.
Let $K/k$ be a finite separable field extension and 
$\bA_K^\times$ be the idele group of $K$. 
We say that {\it the Hasse norm principle holds for $K/k$} 
if $(N_{K/k}(\bA_K^\times)\cap k^\times)/N_{K/k}(K^\times)=1$ 
where $N_{K/k}$ is the norm map. 
%
Hasse \cite[Satz, page 64]{Has31} proved that 
the Hasse norm principle holds for any cyclic extension $K/k$ 
but does not hold for bicyclic extension $\bQ(\sqrt{-39},\sqrt{-3})/\bQ$. 


Let $G$ be a finite group and $M(G)=H^2(G,\bC^\times)\simeq H^2(G,\bQ/\bZ)\xrightarrow[\sim]{\delta} 
H^3(G,\bZ)$ be the Schur multiplier of $G$ 
where $\delta$ is the connecting homomorphism 
(see e.g. Neukirch, Schmidt and Wingberg \cite[Chapter I, \S 3, page 26]{NSW00}). 
For Galois extensions $K/k$, Tate \cite{Tat67} proved: 
%

\begin{theorem}[{Tate \cite[page 198]{Tat67}}]\label{thTate}
Let $k$ be a global field, $K/k$ be a finite Galois extension 
with Galois group $G={\rm Gal}(K/k)$. 
Let $V_k$ be the set of all places of $k$ 
and $G_v$ be the decomposition group of $G$ at $v\in V_k$. 
Then 
\begin{align*}
(N_{K/k}(\bA_K^\times)\cap k^\times)/N_{K/k}(K^\times)\simeq 
{\rm Coker}\left\{\bigoplus_{v\in V_k}\widehat H^{-3}(G_v,\bZ)\xrightarrow{\rm cores}\widehat H^{-3}(G,\bZ)\right\}
\end{align*}
where $\widehat H$ is the Tate cohomology. 
In particular, the Hasse norm principle holds for $K/k$ 
if and only if the restriction map 
$H^3(G,\bZ)\xrightarrow{\rm res}\bigoplus_{v\in V_k}H^3(G_v,\bZ)$ 
is injective. 
In particular, if 
$H^3(G,\bZ)\simeq M(G)=0$, then 
the Hasse norm principle holds for $K/k$. 
\end{theorem}
Let $C_n$ be the cyclic group of order $n$ 
and 
$V_4\simeq C_2\times C_2$ be the Klein four group. 
If $G\simeq C_n$, then 
$\widehat H^{-3}(G,\bZ)\simeq H^3(G,\bZ)\simeq H^1(G,\bZ)=0$ in Theorem \ref{thTate} 
and hence Hasse's original theorem follows. 
If there exists a place $v$ of $k$ such that $G_v=G$, then 
the Hasse norm principle also holds for $K/k$. 
By Theorem \ref{thTate}, for example, the Hasse norm principle holds for $K/k$ with 
$G\simeq V_4=(C_2)^2$ if and only if 
there exists a (ramified) place $v$ of $k$ such that $G_v=V_4$ because 
$H^3(V_4,\bZ)\simeq\bZ/2\bZ$ and $H^3(C_2,\bZ)=0$. 
The Hasse norm principle holds for $K/k$ with 
$G\simeq (C_2)^3$ if and only if {\rm (i)} 
there exists a place $v$ of $k$ such that $G_v=G$ 
or {\rm (ii)} 
there exist (ramified) places $v_1,v_2,v_3$ of $k$ such that 
$G_{v_i}\simeq V_4$ and 
$H^3(G,\bZ)\xrightarrow{\rm res}
H^3(G_{v_1},\bZ)\oplus H^3(G_{v_2},\bZ)\oplus H^3(G_{v_3},\bZ)$ 
is an isomorphism because 
$H^3(G,\bZ)\simeq(\bZ/2\bZ)^{\oplus 3}$ and 
$H^3(V_4,\bZ)\simeq \bZ/2\bZ$. 

The Hasse norm principle for Galois extensions $K/k$ 
was investigated by Gerth \cite{Ger77}, \cite{Ger78} and 
Gurak \cite{Gur78a}, \cite{Gur78b}, \cite{Gur80} 
(see also \cite[pages 308--309]{PR94}). 
Gurak \cite[Corollary 2.2]{Gur78a} showed that 
the Hasse norm principle holds for a Galois extension $K/k$ 
if the restriction map 
$H^3(G_p,\bZ)\xrightarrow{\rm res}\bigoplus_{v\in V_k}H^3(G_p$ $\cap$ $G_v,\bZ)$ 
is injective for any $p\mid |G|$ where $G_p$ is a $p$-Sylow subgroup of $G={\rm Gal}(K/k)$. 
For example, if $H^3(G_p,\bZ)\simeq M(G_p)=0$ for any $p\mid|G|$, then 
the Hasse norm principle holds for $K/k$. 
In particular, 
because $H^3(C_n,\bZ)=0$, 
the Hasse norm principle holds for $K/k$ 
if all the Sylow subgroups of $G={\rm Gal}(K/k)$ are cyclic. 

Let $K/k$ be a separable field extension with $[K:k]=n$ 
and $L/k$ be the Galois closure of $K/k$.  
Let $G={\rm Gal}(L/k)$ and $H={\rm Gal}(L/K)$ with $[G:H]=n$. 
Then we have $\bigcap_{\sigma\in G} H^\sigma=\{1\}$ 
where $H^\sigma=\sigma^{-1}H\sigma$ and hence 
$H$ contains no normal subgroup of $G$ except for $\{1\}$. 
The Galois group $G$ may be regarded as a transitive subgroup of 
the symmetric group $S_n$ of degree $n$. 
We may assume that 
$H$ is the stabilizer of one of the letters in $G\leq S_n$, 
i.e. $L=k(\theta_1,\ldots,\theta_n)$ and $K=L^H=k(\theta_i)$ 
where $1\leq i\leq n$. 
%
Let $D_n$ be the dihedral group of order $2n$, 
$A_n$ be the alternating group of degree $n$
and 
$\PSL_n(\bF_q)$ be the projective special linear group of degree $n$ 
over the finite field $\bF_q$ of $q=p^r$ elements. 

For non-Galois extensions $K/k$ with $[K:k]=n$, 
the Hasse norm principle was investigated by 
Bartels \cite[Lemma 4]{Bar81a} (holds for $n=p$; prime), 
\cite[Satz 1]{Bar81b} (holds for $G\simeq D_n$), 
Voskresenskii and Kunyavskii \cite{VK84} (holds for $G\simeq S_n$ by $H^1(k,{\rm Pic}\, \overline{X})=0$), 
Kunyavskii \cite{Kun84} $(n=4$ (holds except for $G\simeq V_4, A_4))$, 
Drakokhrust and Platonov \cite{DP87} $(n=6$ (holds except for $G\simeq A_4, A_5))$, 
Endo \cite{End11} (holds for $G$ whose all $p$-Sylow subgroups are cyclic (general $n$), 
Macedo \cite{Mac20} (holds for $G\simeq A_n$ ($n\neq 4)$ by $H^1(k,{\rm Pic}\, \overline{X})=0$), 
Macedo and Newton \cite{MN22} 
($G\simeq A_4$, $S_4$, $A_5$, $S_5$, $A_6$, $A_7$ $($general $n))$, 
Hoshi, Kanai and Yamasaki \cite{HKY22} $(n\leq 15$ $(n\neq 12))$, 
(holds for $G\simeq M_n$ ($n=11,12,22,23,24$; $5$ Mathieu groups)), 
\cite{HKY23} $(n=12)$, 
\cite{HKY25} $(n=16$ and $G$; primitive$)$, 
\cite{HKY} $(G\simeq M_{11}$, $J_1$ $($general $n))$, 
Hoshi and Yamasaki \cite{HY2} (holds for $G\simeq \PSL_2(\bF_7)$ $(n=21)$, 
$\PSL_2(\bF_8)$ $(n=63)$) 
where $G={\rm Gal}(L/k)\leq S_n$ is transitive and $L/k$ is the Galois closure of $K/k$. 
We also refer to 
Browning and Newton \cite{BN16} and 
Frei, Loughran and Newton \cite{FLN18}. 


A group $G$ is called {\it metacyclic} if 
there exists a normal subgroup $N\lhd G$ such that 
$N\simeq C_m$ and $G/N\simeq C_n$ 
for some integers $m,n\geq 1$. 
The following is the main theorem of this paper 
(for $\Sha^2_\omega(G,J_{G/H})\leq H^2(G,J_{G/H})$, 
see 
Section \ref{S3}). 
Note that when $K/k$ is Galois, i.e. $L=K$, with Galois group 
$G={\rm Gal}(K/k)$, 
the Hasse norm principle holds for $K/k$ if $M(G)=0$ (resp. $G\simeq C_n$)
by Tate's theorem (Theorem \ref{thTate}) 
(resp. Hasse's original theorem \cite[Satz, page 64]{Has31}). 
\begin{theorem}[Hasse norm principle for metacyclic extensions with trivial Schur multiplier $M(G)=0$, see Theorem \ref{thmain} for the precise statement]\label{thmain0}
Let $k$ be a field, $K/k$ be a finite 
separable field extension and $L/k$ be the Galois closure of $K/k$ with 
Galois groups $G={\rm Gal}(L/k)$ and $H={\rm Gal}(L/K)\lneq G$. 
Let $M(G)=H^2(G,\bC^\times)\simeq H^3(G,\bZ)$ be the Schur multiplier of $G$. 
If $G$ is metacyclic, then  $H$ is cyclic and 
$\Sha^2_\omega(G,J_{G/H})\leq H^2(G,J_{G/H})\simeq M(G)$. 
In particular, when $k$ is a global field, if $G$ is metacyclic 
with trivial Schur multiplier $M(G)=0$, 
then the Hasse norm principle holds for $K/k$. 
\end{theorem}


We organize this paper as follows. 
In Section \ref{S3}, 
we recall some known results about the Hasse norm principle 
and related birational invariants. 
In Section \ref{S4}, we prove Theorem \ref{thmain0} (Theorem \ref{thmain}). 
In Section \ref{S5}, we give some examples of metacyclic groups $G$ 
with trivial Schur multiplier $M(G)=0$ 
as in Theorem \ref{thmain0} (Theorem \ref{thmain}). 
These provide new examples which 
the Hasse norm principle hold for non-Galois extensions $K/k$ 
whose Galois closure is $L/k$ with metacyclic $G={\rm Gal}(L/k)$ and $M(G)=0$. 
In Section \ref{S6}, by using GAP, we give examples of 
transitive groups $G\leq S_n$ $(2\leq n\leq 30)$ 
as in Theorem \ref{thmain0} (Theorem \ref{thmain}) and 
Proposition \ref{prop4.1}. 
In Section \ref{S7}, by using GAP, 
we also give some examples of Proposition \ref{prop4.1} 
which includes the case where $G$ is not metacyclic. 
In Section \ref{S8}, as an appendix of this paper, 
we also give some examples of finite simple groups $G$ 
with trivial Schur multiplier $M(G)=0$. 
These provide examples which 
the Hasse norm principle hold for Galois extensions $K/k$ 
with $G={\rm Gal}(K/k)$ by Tate's theorem (Theorem \ref{thTate}). 
In Section \ref{GAPcomp}, we give GAP computations of Schur multipliers 
$M(G)$ and $M(H)$ for 
Table $1$, Table $2$, Table $3$ and Table $4$ as in 
Theorem \ref{th6.1}, Theorem \ref{th6.2}, Theorem \ref{th7.1} and 
Theorem \ref{th7.2} respectively. 
In Section \ref{GAPcomp2}, GAP computations of Schur multipliers 
of $M(G)$ for Remark \ref{rem8.2} and Remark \ref{rem8.4} are given. 
The GAP algorithms and related ones can be available as {\tt HNP.gap} 
in \cite{Norm1ToriHNP}.
%

\section{Hasse norm principle and norm one tori}\label{S3}
Let $k$ be a global field
and $\overline{k}$ be a fixed separable closure of $k$. 
Let $T$ be an algebraic $k$-torus, 
i.e. a group $k$-scheme with fiber product (base change) 
$T\times_k \overline{k}=
T\times_{{\rm Spec}\, k}\,{\rm Spec}\, \overline{k}
\simeq (\bG_{m,\overline{k}})^n$; 
$k$-form of the split torus $(\bG_{m,k})^n$. 
%
Let $E$ be a principal homogeneous space (= torsor) under $T$.  
{\it Hasse principle holds for $E$} means that 
if $E$ has a $k_v$-rational point for all completions $k_v$ of $k$, 
then $E$ has a $k$-rational point. 
The set $H^1(k,T)$ classifies all such torsors $E$ up 
to (non-unique) isomorphism. 
We take {\it the Shafarevich-Tate group} of $T$: 
\begin{align*}
\Sha(T)={\rm Ker}\left\{H^1(k,T)\xrightarrow{\rm res} \bigoplus_{v\in V_k} 
H^1(k_v,T)\right\}
\end{align*}
where $V_k$ is the set of all places of $k$ and 
$k_v$ is the completion of $k$ at $v$. 
Then 
Hasse principle holds for all torsors $E$ under $T$ 
if and only if $\Sha(T)=0$.  

Let $K/k$ be a separable field extension with $[K:k]=n$ 
and $L/k$ be the Galois closure of $K/k$. 
Let $G={\rm Gal}(L/k)$ and $H={\rm Gal}(L/K)\leq G$ with $[G:H]=n$. 
Then we have $\bigcap_{\sigma\in G} H^\sigma=\{1\}$ 
where $H^\sigma=\sigma^{-1}H\sigma$ and hence 
$H$ contains no normal subgroup of $G$ except for $\{1\}$. 
We have the exact sequence 
$0\rightarrow \bZ\rightarrow \bZ[G/H]\rightarrow J_{G/H}\rightarrow 0$ 
and ${\rm rank}_{\bZ}\, J_{G/H}=n-1$. 
Write $J_{G/H}=\oplus_{1\leq i\leq n-1}\bZ u_i$. 
We define the action of $G$ on $L(x_1,\ldots,x_{n-1})$ by 
$x_i^{\sigma}=\prod_{j=1}^{n-1} x_j^{a_{i,j}} (1\leq i\leq n-1)$ 
for any $\sigma\in G$, 
when $u_i^{\sigma}=\sum_{j=1}^{n-1} a_{i,j} u_j$ $(a_{i,j}\in\bZ)$. 

Let $T=R^{(1)}_{K/k}(\bG_{m,K})$ be the norm one torus of $K/k$,
i.e. the kernel of the norm map $R_{K/k}(\bG_{m,K})\rightarrow \bG_{m,k}$ 
where 
$R_{K/k}$ is the Weil restriction 
(see Ono \cite[Section 1.4]{Ono61}, Voskresenskii \cite[page 37, Section 3.12]{Vos98}). 
It is biregularly isomorphic to the norm hypersurface 
$f(x_1,\ldots,x_n)=1$ where 
$f\in k[x_1,\ldots,x_n]$ is the polynomial of total 
degree $n$ defined by the norm map $N_{K/k}:K^\times\to k^\times$ 
and has the 
Chevalley module $\widehat{T}\simeq J_{G/H}$ as its character module 
where $J_{G/H}=(I_{G/H})^\circ={\rm Hom}_\bZ(I_{G/H},\bZ)$ 
is the dual lattice of $I_{G/H}={\rm Ker}\, \varepsilon$ and 
$\varepsilon : \bZ[G/H]\rightarrow \bZ$, 
$\sum_{\overline{g}\in G/H} a_{\overline{g}}\,\overline{g}\mapsto \sum_{\overline{g}\in G/H} a_{\overline{g}}$ 
is the augmentation map, 
i.e. the function field $k(T)$ of $T$ is isomorphic to 
$L(x_1,\ldots,x_{n-1})^G$ 
(see Endo and Miyata \cite[Section 1]{EM73}, \cite[Section 1]{EM75} 
and Voskresenskii \cite[Section 4.8]{Vos98}). 

Ono \cite{Ono63} established the relationship 
between the Hasse norm principle for $K/k$ 
and Hasse principle for all torsors $E$ under
the norm one torus $R^{(1)}_{K/k}(\bG_{m,K})$: 
\begin{theorem}[{Ono \cite[page 70]{Ono63}, see also Platonov \cite[page 44]{Pla82}, Kunyavskii \cite[Remark 3]{Kun84}, Platonov and Rapinchuk \cite[page 307]{PR94}}]\label{thOno}
Let $k$ be a global field and $K/k$ be a finite separable field extension. 
Let $\bA_K^\times$ be the idele group of $K$. 
Let $T=R^{(1)}_{K/k}(\bG_{m,K})$ be the norm one torus of $K/k$. 
Then 
\begin{align*}
\Sha(T)\simeq (N_{K/k}(\bA_K^\times)\cap k^\times)/N_{K/k}(K^\times)
\end{align*}
where  $N_{K/k}$ is the norm map. 
In particular, $\Sha(T)=0$ if and only if 
the Hasse norm principle holds for $K/k$. 
\end{theorem}

Let $T$ be an algebraic $k$-torus 
and $T(k)$ be the group of $k$-rational points of $T$. 
Then $T(k)$ 
embeds into $\prod_{v\in V_k} T(k_v)$ by the diagonal map 
where 
$V_k$ is the set of all places of $k$ and 
$k_v$ is the completion of $k$ at $v$. 
Let $\overline{T(k)}$ be the closure of $T(k)$  
in the product $\prod_{v\in V_k} T(k_v)$. 
The group 
\begin{align*}
A(T)=\left(\prod_{v\in V_k} T(k_v)\right)/\overline{T(k)}
\end{align*}
is called {\it the kernel of the weak approximation} of $T$. 
We say that {\it $T$ has the weak approximation property} if $A(T)=0$. 

%
\begin{theorem}[{Voskresenskii \cite[Theorem 5, page 1213]{Vos69}, 
\cite[Theorem 6, page 9]{Vos70}, see also \cite[Section 11.6, Theorem, page 120]{Vos98}}]\label{thV}
Let $k$ be a global field, 
$T$ be an algebraic $k$-torus and $X$ be a smooth $k$-compactification of $T$. 
Then there exists an exact sequence
\begin{align*}
0\to A(T)\to H^1(k,{\rm Pic}\,\overline{X})^{\vee}\to \Sha(T)\to 0
\end{align*}
where $M^{\vee}={\rm Hom}(M,\bQ/\bZ)$ is the Pontryagin dual of $M$. 
In particular, if $T$ is retract $k$-rational, then $ H^1(k,{\rm Pic}\,\overline{X})=0$ and hence 
$A(T)=0$ and $\Sha(T)=0$. 
Moreover, if $L$ is the splitting field of $T$ and $L/k$ 
is an unramified extension, then $A(T)=0$ and 
$H^1(k,{\rm Pic}\,\overline{X})^{\vee}\simeq \Sha(T)$. 
\end{theorem}
For the last assertion, see Voskresenskii \cite[Section 11.5]{Vos98}. 
It follows that 
$H^1(k,{\rm Pic}\,\overline{X})=0$ if and only if $A(T)=0$ and $\Sha(T)=0$, 
i.e. $T$ has the weak approximation property and 
Hasse principle holds for all torsors $E$ under $T$. 
If $T$ is (stably/retract) $k$-rational, 
then $H^1(k,{\rm Pic}\,\overline{X})=0$ 
(see Voskresenskii \cite[Theorem 5, page 1213]{Vos69},
Manin \cite[Section 30]{Man74}, 
Manin and Tsfasman \cite{MT86} 
and also Hoshi, Kanai and Yamasaki \cite[Section 1]{HKY22}). 
Theorem \ref{thV} was generalized 
to the case of linear algebraic groups by Sansuc \cite{San81}.



Applying Theorem \ref{thV} to $T=R^{(1)}_{K/k}(\bG_{m,K})$,  
it follows from Theorem \ref{thOno} that 
$H^1(k,{\rm Pic}\,\overline{X})=0$ if and only if 
$A(T)=0$ and $\Sha(T)=0$, 
i.e. 
$T$ has the weak approximation property and 
the Hasse norm principle holds for $K/k$. 
In the algebraic language, 
the latter condition $\Sha(T)=0$ means that 
for the corresponding norm hypersurface $f(x_1,\ldots,x_n)=b$, 
it has a $k$-rational point 
if and only if it has a $k_v$-rational point 
for any place $v$ of $k$ where 
$f\in k[x_1,\ldots,x_n]$ is the polynomial of total 
degree $n$ defined by the norm map $N_{K/k}:K^\times\to k^\times$ 
and $b\in k^\times$ 
(see Voskresenskii \cite[Example 4, page 122]{Vos98}).


When $K/k$ is a finite Galois extension, 
we have that: 

\begin{theorem}[{Voskresenskii \cite[Theorem 7]{Vos70}, Colliot-Th\'{e}l\`{e}ne and Sansuc \cite[Proposition 1]{CTS77}}]\label{thVos70}
Let $k$ be a field and 
$K/k$ be a finite Galois extension with Galois group $G={\rm Gal}(K/k)$. 
Let $T=R^{(1)}_{K/k}(\bG_{m,K})$ be the norm one torus of $K/k$ 
and $X$ be a smooth $k$-compactification of $T$. 
Then 
$H^1(H,{\rm Pic}\, X_K)\simeq H^2(H,\widehat{T})\simeq H^3(H,\bZ)$ 
for any subgroup $H$ of $G$. 
In particular, 
$H^1(k,{\rm Pic}\, \overline{X})\simeq
H^1(G,{\rm Pic}\, X_K)\simeq H^2(G,\widehat{T})\simeq H^2(G,J_G)
\simeq H^3(G,\bZ)\simeq M(G)$ where $M(G)$ is the Schur multiplier of $G$. 
\end{theorem}
In other words, for the $G$-lattice $J_G\simeq \widehat{T}$, 
$H^1(H,[J_G]^{fl})\simeq H^3(H,\bZ)$ for any subgroup $H$ of $G$ 
and $H^1(G,[J_G]^{fl})\simeq H^3(G,\bZ)\simeq H^2(G,\bQ/\bZ)\simeq M(G)$. 
By the exact sequence $0\to\bZ\to\bZ[G]\to J_G\to 0$, 
we also have $\delta:H^1(G,J_G)\simeq H^2(G,\bZ)\simeq H^1(G,\bQ/\bZ)\simeq G^{ab}$ 
where $\delta$ is the connecting homomorphism and 
$G^{ab}:=G/[G,G]$ is the abelianization of $G$ 
where $[G,G]$ is the commutator subgroup of $G$. 


By Poitou-Tate duality (see Milne \cite[Theorem 4.20]{Mil86}, 
Platonov and Rapinchuk \cite[Theorem 6.10]{PR94}, 
Neukirch, Schmidt and Wingberg \cite[Theorem 8.6.8, page 422]{NSW00}, 
Harari \cite[Theorem 17.13]{Har20}), 
we also have 
\begin{align*}
\Sha(T)^\vee\simeq\Sha^2(G,\widehat{T})
\end{align*}
where $\Sha(T)^\vee={\rm Hom}(\Sha(T),\bQ/\bZ)$ and 
\begin{align*}
\Sha^i(G,\widehat{T})={\rm Ker}\left\{H^i(G,\widehat{T})\xrightarrow{\rm res} \bigoplus_{v\in V_k} 
H^i(G_v,\widehat{T})\right\}\quad (i\geq 1)
\end{align*}
is {\it the $i$-th Shafarevich-Tate group} 
of $\widehat{T}={\rm Hom}(T,\bG_{m,K})$, 
$G={\rm Gal}(K/k)$ and $K$ is the minimal splitting field of $T$. 
Note that $\Sha(T)\simeq \Sha^1(G,T)\simeq \Sha^2(G,\widehat{T})$. 
In the special case where 
$T=R^{(1)}_{K/k}(\bG_{m,K})$ and $K/k$ is Galois with $G={\rm Gal}(K/k)$, 
we have 
$H^2(G,\widehat{T})=H^2(G,J_{G})\simeq H^3(G,\bZ)$ and hence 
we get Tate's theorem (Theorem \ref{thTate}) 
via Ono's theorem (Theorem \ref{thOno}). 


Let 
$M$ be a $G$-lattice. We define 
\begin{align*}
\Sha^i_\omega(G,M):={\rm Ker}\left\{H^i(G,M)\xrightarrow{{\rm res}}\bigoplus_{H\leq G:{\rm\, cyclic}}
H^i(H,M)\right\}\quad (i\geq 1) .
\end{align*}
Note that ``$\Sha^i_\omega$'' corresponds to 
the unramified part of ``$\Sha^i$'' because 
if $v\in V_k$ is unramified, then $G_v\simeq C_n$ and 
all the cyclic subgroups of $G$ appear as $G_v$ 
from the Chebotarev density theorem. 
%
\begin{theorem}[{Colliot-Th\'{e}l\`{e}ne and Sansuc \cite[Proposition 9.5 (ii)]{CTS87}, see also \cite[Proposition 9.8]{San81}, \cite[page 98]{Vos98}, \cite[Corollaire 1]{CTHS05}, \cite[Theorem 2.3]{BP20}}]\label{thCTS87}
Let $k$ be a field with ${\rm char}\, k=0$
and $K/k$ be a finite Galois extension with Galois group $G={\rm Gal}(K/k)$. 
Let $T$ be an algebraic $k$-torus which splits over $K$ and 
$X$ be a smooth $k$-compactification of $T$. 
Then we have 
\begin{align*}
\Sha^2_\omega(G,\widehat{T})\simeq 
H^1(G,{\rm Pic}\, X_K)\simeq {\rm Br}(X)/{\rm Br}(k)
\end{align*}
where 
${\rm Br}(X)$ is the \'etale cohomological Brauer Group of $X$ 
$($it is the same as the Azumaya-Brauer group of $X$ 
for such $X$, see \cite[page 199]{CTS87}$)$. 
\end{theorem}

In other words, 
we have 
$H^1(k,{\rm Pic}\, \overline{X})\simeq H^1(G,{\rm Pic}\, X_K)\simeq 
H^1(G,[\widehat{T}]^{fl})\simeq \Sha^2_\omega(G,\widehat{T})\simeq {\rm Br}(X)/{\rm Br}(k)$. 
We also see  
${\rm Br}_{\rm nr}(k(X)/k)={\rm Br}(X)\subset {\rm Br}(k(X))$ 
(see Saltman \cite[Proposition 10.5]{Sal99}, 
Colliot-Th\'{e}l\`{e}ne \cite[Theorem 5.11]{CTS07}, 
Colliot-Th\'{e}l\`{e}ne and Skorobogatov 
\cite[Proposition 6.2.7]{CTS21}).
Moreover, by taking the duals of Voskresenskii's exact sequence as in Theorem \ref{thV}, 
we get the following exact sequence
\begin{align*}
0\to \Sha^2(G,\widehat{T})\to \Sha^2_\omega(G,\widehat{T})\to A(T)^\vee\to 0
\end{align*}
where the map $\Sha^2(G,\widehat{T})\to \Sha^2_\omega(G,\widehat{T})$ 
is the natural inclusion arising from the Chebotarev density theorem 
(see also Macedo and Newton \cite[Proposition 2.4]{MN22}). 

For norm one tori $T=R^{(1)}_{K/k}(\bG_{m,K})$, 
we also obtain the group $T(k)/R$ of $R$-equivalence classes 
over a local field $k$ via 
$T(k)/R\simeq H^1(k,{\rm Pic}\,\overline{X})\simeq 
H^1(G,[\widehat{T}]^{fl})$ 
(see Colliot-Th\'{e}l\`{e}ne and Sansuc \cite[Corollary 5, page 201]{CTS77}, 
Voskresenskii \cite[Section 17.2]{Vos98} and Hoshi, Kanai and Yamasaki \cite[Section 7, Application 1]{HKY22}). 
%

For norm one tori $T=R^{(1)}_{K/k}(\bG_{m,K})$, 
recall that 
the function field $k(T)\simeq L(J_{G/H})^G$ 
for the character module $\widehat{T}={\rm Hom}(T,\bG_{m,L})\simeq J_{G/H}$ and hence we have: 
\begin{align*}
&\hspace*{4.3mm}[J_{G/H}]^{fl}=0\hspace*{10mm}
\ \ \Rightarrow\ \ [J_{G/H}]^{fl}\ \textrm{is\ invertible}\, 
\ \ \Rightarrow\ \  H^1(G,[J_{G/H}]^{fl})=0\\
&\hspace*{14mm}\Updownarrow\hspace*{41mm} \Updownarrow\hspace*{41mm}
\Downarrow\\
&\hspace*{1mm}T\ \textrm{is\ stably}\ k\textrm{-rational}\ \ \hspace*{0mm}
\Rightarrow\ \ \hspace*{0mm}T\ \textrm{is\ retract}\ k\textrm{-rational}\ 
\Rightarrow\ \  A(T)=0\ \textrm{and}\ \Sha(T)=0\
\end{align*}
where the last implication holds over a global field $k$ 
(see Theorem \ref{thV}, 
see also Colliot-Th\'{e}l\`{e}ne and Sansuc \cite[page 29]{CTS77}). 
The last conditions mean that 
$T$ has the weak approximation property and 
the Hasse norm principle holds for $K/k$ as above.  
In particular, it follows that 
$[J_{G/H}]^{fl}$ is invertible, i.e. $T$ is retract $k$-rational, 
and hence $A(T)=0$ and $\Sha(T)=0$ when $G=pTm\leq S_p$ is a transitive 
subgroup of $S_p$ of prime degree $p$ 
and $H=G\cap S_{p-1}\leq G$ with $[G:H]=p$ (see 
Colliot-Th\'{e}l\`{e}ne and Sansuc \cite[Proposition 9.1]{CTS87} and 
\cite[Lemma 2.17]{HY17}). 
Hence the Hasse norm principle holds for $K/k$ when $[K:k]=p$. 

Recall that, by Tate's theorem (Theorem \ref{thTate}), 
the Hasse norm principle holds for $K/k$ 
if and only if the restriction map 
$H^3(G,\bZ)\xrightarrow{\rm res}\bigoplus_{v\in V_k}H^3(G_v,\bZ)$ 
is injective. 
This is also equivalent to $\Sha(T)=0$ 
by Ono's theorem (Theorem \ref{thOno}) 
via $\widehat{T}\simeq J_{G}$ and 
$\Sha(T)^\vee\simeq \Sha^1(G,T)^\vee\simeq\Sha^2(T,\widehat{T})$ 
and $H^2(G,\widehat{T})=H^2(G,J_{G})\simeq H^3(G,\bZ)$ 
where $T=R^{(1)}_{K/k}(\bG_{m,K})$. 
Note also that 
$\Sha(T)=0$ also follows 
from the retract $k$-rationality of $T=R^{(1)}_{K/k}(\bG_{m,K})$ 
when all the Sylow subgroups of $G$ is cyclic 
due to Endo and Miyata \cite[Theorem 2.3]{EM75}. 
For the rationality problem for norm one tori $T=R^{(1)}_{K/k}$, see e.g. 
\cite{EM75}, \cite{CTS77}, \cite{Hur84}, \cite{CTS87}, 
\cite{LeB95}, \cite{CK00}, \cite{LL00}, \cite{Flo}, \cite{End11}, 
\cite{HY17}, \cite{HHY20}, \cite{HY21}, \cite{HY24}, \cite{HY1}, \cite{HY2}. 


\section{Proof of Theorem \ref{thmain0} (Theorem \ref{thmain})}\label{S4}

Let $k$ be a field, 
$K/k$ be a finite separable field extension 
and $L/k$ be the Galois closure of $K/k$ with Galois groups 
$G={\rm Gal}(L/k)$ and $H={\rm Gal}(L/K)\lneq G$.  
Then we have $\bigcap_{\sigma\in G} H^\sigma=\{1\}$ 
where $H^\sigma=\sigma^{-1}H\sigma$ and hence 
$H$ contains no normal subgroup of $G$ except for $\{1\}$. 

Let $Z(G)$ be the center of $G$, 
$[a,b]:=a^{-1}b^{-1}ab$ be the commutator of $a,b\in G$, 
$[G,G]:=\langle [a,b]\mid a,b\in G\rangle$ be the commutator subgroup of $G$ 
and $G^{ab}:=G/[G,G]$ be the abelianization of $G$, i.e. the maximal abelian quotient of $G$. 
Let $M(G)=H^2(G,\bC^\times)\simeq H^2(G,\bQ/\bZ)\xrightarrow[\sim]{\delta} 
H^3(G,\bZ)$ be the Schur multiplier of $G$ 
where $\delta$ is the connecting homomorphism 
(see e.g. Neukirch, Schmidt and Wingberg \cite[Chapter I, \S 3, page 26]{NSW00}). 

\begin{proposition}\label{prop4.1}
Let $k$ be a field, $K/k$ be a finite 
separable field extension and $L/k$ be the Galois closure of $K/k$ with 
Galois groups $G={\rm Gal}(L/k)$ and $H={\rm Gal}(L/K)\lneq G$. 
Then we have an exact sequence 
\begin{align*}
H^1(G,\bQ/\bZ)\simeq G^{ab}
&\xrightarrow{\rm res} 
H^1(H,\bQ/\bZ)\simeq H^{ab}\to H^2(G,J_{G/H})\\
&\xrightarrow{~\delta~} 
H^3(G,\bZ)\simeq M(G)\xrightarrow{\rm res} H^3(H,\bZ)\simeq M(H)
\end{align*}
where $\delta$ is the connecting homomorphism.\\ 
{\rm (0)} 
$H^1(G,\bQ/\bZ)\simeq G^{ab}
\xrightarrow{\rm res}H^1(H,\bQ/\bZ)\simeq H^{ab}$ 
is surjective if and only if $[G,G]\cap H=[H,H]$;\\
{\rm (1)} 
If $[G,G]\cap H=[H,H]$ and $M(G)=0$, 
then $\Sha^2_\omega(G,J_{G/H})\leq H^2(G,J_{G/H})$ $\xrightarrow[\sim]{\delta}$ $H^3(G,\bZ)\simeq M(G)=0$;\\ 
{\rm (2)} 
If $[G,G]\cap H=[H,H]$ and $M(H)=0$, 
then $\Sha^2_\omega(G,J_{G/H})\leq H^2(G,J_{G/H})$ $\xrightarrow[\sim]{\delta}$ $H^3(G,\bZ)\simeq M(G)$;\\
{\rm (3)} If there exists $H^\prime$ $\lhd$ $G$ such that 
$G/H^\prime$ is abelian and $H^\prime$ $\cap$ $H=\{1\}$, then 
$[G,G]\cap H=[H,H]$. 
In particular, 
if $M(G)=0$ $($resp. $M(H)=0$$)$, then 
$\Sha^2_\omega(G,J_{G/H})\leq H^2(G,J_{G/H})$ $\xrightarrow[\sim]{\delta}$ $H^3(G,\bZ)\simeq M(G)=0$  
$($resp.  $\Sha^2_\omega(G,J_{G/H})\leq H^2(G,J_{G/H})$ $\xrightarrow[\sim]{\delta}$ $H^3(G,\bZ)\simeq M(G)$$)$. 

When $k$ is a global field, $\Sha^2_\omega(G,J_{G/H})=0$ implies that 
$A(T)=0$, i.e. 
$T$ has the weak approximation property, 
and $\Sha(T)=0$, i.e. the Hasse norm principle holds for $K/k$ 
$($that is, Hasse principle holds for all torsors $E$ under $T$$)$ 
where $T=R^{(1)}_{K/k}(\bG_{m,K})$ is the norm one torus of $K/k$ 
$($see Section \ref{S3} and Ono's theorem $($Theorem \ref{thOno}$))$. 
\end{proposition}
\begin{proof}
By the definition, we have $0\to \bZ\to \bZ[G/H]\to J_{G/H}\to 0$ where 
$J_{G/H}\simeq \widehat{T}={\rm Hom}(T,\bG_{m,L})$ 
and $T=R^{(1)}_{K/k}(\bG_{m,K})$ is the norm one torus of $K/k$. 
Then we get 
\begin{align*}
H^2(G,\bZ)\to H^2(G,\bZ[G/H])\to H^2(G,J_{G/H})
\xrightarrow{\delta} H^3(G,\bZ)\to H^3(G,\bZ[G/H]). 
\end{align*} 
Then $H^2(G,\bZ)\simeq H^1(G,\bQ/\bZ)
={\rm Hom}(G,\bQ/\bZ)\simeq G^{ab}=G/[G,G]$. 
We have 
$H^2(G,\bZ[G/H])\simeq H^2(H,\bZ)\simeq H^{ab}$ 
and $H^3(G,\bZ[G/H])\simeq H^3(H,\bZ)\simeq M(H)$ by Shapiro's lemma 
(see e.g. Brown \cite[Proposition 6.2, page 73]{Bro82}, Neukirch, Schmidt and Wingberg \cite[Proposition 1.6.3, page 59]{NSW00}). 
This implies that 
\begin{align*}
H^1(G,\bQ/\bZ)\simeq G^{ab}
&\xrightarrow{\rm res} 
H^1(H,\bQ/\bZ)\simeq H^{ab}\to H^2(G,J_{G/H})\\
&\xrightarrow{~\delta~} 
H^3(G,\bZ)\simeq M(G)\xrightarrow{\rm res} H^3(H,\bZ)\simeq M(H). 
\end{align*}
(0) We see that ${\rm Image}\{H^1(G,\bQ/\bZ)\simeq G^{ab}=G/[G,G]\xrightarrow{\rm res} 
H^1(H,\bQ/\bZ)\simeq H^{ab}=H/[H,H]\}=H/([G,G]\cap H)$. 
Hence 
$H^1(G,\bQ/\bZ)\simeq G^{ab}
\xrightarrow{\rm res}H^1(H,\bQ/\bZ)\simeq H^{ab}$ 
is surjective if and only if $[G,G]\cap H=[H,H]$.\\
(1) If $[G,G]\cap H=[H,H]$, then by (0) we have 
$H^2(G,\bZ)\simeq G^{ab}\xrightarrow{\rm res}H^2(H,\bZ)\simeq H^{ab}$ 
is surjective. 
This implies that $H^2(G,J_{G/H})\xrightarrow{\delta}H^3(G,\bZ)\simeq M(G)=0$ becomes isomorphic.\\ 
(2) 
If $[G,G]\cap H=[H,H]$, then by (0) we have 
$H^2(G,\bZ)\simeq G^{ab}\xrightarrow{\rm res}H^2(H,\bZ)\simeq H^{ab}$ 
is surjective. 
Hence if we also have $M(H)=0$, 
then $H^2(G,J_{G/H})\xrightarrow{\delta}H^3(G,\bZ)\simeq M(G)$ 
becomes isomorphic.\\
(3) Because $G/H^\prime$ is abelian, we have $[G,G]\leq H^\prime$. 
It follows from $H^\prime\cap H=\{1\}$ that $[G,G]\cap H=\{1\}$. 
Hence if we take $h\in H$, then 
$\widetilde{h}\in H/(H\cap [G,G])\simeq H[G,G]/[G,G]
\leq G^{ab}=G/[G,G]$ maps to 
$\overline{h}\in H^{ab}=H/[H,H]$. 
This implies that 
$G^{ab}\xrightarrow{\rm res}H^{ab}$ is surjective. 
The last assertion of (3) follows from (1), (2). 

Because $\Sha^2_\omega(G,J_{G/H})\leq H^2(G,J_{G/H})$, 
when $k$ is a global field, 
by Theorem \ref{thV} and Theorem \ref{thCTS87}, we have 
\begin{align*}
H^2(G,J_{G/H})=0\ \Rightarrow\ \Sha^2_\omega(G,J_{G/H})=0\ \Rightarrow\ A(T)=0\ {\rm and}\ \Sha(T)=0
\end{align*}
where $T=R^{(1)}_{K/k}(\bG_{m,K})$ is the norm one torus of $K/k$ 
with $\widehat{T}\simeq J_{G/H}$. 
In particular, it follows from Ono's theorem (Theorem \ref{thOno}) that 
$\Sha(T)=0$ if and only if the Hasse norm principle holds for $K/k$. 
\end{proof}

By Proposition \ref{prop4.1} (3), we get the main theorem 
which is the precise statement of Thereom \ref{thmain0}: 
%
\begin{theorem}[{Hasse norm principle for metacyclic extensions with trivial Schur multiplier $M(G)=0$: the precise statement of Theorem \ref{thmain0}}]\label{thmain}
Let $k$ be a field, $K/k$ be a finite 
separable field extension and $L/k$ be the Galois closure of $K/k$ with 
Galois groups $G={\rm Gal}(L/k)$ and $H={\rm Gal}(L/K)\lneq G$. 
Let $M(G)=H^2(G,\bC^\times)\simeq H^3(G,\bZ)$ be the Schur multiplier of $G$. 
Assume that $G$ is metacyclic with $N\lhd G$, 
$N\simeq C_m$ and $G/N\simeq C_n$. 
Then $H^2(G,\bZ)\simeq G^{ab}\xrightarrow{\rm res}H^2(H,\bZ)\simeq H^{ab}$ is surjective 
and $H\leq C_n$ is cyclic $($this implies $M(H)=0$$)$, 
and hence $\Sha^2_\omega(G,J_{G/H})\leq 
H^2(G,J_{G/H})\xrightarrow[\sim]{\delta} H^3(G,\bZ)\simeq M(G)$. 
In particular, if $M(G)=0$, then $\Sha^2_\omega(G,J_{G/H})=H^2(G,J_{G/H})=0$. 
When $k$ is a global field, $\Sha^2_\omega(G,J_{G/H})=0$ implies that 
$A(T)=0$, i.e. 
$T$ has the weak approximation property, 
and $\Sha(T)=0$, i.e. the Hasse norm principle holds for $K/k$ 
$($that is, Hasse principle holds for all torsors $E$ under $T$$)$ 
where $T=R^{(1)}_{K/k}(\bG_{m,K})$ is the norm one torus of $K/k$ 
$($see Section \ref{S3} and Ono's theorem $($Theorem \ref{thOno}$))$. 
\end{theorem}
\begin{proof}
We should show that 
there exists $H^\prime$ $\lhd$ $G$ such that 
$G/H^\prime$ is abelian and $H^\prime$ $\cap$ $H=\{1\}$ 
and $H\leq C_n$ is cyclic (this implies $M(H)=0)$ 
because the other statements follow from Proposition \ref{prop4.1} (3). 
By the definition, 
we have 
\begin{align*}
0\to H^\prime\simeq C_m\xrightarrow{i} G\xrightarrow{\pi} G/H^\prime\simeq C_n\to 0.
\end{align*}
We see that $H^\prime\cap H\lhd G$ because 
$H^\prime\cap H\leq H^\prime$ is a characteristic subgroup 
and $H^\prime\lhd G$. 
By the condition $\bigcap_{\sigma\in G}H^\sigma=\{1\}$, 
$H$ contains no normal subgroup of $G$ except for $\{1\}$. 
This implies that $H^\prime$ $\cap$ $H=\{1\}$. 
We also find that the map 
$\pi|_H$ (the map $\pi$ restricted to $H$) is injective 
because $H^\prime \cap H=\{1\}$. 
Hence $H\leq C_n$ is cyclic because $H\simeq \pi(H)\leq G/H^\prime\simeq C_n$. 
\end{proof}

As a consequence of Proposition \ref{prop4.1} and 
Theorem \ref{thmain0} (Theorem \ref{thmain}), 
we get the 
Tamagawa number $\tau(T)$ of the norm one tori 
$T=R^{(1)}_{K/k}(\bG_{m,K})$ of $K/k$ 
over a global field $k$ 
via Ono's formula $\tau(T)=|H^1(k,\widehat{T})|/|\Sha(T)|
=|H^1(G,J_{G/H})|/|\Sha(T)|$ 
where $J_{G/H}\simeq\widehat{T}={\rm Hom}(T,\bG_{m,L})$ 
(see Ono \cite[Main theorem, page 68]{Ono63}, \cite{Ono65}, 
Voskresenskii \cite[Theorem 2, page 146]{Vos98} and 
Hoshi, Kanai and Yamasaki \cite[Section 8, Application 2]{HKY22}). 
\begin{corollary}
Let $K/k$, $L/k$, $G$ and $H$ be the same as in Proposition \ref{prop4.1}. 
Let $T=R^{(1)}_{K/k}(\bG_{m,K})$ be the norm one torus of $K/k$.  
When $k$ is a global field, 
if $H^1(G,\bQ/\bZ)\simeq G^{ab}\xrightarrow{\rm res}H^1(H,\bQ/\bZ)\simeq 
H^{ab}$ is surjective 
$($that is, $[G,G]\cap H=[H,H]$ by Proposition \ref{prop4.1} (0)$)$, 
and $M(G)=0$, 
then the Tamagawa number $\tau(T)=|H^1(G,J_{G/H})|=|G^{ab}|/|H^{ab}|$. 
In particular, if $G$ is metacyclic with trivial Schur multiplier $M(G)=0$, 
then $H$ is cyclic and $\tau(T)=|H^1(G,J_{G/H})|=|G^{ab}|/|H|$.
\end{corollary}
\begin{proof}
By the definition, we have $0\to \bZ\xrightarrow{\varepsilon^\circ}\bZ[G/H]\to J_{G/H}\to 0$ where 
$J_{G/H}\simeq \widehat{T}={\rm Hom}(T,\bG_{m,L})$. 
Then we get $H^1(G,\bZ[G/H])\to H^1(G,J_{G/H})$ 
$\xrightarrow{\delta}$ $H^2(G,\bZ)\to H^2(G,\bZ[G/H])\to H^2(G,J_{G/H})$ 
where $\delta$ is the connecting homomorphism. 
We have $H^2(G,\bZ)\simeq H^1(G,\bQ/\bZ)={\rm Hom}(G,\bQ/\bZ)\simeq G^{ab}$. 
By Shapiro's lemma 
(see e.g. Brown \cite[Proposition 6.2, page 73]{Bro82}, Neukirch, Schmidt and Wingberg \cite[Proposition 1.6.3, page 59]{NSW00}), 
we also have  
$H^1(G,\bZ[G/H])\simeq H^1(H,\bZ)={\rm Hom}(H,\bZ)=0$
and 
$H^2(G,\bZ[G/H])\simeq H^2(H,\bZ)\simeq H^1(H,\bQ/\bZ)={\rm Hom}(H,\bQ/\bZ)=H^{ab}$.
By Proposition \ref{prop4.1}, 
if $H^1(G,\bQ/\bZ)\simeq G^{ab}\xrightarrow{\rm res}H^1(H,\bQ/\bZ)\simeq H^{ab}$ is surjective (that is, $[G,G]\cap H=[H,H]$) and $M(G)=0$, 
then we get $H^2(G,J_{G/H})=0$ and hence $\Sha(T)=0$. 
Hence we have 
$0\to H^1(G,J_{G/H})\xrightarrow{\delta} G^{ab}\to H^{ab}\to 0$. 
Applying Ono's formula $\tau(T)=|H^1(k,\widehat{T})|/|\Sha(T)|$, 
we get $\tau(T)=|H^1(G,J_{G/H})|=|G^{ab}|/|H^{ab}|$. 
The last assertion follows from Theorem \ref{thmain}. 
\end{proof}
%


\section{Examples of Theorem \ref{thmain0} (Theorem \ref{thmain}): Metacyclic groups $G$ with $M(G)=0$}\label{S5}

We will give some examples of metacyclic groups $G$ with trivial Schur multiplier $M(G)=0$. 
By Theorem \ref{thmain0} (Theorem \ref{thmain}), 
these provide new examples which 
the Hasse norm principle hold for non-Galois extensions $K/k$ 
whose Galois closure is $L/k$ with metacyclic $G={\rm Gal}(L/k)$ and $M(G)=0$. 
We first note that if $G\simeq C_n$, then $M(C_n)\simeq H^3(C_n,\bZ)\simeq H^1(C_n,\bZ)={\rm Hom}(C_n,\bZ)=0$.
\begin{lemma}\label{lemGp}
Let $G$ be a finite group and $G_p={\rm Syl}_p(G)$ be a $p$-Sylow subgroup of $G$. 
Then $M(G)\xrightarrow{\oplus {\rm res}_p} \oplus_{p\mid |G|} M(G_p)$ becomes injective. 
In particular, if $M(G_p)=0$ for any $p\mid |G|$, then $M(G)=0$.  
\end{lemma}
\begin{proof}
We see that the composite map 
${\rm cores}_p\circ {\rm res}_p: M(G)\xrightarrow{{\rm res}_p} M(G_p)\xrightarrow{{\rm cores}_p} M(G)$ 
is the multiplication by $[G:G_p]$ which is coprime to $p$ 
(see e.g. Serre \cite[Proposition 4, page 130]{Ser79}, Neukirch, Schmidt and Wingberg \cite[Corollary 1.5.7]{NSW00}). 
If $\alpha \in {\rm Ker}\{M(G)\xrightarrow{\oplus {\rm res}_p} \oplus_{p\mid |G|} M(G_p)\}$, 
then $[G:G_p]\alpha=0$ for all $p\mid |G|$. 
This implies that $\alpha=0\in M(G)$. 
We get the injection $M(G)\hookrightarrow \oplus_{p\mid |G|} M(G_p)$. 
\end{proof}
For example, if $G$ is nilpotent, i.e. 
$G\simeq \prod_{i=1}^r G_{p_i}$ where $G_{p_i}=Syl_{p_i}(G)$ $(1\leq i\leq r)$ are 
all the Sylow subgroups of $G$, then 
$M(G)\simeq \bigoplus_{i=1}^r M(G_{p_i})$ 
(see Karpilovsky \cite[Corollary 2.2.11]{Kar87}). 
In particular, if $G\simeq \prod_{i=1}^r G_{p_i}$ with $M(G_{p_i})=0$ for any $1\leq i\leq r$, then $M(G)=0$.\\

\subsection{Metacyclic groups $G$ with $M(G)=0$}

\begin{example}[Metacyclic groups $G$ with $M(G)=0$]\label{exmc}
A group $G$ is called {\it metacyclic} if 
there exists a normal subgroup $N\lhd G$ 
such that $N\simeq C_m$ and $G/N\simeq C_n$. 
In particular, if $\gcd(m,n)=1$, then $G\simeq C_m\rtimes C_n$ and 
$M(G)\simeq M(C_m)^{C_n}\oplus M(C_n)=0$ where $\gcd(m,n)$ is the greatest common divisor of integers 
$m$ and $n$ (see Karpilovsky \cite[Corollary 2.2.6, page 35]{Kar87}). 
%
%

If $G$ is metacyclic, then $G$ can be presented as 
\begin{align*}
G=G_0(m,s,r,t)=\langle a,b\mid a^m=1, b^s=a^t, bab^{-1}=a^r\rangle
\end{align*}
where $m,s,r,t\geq 1$, $r^s\equiv 1\ ({\rm mod}\ m)$, 
$m\mid t(r-1)$, $t\mid m$ and we have 
\begin{align*}
M(G)\simeq \bZ/u\bZ
\end{align*}
where 
\begin{align*}
u=\gcd(m,r-1)\gcd(1+r+\cdots+r^{s-1},t)/m
\end{align*}
(see Karpilovsky \cite[Section 2.11, page 96, Theorem 2.11.3, page 98]{Kar87}, 
\cite[Chapter 10, C, page 288, Theorem 1.25, page 289]{Kar93}). 
Conversely, such $G$ becomes a metacyclic group of order $ms$. 
In particular, if $t=l:=m/\gcd(m,r-1)$ and $s=n$, then 
\begin{align*}
G=G_0(m,n,r,l)=\langle a,b\mid a^m=1, b^n=a^l, bab^{-1}=a^r\rangle
\end{align*}
satisfies $M(G)=0$ (see Beyl and Tappe \cite[Theorem 2.9, page 199]{BT82}, 
$G=G(m,n,r,1)$ where $G(m,n,r,\lambda)$ is given as in \cite[Definition 2.3, page 196]{BT82}). 
\end{example}

\begin{theorem}
Let $K/k$, $L/k$, $G$ and $H$ be the same as in Theorem \ref{thmain}. 
If $G$ is metacyclic, then 
$\Sha^2_\omega(G,J_{G/H})\leq H^2(G,J_{G/H})\simeq M(G)$. 
In particular, when $k$ is a global field, 
if $G=G_0(m,n,r,l)$ is metacyclic with trivial Schur multiplier $M(G)=0$ given as in Example \ref{exmc}, 
then 
$A(T)=0$, i.e. 
$T$ has the weak approximation property, 
and $\Sha(T)=0$, i.e. the Hasse norm principle holds for $K/k$ 
$($that is, Hasse principle holds for all torsors $E$ under $T$$)$ 
where $T=R^{(1)}_{K/k}(\bG_{m,K})$ is the norm one torus of $K/k$ 
$($see Section \ref{S3} and Ono's theorem $($Theorem \ref{thOno}$))$. 
\end{theorem}

\subsection{The $Z$-groups $G$ with $M(G)=0$}

\begin{example}[The $Z$-groups $G$ with $M(G)=0$]\label{exZgp}
A group $G$ is called {\it $Z$-group} if 
all the Sylow subgroups of $G$ is cyclic 
(see Suzuki \cite[page 658]{Suz55}). 
For example, $|G|$ is square-free, then $G$ is a $Z$-group. 
If $G$ is a $Z$-group, then $G$ is metacyclic 
(see e.g. Gorenstein \cite[Theorem 6.2, page 258]{Gor80}) 
and can be presented as 
\begin{align*}
G=\langle a,b\mid a^m=1, b^n=1, bab^{-1}=a^r\rangle\simeq C_m\rtimes C_n
\end{align*}
where 
$r^n\equiv 1\ ({\rm mod}\ m)$, $m$ is odd, $0\leq r<m$, 
$\gcd(m,n)=\gcd(m,r-1)=1$ 
(see Zassenhaus \cite[Satz 5, page 198]{Zas35}, 
see also Hall \cite[Theorem 9.4.3]{Hal59}, 
Robinson \cite[10.1.10, page 290]{Rob96}). 
Conversely, such $G$ becomes a $Z$-group of order $mn$. 
It follows from Lemma \ref{lemGp} that 
if $G$ is a $Z$-group, then $M(G)=0$ (see also Neumann \cite[page 190]{Neu55}).
\end{example}

\begin{theorem}
Let $K/k$, $L/k$, $G$ and $H$ be the same as in Theorem \ref{thmain}. 
If $G$ is metacyclic, then 
$\Sha^2_\omega(G,J_{G/H})\leq H^2(G,J_{G/H})\simeq M(G)$. 
In particular, when $k$ is a global field, 
if $G$ is a $Z$-group with trivial Schur multiplier $M(G)=0$ given as in Example \ref{exZgp}, 
then 
$A(T)=0$, i.e. 
$T$ has the weak approximation property, 
and $\Sha(T)=0$, i.e. the Hasse norm principle holds for $K/k$ 
$($that is, Hasse principle holds for all torsors $E$ under $T$$)$ 
where $T=R^{(1)}_{K/k}(\bG_{m,K})$ is the norm one torus of $K/k$ 
$($see Section \ref{S3} and Ono's theorem $($Theorem \ref{thOno}$))$.
\end{theorem}

\subsection{Dihedral $D_n$, quasidihedral $QD_{8m}$, modular $M_{16m}$ and generalized quaternion $Q_{4m}$ groups}\label{ssDQDMQ}~\\

Let $G$ be a finite group and $Z(G)$ be the center of $G$, 
$[a,b]:=a^{-1}b^{-1}ab$ be the commutator of $a,b\in G$, 
$[G,G]:=\langle [a,b]\mid a,b\in G\rangle$ be the commutator subgroup of $G$ 
and $G^{ab}:=G/[G,G]$ be the abelianization of $G$, i.e. the maximal abelian quotient of $G$. 
\begin{example}[Metacyclic groups $G$ with $M(G)=0$: dihedral groups $D_n$; quasidihedral groups $QD_{8m}$; modular groups $M_{16m}$; generalized quaternion groups $Q_{4m}$]\label{exDQDMQ}
We refer to Neumann \cite{Neu55}, 
Beyl and Tappe \cite[IV.2, page 193]{BT82} 
and 
Karpilovsky \cite{Kar87}, \cite{Kar93} for Schur multipliers $M(G)$.\\
(1) Let 
\begin{align*}
G=D_n=\langle\sigma,\tau\mid\sigma^n
=\tau^2=1, \tau\sigma\tau^{-1}=\sigma^{-1}\rangle
=\langle\sigma\rangle\rtimes\langle\tau\rangle\simeq C_n\rtimes C_2
\end{align*}
be {\it the dihedral group of order $2n$} $(n\geq 3)$. 
We have $G=\{\sigma^i,\sigma^i\tau\mid 1\leq i\leq n\}$. 
By the condition $\bigcap_{\sigma\in G}H^\sigma=\{1\}$, we have $H=\{1\}$ or $|H|=2$ 
with $H\neq Z(G)=\langle\sigma^{n/2}\rangle$ $($when $n$ is even$)$. 
Then 
\begin{align*}
Z(D_n)&=
\begin{cases}
\langle\sigma^{n/2}\rangle\simeq C_2,\\
\{1\}, 
\end{cases}\quad 
(D_n)^{ab}=
\begin{cases}
D_n/\langle\sigma^2\rangle\simeq C_2\times C_2,\\
D_n/\langle\sigma\rangle\simeq C_2,
\end{cases}\\
M(D_n)&\simeq
\begin{cases}
\bZ/2\bZ & {\rm if}\quad n\ {\rm is\ even},\\
0 & {\rm if}\quad n\ {\rm is\ odd}, 
\end{cases}
\end{align*} 
see \cite[Example 2.12, page 201]{BT82}, 
\cite[Proposition 2.11.4, page 100]{Kar87}, 
\cite[Corollary 1.27, page 292]{Kar93} for $M(D_n)$.\\
(2) Let 
\begin{align*}
G=QD_{8m}=\langle\sigma,\tau\mid\sigma^{8m}
=\tau^2=1, \tau\sigma\tau^{-1}=\sigma^{4m-1}\rangle\simeq \langle\sigma\rangle\rtimes\langle\tau\rangle\simeq C_{8m}\rtimes C_2\ (m\geq 1). 
\end{align*}
We have $G=\{\sigma^i,\sigma^i\tau\mid 1\leq i\leq 8m\}$ and $|G|=16m$. 
By the condition $\bigcap_{\sigma\in G}H^\sigma=\{1\}$, 
we have $H=\{1\}$ or $|H|=2$ with $H\neq Z(G)=\langle\sigma^{4m}\rangle\simeq C_2$. 
When $8m=2^d$ $(d\geq 3)$, the group $QD_{2^d}$ is called {\it the quasidihedral group of order $2^{d+1}$}. 
Then 
\begin{align*} 
Z(QD_{8m})=\langle\sigma^{4m}\rangle\simeq C_2,\quad 
(QD_{8m})^{ab}= QD_{8m}/\langle\sigma^2\rangle\simeq C_2\times C_2,\quad M(QD_{8m})=0,
\end{align*}
see \cite[page 192]{Neu55}, 
\cite[Example 2.12, page 202]{BT82}, 
\cite[Corollary 1.27, page 292]{Kar93} for $M(QD_{2^d})$.\\
(3) Let 
\begin{align*}
G=M_{16m}=\langle\sigma,\tau\mid\sigma^{8m}
=\tau^2=1, \tau\sigma\tau^{-1}=\sigma^{4m+1}\rangle\simeq \langle\sigma\rangle\rtimes\langle\tau\rangle\simeq C_{8m}\rtimes C_2\ (m\geq 1).
\end{align*}
We have $G=\{\sigma^i,\sigma^i\tau\mid 1\leq i\leq 8m\}$ and $|G|=16m$. 
By the condition $\bigcap_{\sigma\in G}H^\sigma=\{1\}$, 
we have $H=\{1\}$ or $|H|=2$ with $H\neq \langle\sigma^{4m}\rangle\simeq C_2\leq Z(G)=\langle\sigma^2\rangle\simeq C_{4m}$. 
When $16m=2^d$ $(d\geq 4)$, the group $M_{2^d}$ is called {\it the modular group of order $2^d$}. 
Then 
\begin{align*}
Z(M_{16m})=\langle \sigma^2\rangle\simeq C_{4m},\quad 
(M_{16m})^{ab}=M_{16m}/\langle\sigma^{4m}\rangle\simeq C_{4m}\times C_2,\quad 
M(M_{16m})=0,
\end{align*}
see \cite[page 193]{Neu55}, 
\cite[Example 2.4.9, page 53]{Kar87}, 
\cite[Corollary 1.27, page 292]{Kar93} for $M(M_{2^d})$.\\
(4) Let 
\begin{align*}
G=Q_{4m}=\langle\sigma,\tau\mid\sigma^{2m}=\tau^4=1, 
\sigma^m=\tau^2, \tau\sigma\tau^{-1}=\sigma^{-1}\rangle
\end{align*}
be {\it a generalized quaternion group of order $4m$} $(m\geq 2)$. 
The group $G$ is metacyclic because $\langle\sigma\rangle\simeq C_{2m}$ and $G/\langle\sigma\rangle\simeq C_2$. 
If $m$ is odd, then 
\begin{align*}
Q_{4m}=\langle\sigma^2,\tau\rangle\simeq C_m\rtimes C_4.
\end{align*}
If $m$ is even, then 
\begin{align*}
Q_{4m}=\langle \sigma^{2^{t+1}}\rangle\rtimes \langle \sigma^{m^\prime},\tau\rangle\simeq C_{m^\prime}\rtimes Q_{2^{t+2}}
\end{align*}
where $m=2^tm^\prime$ $(t\geq 1)$ with $m^\prime$ odd. 
This is a non-split metacyclic group, i.e not a semidirect product of two cyclic groups. 
By the condition $\bigcap_{\sigma\in G}H^\sigma=\{1\}$, we have $H=\{1\}$.  
Then 
\begin{align*}
Z(Q_{4m})&=\langle\sigma^n\rangle=\langle\tau^2\rangle\simeq C_2,\quad 
(Q_{4m})^{ab}=Q_{4m}/\langle\sigma^2\rangle\simeq 
\begin{cases}
C_2\times C_2 & {\rm if}\quad m\ {\rm is\ even},\\
C_4 & {\rm if}\quad m\ {\rm is\ odd}, 
\end{cases}\\
M(Q_{4m})&=0, 
\end{align*}
see \cite[Example 2.12, page 202]{BT82}, 
\cite[Example 2.4.8, page 52]{Kar87}, 
\cite[Corollary 1.27, page 292]{Kar93} for $M(Q_{4m})$. 
\end{example}
\begin{remark} 
(1) Let $G$ be a non-abelian $p$-group of order $p^n$ 
which contains a cyclic subgroup 
of order $p^{n-1}$. 
Then 
{\rm (i)} if $p$ is odd, then $G\simeq M_{p^n}$; 
{\rm (ii)} if $p=2$ and $n=3$, then $G\simeq D_4$ or $Q_8$; 
{\rm (iii)} if $p=2$ and $n\geq 4$, then $G\simeq D_{2^{n-1}}$, 
$QD_{2^{n-1}}$, $M_{2^n}$ or $Q_{2^n}$ 
(see Gorenstein \cite[Theorem 4.4, page 193]{Gor80}, Suzuki \cite[Theorem 4.1, page 54]{Suz86}).\\
(2) Let $G$ be a non-abelian $2$-groups of order $2^n$ $(n\geq 4)$. 
Then the following conditions are equivalent: 
{\rm (i)} $G$ is of maximal class, 
i.e. nilpotency class of $G$ is $n-1$ 
$($cf. \cite[page 21]{Gor80}, \cite[Definition 1.7, page 89]{Suz82}$)$; 
{\rm (ii)} $|G^{ab}|=4$; 
{\rm (iii)} $G\simeq D_{2^{n-1}}$, $QD_{2^{n-1}}$ or $Q_{2^n}$ 
(see Gorenstein \cite[Theorem 4.5, page 194]{Gor80}, Suzuki \cite[Exercise 2, page 32]{Suz86}). 
\end{remark}

\begin{theorem}
Let $K/k$, $L/k$, $G$ and $H$ be the same as in Theorem \ref{thmain}. 
If $G$ is metacyclic, then 
$\Sha^2_\omega(G,J_{G/H})\leq H^2(G,J_{G/H})\simeq M(G)$. 
In particular, when $k$ is a global field, 
if $G$ is one of $D_n$, $QD_{8m}$, $M_{16m}$, $Q_{4m}$ with trivial Schur multiplier $M(G)=0$ 
given as in Example \ref{exDQDMQ}, 
then 
$A(T)=0$, i.e. 
$T$ has the weak approximation property, 
and $\Sha(T)=0$, i.e. the Hasse norm principle holds for $K/k$ 
$($that is, Hasse principle holds for all torsors $E$ under $T$$)$ 
where $T=R^{(1)}_{K/k}(\bG_{m,K})$ is the norm one torus of $K/k$ 
$($see Section \ref{S3} and Ono's theorem $($Theorem \ref{thOno}$))$.
\end{theorem}

\begin{remark}[$G=D_n$: the dihedral group of order $2n$]
Let $k$ be a global field, $K/k$ be a finite separable field extension and 
$L/k$ be the Galois closure of $K/k$ with Galois groups 
$G={\rm Gal}(L/k)\simeq D_n$ and $H={\rm Gal}(L/K)\lneq G$. 
By the condition $\bigcap_{\sigma\in G}H^\sigma=\{1\}$, we have 
(i) $H=\{1\}$ or  (ii) $|H|=2$ and $H\neq Z(G)$.\\
(i) When $H=\{1\}$, i.e. $K/k$ is Galois. 
If $n$ is odd, then $H^2(G,J_G)\simeq H^3(G,\bZ)\simeq M(G)=0$ and hence by Tate's theorem (Theorem \ref{thTate})  
the Hasse norm principle holds for $K/k$. 
If $n$ is even, then $H^2(G,J_G)\simeq H^3(G,\bZ)\simeq M(G)\simeq\bZ/2\bZ$. 
By Tate's theorem (Theorem \ref{thTate}) again, we find that   
the Hasse norm principle holds for $K/k$ if and only if there exists a (ramified) place $v$ of $k$ 
such that $V_4\leq G_v$ because $H^2(G,J_G)\xrightarrow{\rm res}H^2(H^\prime,J_G)$ is injective if and only if 
$H^2(H^\prime,J_G)\simeq H^3(H^\prime,\bZ)\simeq\bZ/2\bZ$ if and only if $V_4\leq H^\prime$ for any $H^\prime\leq G$.\\ 
(ii) When $|H|=2$ and $H\neq Z(G)$. 
If $n$ is odd, then it follows from Proposition \ref{prop4.1} (2) that 
$\Sha^2_\omega(G,J_{G/H})\leq H^2(G,J_{G/H})\simeq M(G)=0$. 
Indeed, Bartels \cite[Lemma 1]{Bar81b} showed that 
the Hasse norm principle holds for $K/k$. 
We note that $\Sha^2_\omega(G,J_{G/H})=0$ follows also from the retract $k$-rationality of 
$R^{(1)}_{K/k}(\bG_{m,K})$ due to Endo \cite[Theorem 3.1]{End11}. 
If $n$ is even, then $\Sha^2_\omega(G,J_{G/H})\leq H^2(G,J_{G/H})\simeq M(G)\simeq \bZ/2\bZ$. 
However, Bartels \cite[Lemma 3]{Bar81b} showed that the following theorem: 
\begin{theorem}[{Bartels \cite{Bar81b}}]
Let $G=D_n$ be the dihedral group of order $2n$ $(n\geq 3)$ 
and $H\leq G$ with $|H|=2$ and $H\neq Z(G)$. 
If $n=2^m$, then $H_1(H,I_{G/H})\xrightarrow{\rm cor} H_1(G,I_{G/H})\simeq \bZ/2\bZ$ 
is surjective, i.e. ${\rm Ker}\{H^2(G,J_{G/H})\simeq \bZ/2\bZ\xrightarrow{\rm res}H^2(H,J_{G/H})\}=0$.  
This implies that 
$\Sha^2_\omega(G,J_{G/H})=0$. 
In particular, when $k$ is a global field, 
the Hasse norm principle holds for $K/k$ 
with $[K:k]=n$ and ${\rm Gal}(L/k)\simeq D_n$
where $L/k$ is the Galois closure of $K/k$. 
\end{theorem}
\end{remark}
\subsection{The metacyclic $p$-groups $G$ with $M(G)=0$}\label{ssMetacyclic}
\begin{example}[The metacyclic $p$-groups $G$ with $M(G)=0$]\label{ex1.7}
Beyl \cite{Bey73} showed that the Schur multiplier of a finite metacyclic group 
is cyclic, and computed its order 
(see also Beyl and Jones \cite{BJ84}). 
A complete list of metacyclic $p$-groups $G$ with trivial Schur multiplier $M(G)=0$ is given as follows (Beyl \cite[Section 5, page 417]{Bey73}): 
\begin{align*}
&{\rm I}& &G(p^m,p^n,p^{m-n}+1,0)\simeq C_{p^m}\rtimes C_{p^n}& &p\ {\rm odd}, m>n\geq 1;\\
&{\rm II}& &G(p^m,p^n,p^k+1,1) & &p\ {\rm odd}, n\geq 3, 2n-2>m>2,\\
& & & & & \tfrac{m}{2}>k\geq {\rm max}(1,m-n+1);\\
&{\rm III}& &G(2^m,2^n,2^{m-n}+1,0)\simeq C_{2^m}\rtimes C_{2^n}& & m-2\geq n\geq 1;\\
&{\rm IV}& &G(2^m,2^n,2^{m-n}-1,0)\simeq C_{2^m}\rtimes C_{2^n}& &m-2\geq n\geq 1;\\
&{\rm V}& &G(2^m,2^n,2^k+1,1)& &n\geq 4, 2n-2>m>4,\\
& & & & & \tfrac{m}{2}>k\geq {\rm max}(2,m-n+1);\\
&{\rm VI}& &G(2^m,2,-1,1)\simeq Q_{2^{m+1}}& &m\geq 2;\\
&{\rm VII}& &G(2^m,2^n,2^k-1,1)& &n\geq 2, m\geq 3, m>k\geq {\rm max}(2,m-n+1)
\end{align*}
where 
\begin{align*}
G(M,N,r,\lambda)=\langle a,b\mid a^M=1, b^N=a^{M\lambda/\gcd(M,r-1)}, bab^{-1}=a^{r}\rangle
\end{align*}
and $r^N\equiv 1\ ({\rm mod}\ M)$. 
When $n=1$, we find that $G(2^m,2,2^{m-1}+1,0)\simeq M_{2^{m+1}}$ in case III
and $G(2^m,2,2^{m-1}-1,0)\simeq QD_{2^m}$ in case IV. 
For cases I and III, see also Neumann \cite[page 193]{Neu55} and 
Karpilovsky \cite[Example 2.4.9, page 53]{Kar87}. 
\end{example}
\begin{theorem}
Let $K/k$, $L/k$, $G$ and $H$ be the same as in Theorem \ref{thmain}. 
If $G$ is metacyclic, then 
$\Sha^2_\omega(G,J_{G/H})\leq H^2(G,J_{G/H})\simeq M(G)$. 
In particular, when $k$ is a global field, 
if $G$ is a metacyclic $p$-group with trivial Schur multiplier $M(G)=0$ given as in Example \ref{ex1.7}, 
then 
$A(T)=0$, i.e. 
$T$ has the weak approximation property, 
and $\Sha(T)=0$, i.e. the Hasse norm principle holds for $K/k$ 
$($that is, Hasse principle holds for all torsors $E$ under $T$$)$ 
where $T=R^{(1)}_{K/k}(\bG_{m,K})$ is the norm one torus of $K/k$ 
$($see Section \ref{S3} and Ono's theorem $($Theorem \ref{thOno}$))$.
\end{theorem}

\subsection{The extraspecial group $G=E_{p^2}(p^3)$ of order $p^3$ with exponent $p^2$ with $M(G)=0$}\label{ssExtra}
~\\

A $p$-group $G$ is called {\it special} if 
either (i) $G$ is elementary abelian; or 
(ii) $Z(G)=[G,G]=\Phi(G)$ and $Z(G)$ is elementary abelian 
where $Z(G)$ is the center of $G$, 
$[G,G]$ is the commutator subgroup of $G$ and 
$\Phi(G)$ is the Frattini subgroup of $G$, i.e. 
the intersection of all maximal subgroups of $G$ 
(see e.g. Suzuki \cite[Definition 4.14, page 67]{Suz86}). 
It follows from $[G,G]=\Phi(G)$ that $G^{ab}=G/[G,G]$ 
is elementary abelian (see Karpilovsky \cite[Lemma 3.1.2, page 114]{Kar87}). 
A $p$-group $G$ is called {\it extraspecial} if 
$G$ is non-abelian, special and $Z(G)$ is cyclic. 
We see that $Z(G)\simeq C_p$. 
If $G$ is an extraspecial $p$-group, then $|G|=p^{2n+1}$ 
and there exist exactly $2$ extraspecial groups of order $p^{2n+1}$. 
For odd prime $p\geq 3$, they are 
$E_p(p^{2n+1})$, $E_{p^2}(p^{2n+1})$  
of order $p^{2n+1}$ with exponent $p$ and $p^2$ respectively 
(see 
Gorenstein \cite[Section 5.3, page 183]{Gor80}, 
Suzuki \cite[Definition 4.14, page 67]{Suz86}). 
Note that $D_4$ and $Q_8$ are extraspecial groups of order $8$ 
with exponent $4$. 
\begin{example}[$G=E_{p^2}(p^3)$: the extraspecial group of order $p^3$ with exponent $p^2$ with $M(G)=0$ (the case $m=2, n=1$ as in Example \ref{ex1.7} (I)]\label{exExtra}
Let $k$ be a field, $K/k$ be a finite 
separable field extension and $L/k$ be the Galois closure of $K/k$ with 
Galois groups $G={\rm Gal}(L/k)$ and $H={\rm Gal}(L/K)\lneq G$.
Let $p\geq 3$ be a prime number. 
Assume that 
\begin{align*}
G=E_{p^2}(p^3)&=\langle a,b,c\mid a^p=1, b^p=c, c^p=1, [b,a]=c, [c,a]=1, [c,b]=1\rangle\\
&=\langle b\rangle\rtimes\langle a\rangle\simeq C_{p^2}\rtimes C_p
\end{align*}
is the extraspecial group of order $p^3$ with exponent $p^2$ 
which is the modular $p$-group of order $p^3$. 
By the condition $\bigcap_{\sigma\in G}H^\sigma=\{1\}$, we have $H=\{1\}$ or $
H\simeq C_p$ with $H\neq Z(G)=\langle c\rangle\simeq C_p$. 
We also have $M(G)=0$ 
(Beyl and Tappe \cite[Corollary 4.16 (b), page 223]{BT82}, see also Karpilovsky \cite[Theorem 3.3.6, page 138]{Kar87}). 
\end{example} 
\begin{theorem}
Let $K/k$, $L/k$, $G$ and $H$ be the same as in Theorem \ref{thmain}. 
If $G$ is metacyclic, then 
$\Sha^2_\omega(G,J_{G/H})\leq H^2(G,J_{G/H})\simeq M(G)$. 
In particular, when $k$ is a global field, 
if $G=E_{p^2}(p^3)$ is the extraspecial group of order $p^3$ 
with exponent $p^2$ $(p\geq 3)$ 
with trivial Schur multiplier $M(G)=0$ given as in Example \ref{exExtra}, 
then 
$A(T)=0$, i.e. 
$T$ has the weak approximation property, 
and $\Sha(T)=0$, i.e. the Hasse norm principle holds for $K/k$ 
$($that is, Hasse principle holds for all torsors $E$ under $T$$)$ 
where $T=R^{(1)}_{K/k}(\bG_{m,K})$ is the norm one torus of $K/k$ 
$($see Section \ref{S3} and Ono's theorem $($Theorem \ref{thOno}$))$.
\end{theorem}

\section{Metacyclic groups $G=nTm\leq S_n$ $(2\leq n\leq 30)$ and $H\leq G$ with $[G:H]=n$}\label{S6}

Let $k$ be a field, 
$K/k$ be a separable field extension with $[K:k]=n$ 
and $L/k$ be the Galois closure of $K/k$.  
Let $G={\rm Gal}(L/k)$ and $H={\rm Gal}(L/K)$ with $[G:H]=n$. 
Then we have $\bigcap_{\sigma\in G} H^\sigma=\{1\}$ 
where $H^\sigma=\sigma^{-1}H\sigma$ and hence 
$H$ contains no normal subgroup of $G$ except for $\{1\}$. 
The Galois group $G$ may be regarded as a transitive subgroup of 
the symmetric group $S_n$ of degree $n$. 
We may assume that 
$H$ is the stabilizer of one of the letters in $G\leq S_n$, 
i.e. $L=k(\theta_1,\ldots,\theta_n)$ and $K=L^H=k(\theta_i)$ 
where $1\leq i\leq n$. 

Let $nTm$ be the $m$-th transitive subgroup of the symmetric group 
$S_n$ of degree $n$ up to conjugacy 
(see Butler and McKay \cite{BM83}, \cite{GAP}). 
Let $p$ be a prime number and 
$F_{pl}\simeq C_p\rtimes C_l$ $(l\mid p-1)$ 
be the Frobenius group of order $pl$. 
By using GAP (\cite{GAP}), we can obtain all the transitive subgroup 
$G=nTm\leq S_n$ $(2\leq n\leq30)$ which is metacyclic 
with trivial Schur multiplier $M(G)=0$ as in Table $1$ 
(see Section \ref{GAPcomp} for GAP computations). 
The GAP algorithms can be available as {\tt HNP.gap} 
in \cite{Norm1ToriHNP}. 
Then by Theorem \ref{thmain} we get: 
\begin{theorem}\label{th6.1}
Let $k$ be a field, 
$K/k$ be a separable field extension with $[K:k]=n$ 
and $L/k$ be the Galois closure of $K/k$. 
Assume that $G={\rm Gal}(L/k)=nTm$ $(2\leq n\leq 30)$ 
is a transitive subgroup of $S_n$ which is metacyclic 
with trivial Schur multiplier $M(G)=0$ 
and $H=Gal(L/K)$ with $[G:H]=n$. 
Then $G$ is given as in Table $1$ which satisfies 
$\Sha^2_\omega(G,J_{G/H})=H^2(G,J_{G/H})\simeq M(G)=0$. 
In particular, when $k$ is a global field, 
$A(T)=0$, i.e. 
$T$ has the weak approximation property, 
and $\Sha(T)=0$, i.e. the Hasse norm principle holds for $K/k$ 
$($that is, Hasse principle holds for all torsors $E$ under $T$$)$ 
where $T=R^{(1)}_{K/k}(\bG_{m,K})$ is the norm one torus of $K/k$ 
$($see Section \ref{S3} and Ono's theorem $($Theorem \ref{thOno}$))$. 
\end{theorem}

\begin{remark}
Hoshi, Kanai and Yamasaki \cite{HKY22}, \cite{HKY23}, \cite{HKY25} 
determined $H^1(k,{\rm Pic}\, \overline{X})\simeq \Sha^2_\omega(G,J_{G/H})$ 
where $T=R^{(1)}_{K/k}(\bG_{m,K})$ is the norm one torus of $K/k$ 
with $[K:k]=n\leq 16$, $\widehat{T}\simeq J_{G/H}$, 
$X$ is a smooth $k$-compactification of $T$ and 
$\overline{X}=X\times_k\overline{k}$. 
Moreover, 
Hoshi, Kanai and Yamasaki \cite{HKY22}, \cite{HKY23} 
gives a necessary and sufficient condition 
for the Hasse norm principle for $K/k$, i.e. 
$\Sha(T)$ where $T=R^{(1)}_{K/k}(\bG_{m,K})$ is the norm one torus of $K/k$ 
with $[K:k]=n\leq 15$. 
\end{remark}
Similarly, by using GAP (\cite{GAP}), 
we can obtain all the transitive subgroup $G=nTm\leq S_n$ $(2\leq n\leq 30)$ 
which is metacyclic 
with $M(G)\neq 0$ but $M(H)=0$ where $H\leq G$ with $[G:H]=n$ 
as in Table $2$ (see Section \ref{GAPcomp} for GAP computations). 
Hence by Proposition \ref{prop4.1} (2) we get: 
\begin{theorem}\label{th6.2}
Let $k$ be a field, 
$K/k$ be a separable field extension with $[K:k]=n$ 
and $L/k$ be the Galois closure of $K/k$. 
Assume that $G={\rm Gal}(L/k)=nTm$ $(2\leq n\leq 30)$ 
is a transitive subgroup of $S_n$ which is metacyclic 
with $M(G)\neq 0$ but $M(H)=0$ where $H=Gal(L/K)$ with $[G:H]=n$. 
Then $G$, $H$ and $M(G)$ are given as in Table $2$ which satisfy  
$\Sha^2_\omega(G,J_{G/H})\leq H^2(G,J_{G/H})\simeq M(G)$. 
\end{theorem}

\newpage 
\begin{center}
Table 1: $G=nTm\leq S_n$ $(2\leq n\leq 30)$: metacyclic with $M(G)=0$ 
and $H\leq G$ with $[G:H]=n$ 
which satisfy  
$\Sha^2_\omega(G,J_{G/H})=H^2(G,J_{G/H})=0$ as in Theorem 4.2
\vspace*{2mm}\\
\renewcommand{\arraystretch}{1.05}
{\small 
\begin{tabular}{lll} 
$nTm$ & $G$ & $H$\\\hline
$2T1$ & $C_2$ & $\{1\}$\\
$3T1$ & $C_3$ & $\{1\}$\\
$3T2$ & $S_3\simeq C_3\rtimes C_2$ & $C_2$\\
$4T1$ & $C_4$ & $\{1\}$\\
$5T1$ & $C_5$ & $\{1\}$\\
$5T2$ & $D_5\simeq C_5\rtimes C_2$ & $C_2$\\
$5T3$ & $F_{20}\simeq C_5\rtimes C_4$ & $C_4$\\
$6T1$ & $C_6$ & $\{1\}$\\
$6T2$ & $S_3\simeq C_3\rtimes C_2$ & $\{1\}$\\
$6T5$ & $C_3\times S_3$ & $C_3$\\
$7T1$ & $C_7$ & $\{1\}$\\
$7T2$ & $D_7\simeq C_7\rtimes C_2$ & $C_2$\\
$7T3$ & $F_{21}\simeq C_7\rtimes C_3$ & $C_3$\\
$7T4$ & $F_{42}\simeq C_7\rtimes C_6$ & $C_6$\\
$8T1$ & $C_8$ & $\{1\}$\\
$8T5$ & $Q_8$ & $\{1\}$\\
$8T7$ & $M_{16}\simeq C_8\rtimes C_2$ & $C_2$\\
$8T8$ & $QD_8\simeq C_8\rtimes C_2$ & $C_2$\\
$9T1$ & $C_9$ & $\{1\}$\\
$9T3$ & $D_9\simeq C_9\rtimes C_2$ & $C_2$\\
$9T4$ & $C_3\times S_3$ & $C_2$\\
$9T6$ & $E_9(27)\simeq C_9\rtimes C_3$ & $C_3$\\
$9T10$ & $C_9\rtimes C_6$ & $C_6$\\
$10T1$ & $C_{10}$ & $\{1\}$\\
$10T2$ & $D_5\simeq C_5\rtimes C_2$ & $\{1\}$\\
$10T4$ & $F_{20}\simeq C_5\rtimes C_4$ & $C_2$\\
$10T6$ & $C_5\times D_5$ & $C_5$\\
$11T1$ & $C_{11}$ & $\{1\}$\\
$11T2$ & $D_{11}\simeq C_{11}\rtimes C_2$ & $C_2$\\
$11T3$ & $F_{55}\simeq C_{11}\rtimes C_5$ & $C_5$\\
$11T4$ & $F_{110}\simeq C_{11}\rtimes C_{10}$ & $C_{10}$\\
$12T1$ & $C_{12}$ & $\{1\}$\\
$12T5$ & $Q_{12}\simeq C_3\rtimes C_4$ & $\{1\}$\\
$12T19$ & $C_3\times Q_{12}$ & $C_3$\\
$13T1$ & $C_{13}$ & $\{1\}$\\
$13T2$ & $D_{13}\simeq C_{13}\rtimes C_2$ & $C_2$\\
$13T3$ & $F_{39}\simeq C_{13}\rtimes C_3$ & $C_3$\\
$13T4$ & $F_{52}\simeq C_{13}\rtimes C_4$ & $C_4$\\
$13T5$ & $F_{78}\simeq C_{13}\rtimes C_6$ & $C_6$\\
$13T6$ & $F_{156}\simeq C_{13}\rtimes C_{12}$ & $C_{12}$\\
$14T1$ & $C_{14}$ & $\{1\}$\\
\end{tabular}\hspace*{5mm}
\begin{tabular}{lll}
$nTm$ & $G$ & $H$\\\hline
$14T2$ & $D_7\simeq C_7\rtimes C_2$ & $\{1\}$\\
$14T4$ & $F_{42}\simeq C_7\rtimes C_6$ & $C_3$\\
$14T5$ & $C_2\times F_{21}$ & $C_3$\\
$14T8$ & $C_7\times D_7$ & $C_7$\\
$15T1$ & $C_{15}$ & $\{1\}$\\
$15T2$ & $D_{15}\simeq C_{15}\rtimes C_2$ & $C_2$\\
$15T3$ & $C_3\times D_5$ & $C_2$\\
$15T4$ & $C_5\times S_3$ & $C_2$\\
$15T6$ & $C_{15}\rtimes C_4$ & $C_4$\\
$15T8$ & $C_3\times F_{20}$ & $C_4$\\
$16T1$ & $C_{16}$ & $\{1\}$\\
$16T6$ & $M_{16}\simeq C_8\rtimes C_2$ & $\{1\}$\\
$16T12$ & $QD_8\simeq C_8\rtimes C_2$ & $\{1\}$\\
$16T14$ & $Q_{16}$ & $\{1\}$\\
$16T22$ & $M_{32}\simeq C_{16}\rtimes C_2$ & $C_2$\\
$16T49$ & $C_4.(C_4\times C_2)$ & $C_2$\\
$16T55$ & $QD_{16}\simeq C_{16}\rtimes C_2$ & $C_2$\\
$16T124$ & $C_4.(C_8\times C_2)$ & $C_4$\\
$16T125$ & $C_{16}\rtimes C_4$ & $C_4$\\
$16T136$ & $C_{16}\rtimes C_4$ & $C_4$\\ 
$17T1$ & $C_{17}$ & $\{1\}$\\
$17T2$ & $D_{17}\simeq C_{17}\rtimes C_2$ & $C_2$\\
$17T3$ & $F_{68}\simeq C_{17}\rtimes C_4$ & $C_4$\\
$17T4$ & $F_{136}\simeq C_{17}\rtimes C_8$ & $C_8$\\
$17T5$ & $F_{272}\simeq C_{17}\rtimes C_{16}$ & $C_{16}$\\
$18T1$ & $C_{18}$ & $\{1\}$\\
$18T3$ & $C_3\times S_3$ & $\{1\}$\\
$18T5$ & $D_9\simeq C_9\rtimes C_2$ & $\{1\}$\\
$18T14$ & $C_2\times E_9(27)$ & $C_3$\\
$18T16$ & $C_9\times S_3$ & $C_3$\\
$18T18$ & $C_9\rtimes C_6$ & $C_3$\\
$18T19$ & $C_3\times D_9$ & $C_3$\\
$18T74$ & $C_9\times D_9$ & $C_9$\\
$18T80$ & $C_9^2\rtimes C_2$ & $C_9$\\
$19T1$ & $C_{19}$ & $\{1\}$\\
$19T2$ & $D_{19}\simeq C_{19}\rtimes C_2$ & $C_2$\\
$19T3$ & $F_{57}\simeq C_{19}\rtimes C_3$ & $C_3$\\
$19T4$ & $F_{114}\simeq C_{19}\rtimes C_6$ & $C_6$\\
$19T5$ & $F_{171}\simeq C_{19}\rtimes C_9$ & $C_9$\\
$19T6$ & $F_{342}\simeq C_{19}\rtimes C_{18}$ & $C_{18}$\\
$20T1$ & $C_{20}$ & $\{1\}$
\end{tabular}
}

\newpage
Table 1 (continued): $G=nTm\leq S_n$ $(2\leq n\leq 30)$: metacyclic with $M(G)=0$ 
and $H\leq G$ with $[G:H]=n$ 
which satisfy  
$\Sha^2_\omega(G,J_{G/H})=H^2(G,J_{G/H})=0$ as in Theorem 4.2
\vspace*{2mm}\\
\renewcommand{\arraystretch}{1.05}
{\small 
\begin{tabular}{lll}
$nTm$ & $G$ & $H$\\\hline
$20T2$ & $Q_{20}\simeq C_5\rtimes C_4$ & $\{1\}$\\
$20T5$ & $F_{20}\simeq C_5\rtimes C_4$ & $\{1\}$\\
$20T25$ & $C_5\times Q_{20}$ & $C_5$\\
$20T29$ & $C_5\times F_{20}$ & $C_5$\\
$21T1$ & $C_{21}$ & $\{1\}$\\
$21T2$ & $F_{21}\simeq C_7\rtimes C_3$ & $\{1\}$\\
$21T3$ & $C_3\times D_7$ & $C_2$\\
$21T4$ & $F_{42}\simeq C_7\rtimes C_6$ & $C_2$\\
$21T5$ & $D_{21}\simeq C_{21}\rtimes C_2$ & $C_2$\\
$21T6$ & $C_7\times S_3$ & $C_2$\\
$21T10$ & $C_7\times (C_3\times S_3)$ & $C_6$\\
$21T11$ & $F_{21}\times S_3$ & $C_6$\\
$21T13$ & $C_7\times F_{21}$ & $C_7$\\
$22T1$ & $C_{22}$ & $\{1\}$\\
$22T2$ & $D_{11}\simeq C_{11}\rtimes C_2$ & $\{1\}$\\
$22T4$ & $F_{110}\simeq C_{11}\rtimes C_{10}$ & $C_5$\\
$22T5$ & $C_2\times F_{55}$ & $C_5$\\
$22T7$ & $C_{11}\times D_{11}$ & $C_{11}$\\
$23T1$ & $C_{23}$ & $\{1\}$\\
$23T2$ & $D_{23}\simeq C_{23}\rtimes C_2$ & $C_2$\\
$23T3$ & $F_{253}\simeq C_{23}\rtimes C_{11}$ & $C_{11}$\\
$23T4$ & $F_{206}\simeq C_{23}\rtimes C_{22}$ & $C_{22}$\\
$24T1$ & $C_{24}$ & $\{1\}$\\
$24T4$ & $C_3\times Q_8$ & $\{1\}$\\
$24T5$ & $C_3\rtimes Q_8$ & $\{1\}$\\
$24T8$ & $C_3\rtimes C_8$ & $\{1\}$\\
$24T16$ & $C_3\times M_{16}$ & $C_2$\\
$24T20$ & $(C_3\rtimes C_8)\rtimes C_2$ & $C_2$\\
$24T31$ & $C_{24}\rtimes C_2$ & $C_2$\\
$24T35$ & $QD_{24}\simeq C_{24}\rtimes C_2$ & $C_2$\\
$24T41$ & $C_3\times QD_8$ & $C_2$\\
$24T64$ & $C_3\times (C_3\rtimes Q_8)$ & $C_3$\\
$24T69$ & $C_3\times (C_3\rtimes C_8)$ & $C_3$\\
$24T211$ & $C_3\times ((C_3\rtimes C_8)\rtimes C_2)$ & $C_6$\\
$25T1$ & $C_{25}$ & $\{1\}$\\
$25T3$ & $C_5\times D_5$ & $C_2$\\
$25T4$ & $D_{25}\simeq C_{25}\rtimes C_2$ & $C_2$\\
$25T7$ & $C_5\times F_{20}$ & $C_4$\\
$25T8$ & $C_{25}\rtimes C_4$ & $C_4$\\
$25T13$ & $E_{25}(125)\simeq C_{25}\rtimes C_5$ & $C_5$\\
$25T25$ & $E_{25}(125)\rtimes C_2$ & $C_{10}$
\end{tabular}\hspace*{5mm}
\begin{tabular}{lll}
$nTm$ & $G$ & $H$\\\hline
$25T40$ & $E_{25}(125)\rtimes C_4$ & $C_{20}$\\
$26T1$ & $C_{26}$ & $\{1\}$\\
$26T2$ & $D_{13}\simeq C_{13}\rtimes C_2$ & $\{1\}$\\
$26T4$ & $F_{52}\simeq C_{13}\rtimes C_4$ & $C_2$\\
$26T5$ & $C_2\times F_{39}$ & $C_3$\\
$26T6$ & $F_{78}\simeq C_{13}\rtimes C_6$ & $C_3$\\
$26T8$ & $F_{156}\simeq C_{13}\rtimes C_{12}$ & $C_6$\\
$26T11$ & $C_{13}\times D_{13}$ & $C_{13}$\\
$27T1$ & $C_{27}$ & $\{1\}$\\
$27T5$ & $E_9(27)\simeq C_9\rtimes C_3$ & $\{1\}$\\
$27T8$ & $D_{27}\simeq C_{27}\rtimes C_2$ & $C_2$\\
$27T9$ & $C_3\times D_9$ & $C_2$\\
$27T12$ & $C_9\times S_3$ & $C_2$\\
$27T14$ & $C_9\rtimes C_6$ & $C_2$\\
$27T22$ & $C_{27}\rtimes C_3$ & $C_3$\\
$27T55$ & $(C_{27}\rtimes C_3)\rtimes C_2$ & $C_6$\\
$27T107$ & $C_{27}\rtimes C_9$ & $C_9$\\
$27T176$ & $(C_{27}\rtimes C_9)\rtimes C_2$ & $C_{18}$\\
$28T1$ & $C_{28}$ & $\{1\}$\\
$28T3$ & $Q_{28}\simeq C_7\rtimes C_4$ & $\{1\}$\\
$28T12$ & $C_7\rtimes C_{12}$ & $C_3$\\
$28T13$ & $C_4\times F_{21}$ & $C_3$\\
$28T33$ & $C_7\times Q_{28}$ & $C_7$\\
$29T1$ & $C_{29}$ & $\{1\}$\\
$29T2$ & $D_{29}\simeq C_{29}\rtimes C_2$ & $C_2$\\
$29T3$ & $F_{116}\simeq C_{29}\rtimes C_4$ & $C_4$\\
$29T4$ & $F_{203}\simeq C_{29}\rtimes C_7$ & $C_7$\\
$29T5$ & $F_{406}\simeq C_{29}\rtimes C_{14}$ & $C_{14}$\\
$29T6$ & $F_{812}\simeq C_{29}\rtimes C_{28}$ & $C_{28}$\\
$30T1$ & $C_{30}$ & $\{1\}$\\
$30T2$ & $C_5\times S_3$ & $\{1\}$\\
$30T3$ & $D_{15}\simeq C_{15}\rtimes C_2$ & $\{1\}$\\
$30T4$ & $C_3\times D_5$ & $\{1\}$\\
$30T6$ & $C_{15}\rtimes C_4$ & $C_2$\\
$30T7$ & $C_3\times F_{20}$ & $C_2$\\
$30T15$ & $C_{15}\times S_3$ & $C_3$\\
$30T16$ & $C_3\times D_{15}$ & $C_3$\\
$30T36$ & $C_5\times D_{15}$ & $C_5$\\
$30T39$ & $C_{15}\times D_5$ & $C_5$\\
$30T47$ & $C_3\times (C_{15}\rtimes C_4)$ & $C_6$\\
$30T104$ & $C_{15}\times D_{15}$ & $C_{15}$\\
\end{tabular}
}
\end{center}

\newpage
\begin{center}
Table 2: $G=nTm\leq S_n$ $(2\leq n\leq 30)$: metacyclic  
with $M(G)\neq 0$ 
and $H\leq G$ with $[G:H]=n,$\\ 
$M(H)=0$ 
which satisfy 
$\Sha^2_\omega(G,J_{G/H})\leq H^2(G,J_{G/H})\simeq M(G)$ 
as in Proposition \ref{prop4.1} (2)
\vspace*{2mm}\\
\renewcommand{\arraystretch}{1.05}
{\small 
\begin{tabular}{llll} 
$nTm$ & $G$ & $H$ & $M(G)$\\\hline
$4T2$ & $C_2^2$ & $\{1\}$ & $\bZ/2\bZ$\\
$4T3$ & $D_4\simeq C_4\rtimes C_2$ & $C_2$ & $\bZ/2\bZ$\\
$6T3$ & $D_6\simeq C_6\rtimes C_2$ & $C_2$ & $\bZ/2\bZ$\\
$8T2$ & $C_4\times C_2$ & $\{1\}$ & $\bZ/2\bZ$\\
$8T4$ & $D_4\simeq C_4\rtimes C_2$ & $\{1\}$ & $\bZ/2\bZ$\\
$8T6$ & $D_8\simeq C_8\rtimes C_2$ & $C_2$ & $\bZ/2\bZ$\\
$9T2$ & $C_3^2$ & $\{1\}$ & $\bZ/3\bZ$\\
$10T3$ & $D_{10}\simeq C_{10}\rtimes C_2$ & $C_2$ & $\bZ/2\bZ$\\
$10T5$ & $C_2\times F_{20}$ & $C_4$ & $\bZ/2\bZ$\\
$12T2$ & $C_6\times C_2$ & $\{1\}$ & $\bZ/2\bZ$\\
$12T3$ & $D_6\simeq C_6\rtimes C_2$ & $\{1\}$ & $\bZ/2\bZ$\\
$12T11$ & $C_4\times S_3$ & $C_2$ & $\bZ/2\bZ$\\
$12T12$ & $D_{12}\simeq C_{12}\rtimes C_2$ & $C_2$ & $\bZ/2\bZ$\\
$12T14$ & $C_3\times D_4$ & $C_2$ & $\bZ/2\bZ$\\
$12T18$ & $C_6\times S_3$ & $C_3$ & $\bZ/2\bZ$\\
$14T3$ & $D_{14}\simeq C_{14}\rtimes C_2$ & $C_2$ & $\bZ/2\bZ$\\
$14T7$ & $C_2\times F_{42}$ & $C_6$ & $\bZ/2\bZ$\\
$16T4$ & $C_4^2$ & $\{1\}$ & $\bZ/4\bZ$\\
$16T5$ & $C_8\times C_2$ & $\{1\}$ & $\bZ/2\bZ$\\
$16T8$ & $C_4\rtimes C_4$ & $\{1\}$ & $\bZ/2\bZ$\\
$16T13$ & $D_8\simeq C_8\rtimes C_2$ & $\{1\}$ & $\bZ/2\bZ$\\
$16T56$ & $D_{16}\simeq C_{16}\rtimes C_2$ & $C_2$ & $\bZ/2\bZ$\\
$18T2$ & $C_6\times C_3$ & $\{1\}$ & $\bZ/3\bZ$\\
$18T6$ & $C_6\times S_3$ & $C_2$ & $\bZ/2\bZ$\\
$18T13$ & $D_{18}\simeq C_{18}\rtimes C_2$ & $C_2$ & $\bZ/2\bZ$\\
$18T45$ & $C_2\times (C_9\rtimes C_6)$ & $C_6$ & $\bZ/2\bZ$\\
$20T3$ & $C_{10}\times C_2$ & $\{1\}$ & $\bZ/2\bZ$\\
$20T4$ & $D_{10}\simeq C_{10}\rtimes C_2$ & $\{1\}$ & $\bZ/2\bZ$\\
$20T6$ & $C_4\times D_5$ & $C_2$ & $\bZ/2\bZ$\\
$20T9$ & $C_2\times F_{20}$ & $C_2$ & $\bZ/2\bZ$\\
$20T10$ & $D_{20}\simeq C_{20}\rtimes C_2$ & $C_2$ & $\bZ/2\bZ$\\
$20T12$ & $C_5\times D_4$ & $C_2$ & $\bZ/2\bZ$\\
$20T13$ & $C_2\times F_{20}$ & $C_2$ & $\bZ/2\bZ$\\
$20T18$ & $C_{20}\rtimes C_4$ & $C_4$ & $\bZ/2\bZ$\\
$20T20$ & $C_4\times F_{20}$ & $C_4$ & $\bZ/4\bZ$\\
$20T24$ & $C_{10}\times D_5$ & $C_5$ & $\bZ/2\bZ$\\
$21T7$ & $C_3\times F_{21}$ & $C_3$ & $\bZ/3\bZ$
\end{tabular}\hspace*{5mm}
\begin{tabular}{llll}
$nTm$ & $G$ & $H$ & $M(G)$\\\hline
$21T9$ & $C_3\times F_{42}$ & $C_6$ & $\bZ/3\bZ$\\
$22T3$ & $D_{22}\simeq C_{22}\rtimes C_2$ & $C_2$ & $\bZ/2\bZ$\\
$22T6$ & $C_2\times F_{110}$ & $C_{10}$ & $\bZ/2\bZ$\\
$24T2$ & $C_{12}\times C_2$ & $\{1\}$ & $\bZ/2\bZ$\\
$24T6$ & $C_2\times Q_{12}\simeq C_2\times (C_3\rtimes C_4)$ & $\{1\}$ & $\bZ/2\bZ$\\
$24T12$ & $C_4\times S_3$ & $\{1\}$ & $\bZ/2\bZ$\\
$24T13$ & $D_{12}\simeq C_{12}\rtimes C_2$ & $\{1\}$ & $\bZ/2\bZ$\\
$24T15$ & $C_3\times D_4$ & $\{1\}$ & $\bZ/2\bZ$\\
$24T32$ & $C_8\times S_3$ & $C_2$ & $\bZ/2\bZ$\\
$24T34$ & $D_{24}\simeq C_{24}\rtimes C_2$ & $C_2$ & $\bZ/2\bZ$\\
$24T40$ & $C_3\times D_8$ & $C_2$ & $\bZ/2\bZ$\\
$24T65$ & $C_{12}\times S_3$ & $C_3$ & $\bZ/2\bZ$\\
$24T66$ & $C_6\times Q_{12}\simeq C_6\times (C_3\rtimes C_4)$ & $C_3$ & $\bZ/2\bZ$\\
$24T67$ & $C_3\times D_{12}$ & $C_3$ & $\bZ/2\bZ$\\
$25T2$ & $C_5^2$ & $\{1\}$ & $\bZ/5\bZ$\\
$26T3$ & $D_{26}\simeq C_{26}\rtimes C_2$ & $C_2$ & $\bZ/2\bZ$\\
$26T7$ & $C_2\times F_{52}$ & $C_4$ & $\bZ/2\bZ$\\
$26T9$ & $C_2\times F_{78}$ & $C_6$ & $\bZ/2\bZ$\\
$26T10$ & $C_2\times F_{156}$ & $C_{12}$ & $\bZ/2\bZ$\\
$27T2$ & $C_9\times C_3$ & $\{1\}$ & $\bZ/3\bZ$\\
$28T2$ & $C_{14}\times C_2$ & $\{1\}$ & $\bZ/2\bZ$\\
$28T4$ & $D_{14}\simeq C_{14}\rtimes C_2$ & $\{1\}$ & $\bZ/2\bZ$\\
$28T5$ & $C_7\times D_4$ & $C_2$ & $\bZ/2\bZ$\\
$28T8$ & $C_4\times D_7$ & $C_2$ & $\bZ/2\bZ$\\
$28T10$ & $D_{28}\simeq C_{28}\rtimes C_2$ & $C_2$ & $\bZ/2\bZ$\\
$28T14$ & $C_2^2\times F_{21}$ & $C_3$ & $\bZ/2\bZ$\\
$28T15$ & $C_2\times F_{42}$ & $C_3$ & $\bZ/2\bZ$\\
$28T22$ & $D_4\times F_{21}$ & $C_6$ & $\bZ/2\bZ$\\
$28T23$ & $C_7\rtimes (C_3\times D_4)$ & $C_6$ & $\bZ/2\bZ$\\
$28T26$ & $C_4\times F_{42}$ & $C_6$ & $\bZ/2\bZ$\\
$28T34$ & $C_{14}\times D_7$ & $C_7$ & $\bZ/2\bZ$\\
$30T5$ & $C_6\times D_5$ & $C_2$ & $\bZ/2\bZ$\\
$30T12$ & $C_{10}\times S_3$ & $C_2$ & $\bZ/2\bZ$\\
$30T14$ & $D_{30}\simeq C_{30}\rtimes C_2$ & $C_2$ & $\bZ/2\bZ$\\
$30T17$ & $C_2\times (C_{15}\rtimes C_4)$ & $C_4$ & $\bZ/2\bZ$\\
$30T26$ & $C_6\times F_{20}$ & $C_4$ & $\bZ/2\bZ$\\
\\
\end{tabular}
}
\end{center}

\section{Examples of Proposition \ref{prop4.1}: Not metacyclic groups $G=nTm\leq S_n$ $(2\leq n\leq 19)$ and $H\leq G$ with $[G:H]=n$}\label{S7}

In this section, we provide some examples of 
Proposition \ref{prop4.1} (1), (2) for 
not metacyclic groups $G$ 
(see Section \ref{S6} for metacyclic groups $G$).
%
%

By using GAP (\cite{GAP}), we can obtain all the transitive subgroup 
$G=nTm\leq S_n$ $(2\leq n\leq 19)$ which is not metacyclic 
with trivial Schur multiplier $M(G)=0$ and 
$H\leq G$ with $[G:H]=n$ and $[G,G]\cap H=[H,H]$
as in Table $3$ (see Section \ref{GAPcomp} for GAP computations). 
Then by Proposition \ref{prop4.1} (1) we get: 
\begin{theorem}\label{th7.1}
Let $k$ be a field, 
$K/k$ be a separable field extension with $[K:k]=n$ 
and $L/k$ be the Galois closure of $K/k$. 
Assume that $G={\rm Gal}(L/k)=nTm$ $(2\leq n\leq 19)$ 
is a transitive subgroup of $S_n$ which is not metacyclic 
with trivial Schur multiplier $M(G)=0$ 
and $H=Gal(L/K)$ with $[G:H]=n$ and $[G,G]\cap H=[H,H]$. 
Then $G$ is given as in Table $3$ which satisfies 
$\Sha^2_\omega(G,J_{G/H})=H^2(G,J_{G/H})\simeq M(G)=0$. 
In particular, when $k$ is a global field, 
$A(T)=0$, i.e. 
$T$ has the weak approximation property, 
and $\Sha(T)=0$, i.e. the Hasse norm principle holds for $K/k$ 
$($that is, Hasse principle holds for all torsors $E$ under $T$$)$ 
where $T=R^{(1)}_{K/k}(\bG_{m,K})$ is the norm one torus of $K/k$ 
$($see Section \ref{S3} and Ono's theorem $($Theorem \ref{thOno}$))$. 
\end{theorem}
%
\begin{remark}
Hoshi, Kanai and Yamasaki \cite{HKY22}, \cite{HKY23}, \cite{HKY25} 
determined $H^1(k,{\rm Pic}\, \overline{X})\simeq \Sha^2_\omega(G,J_{G/H})$ 
where $T=R^{(1)}_{K/k}(\bG_{m,K})$ is the norm one torus of $K/k$ 
with $[K:k]=n\leq 16$, $\widehat{T}\simeq J_{G/H}$, 
$X$ is a smooth $k$-compactification of $T$ and 
$\overline{X}=X\times_k\overline{k}$. 
Moreover, 
Hoshi, Kanai and Yamasaki \cite{HKY22}, \cite{HKY23} 
gives a necessary and sufficient condition 
for the Hasse norm principle for $K/k$, i.e. 
$\Sha(T)$ where $T=R^{(1)}_{K/k}(\bG_{m,K})$ is the norm one torus of $K/k$ 
with $[K:k]=n\leq 15$. 
\end{remark}

Similarly, by using GAP (\cite{GAP}), 
we can obtain all the transitive subgroup $G=nTm\leq S_n$ $(2\leq n\leq 19)$ 
with $M(G)\neq 0$ but $M(H)=0$ where $H\leq G$ with $[G:H]=n$ 
and $[G,G]\cap H=[H,H]$ 
as in Table $4$ (see Section \ref{GAPcomp} for GAP computations). 
Hence by Proposition \ref{prop4.1} (2) we get: 
\begin{theorem}\label{th7.2}
Let $k$ be a field, 
$K/k$ be a separable field extension with $[K:k]=n$ 
and $L/k$ be the Galois closure of $K/k$. 
Assume that $G={\rm Gal}(L/k)=nTm$ $(2\leq n\leq 19)$ 
is a transitive subgroup of $S_n$ which is not metacyclic 
with $M(G)\neq 0$ but $M(H)=0$ where $H=Gal(L/K)$ with $[G:H]=n$ 
and $[G,G]\cap H=[H,H]$. 
Then $G$, $H$ and $M(G)$ are given as in Table $4$ which satisfy  
$\Sha^2_\omega(G,J_{G/H})\leq H^2(G,J_{G/H})\simeq M(G)$. 
\end{theorem}

\newpage
\begin{center}
Table 3: $G=nTm\leq S_n$ $(2\leq n\leq 19)$: not metacyclic with $M(G)=0$ 
and $H\leq G$ with $[G:H]=n$, \\ 
$[G,G]\cap H=[H,H]$ which satisfy 
$\Sha^2_\omega(G,J_{G/H})=H^2(G,J_{G/H})=0$ as in 
Proposition \ref{prop4.1} (1) 
\vspace*{2mm}\\
\renewcommand{\arraystretch}{1.05}
{\small 
\begin{tabular}{lll} 
$nTm$ & $G$ & $H$\\\hline
$8T12$ & $\SL_2(\bF_3)\simeq Q_8\rtimes C_3$ & $C_3$\\
$8T23$ & $\GL_2(\bF_3)\simeq Q_8\rtimes S_3$ & $S_3\simeq C_3\rtimes C_2$\\
$8T25$ & $C_2^3\rtimes C_7$ & $C_7$\\
$8T36$ & $C_2^3\rtimes F_{21}$ & $F_{21}\simeq C_7\rtimes C_3$\\
$9T12$ & $(C_3^2\rtimes C_3)\rtimes C_2$ & $S_3\simeq C_3\rtimes C_2$\\
$9T15$ & $C_3^2\rtimes C_8$ & $C_8$\\
$9T19$ & $C_3^2\rtimes QD_8$ & $QD_8\simeq C_8\rtimes C_2$\\
$9T20$ & $(C_3^3\rtimes C_3)\rtimes C_2$ & $C_3\times S_3$\\
$9T26$ & $((C_3^2\rtimes Q_8)\rtimes C_3)\rtimes C_2$ & $\GL_2(\bF_3)\simeq Q_8\rtimes S_3$\\
$9T32$ & $\PSL_2(\bF_8)\rtimes C_3$ & $C_2^3\rtimes F_{21}$\\
$10T18$ & $C_5^2\rtimes C_8$ & $F_{20}\simeq C_5\rtimes C_4$\\
$12T46$ & $C_3^2\rtimes C_8$ & $S_3\simeq C_3\rtimes C_2$\\
$12T272$ & $M_{11}$ & $\PSL_2(\bF_{11})$\\
$14T11$ & $C_2^3\rtimes F_{21}$ & $A_4\simeq C_2^2\rtimes S_3$\\
$14T14$ & $C_7^2\rtimes C_6$ & $F_{21}\simeq C_7\rtimes C_3$\\
$14T18$ & $C_2\times (C_2^3\rtimes F_{21})$ & $C_2\times A_4$\\
$14T23$ & $C_7^2\rtimes C_{12}$ & $F_{42}\simeq C_7\rtimes C_6$\\
$15T13$ & $C_5^2\rtimes S_3$ & $D_5\simeq C_5\rtimes C_2$\\
$15T19$ & $C_5^2\rtimes C_{12}$ & $F_{20}\simeq C_5\rtimes C_4$\\
$15T32$ & $C_5\times (C_5^2\rtimes S_3)$ & $C_5\times D_{10}$\\
$15T38$ & $C_5^3\rtimes C_{12}$ & $C_5^2\rtimes C_4$\\
$15T41$ & $C_3^4\rtimes F_{20}$ & $C_3^3\rtimes C_4$\\
$15T56$ & $C_3\times (C_3^4\rtimes F_{20})$ & $C_3\times (C_3^3\rtimes C_4)$\\
$16T59$ & $C_2\times \SL_2(\bF_3)$ & $C_3$\\
$16T60$ & $((C_4\times C_2)\rtimes C_2)\rtimes C_3$ & $C_3$\\
$16T196$ & $C_2\times (C_2^3\rtimes C_7)$ & $C_7$\\
$16T439$ & $(C_4.C_4^2)\rtimes C_3$ & $A_4\simeq C_2^2\rtimes S_3$\\
$16T447$ & $C_2^4\rtimes C_{15}$ & $C_{15}$\\
$16T712$ & $C_2\times (C_2^3\rtimes F_{21})$ & $F_{21}\simeq C_7\rtimes C_3$\\
$16T728$ & $((C_4.C_4^2)\rtimes C_2)\rtimes C_3$ & $C_2\times A_4$\\
$16T732$ & $((C_4.C_4^2)\rtimes C_2)\rtimes C_3$ & $C_2\times A_4$\\
$16T773$ & $((C_4.C_4^2)\rtimes C_3)\rtimes C_2$ & $S_4\simeq C_2^2\rtimes S_3$\\
$16T777$ & $((C_2^4\rtimes C_5)\rtimes C_2)\rtimes C_3$ & $C_3\times D_5$\\
$16T1064$ & $((C_4.C_4^2)\rtimes C_3)\rtimes C_4$ & $A_4\rtimes C_4$\\
$16T1075$ & $((C_2\times (C_2^4\rtimes C_2))\rtimes C_2)\rtimes C_7$ & $C_2^3\rtimes C_7$\\
$16T1079$ & $(C_2^4\rtimes C_{15})\rtimes C_4$ & $C_{15}\rtimes C_4$\\
$16T1501$ & $((C_2\times (C_2^4\rtimes C_2))\rtimes C_2)\rtimes F_{21}$ & $C_2^3\rtimes F_{21}$\\
$16T1503$ & $((C_2\times (C_2^3\rtimes (C_2^2)))\rtimes C_2)\rtimes F_{21}$ & $C_2^3\rtimes F_{21}$\\
$16T1798$ & $(((C_2^6\rtimes C_7)\rtimes C_2)\rtimes C_7)\rtimes C_3$ & $C_2^3\rtimes (C_7^2\rtimes C_3)$\\
$17T8$ & $\PSL_2(\bF_{16})\rtimes C_4$ & $(C_2^4\rtimes C_{15})\rtimes C_4$\\
$18T28$ & $C_3^2\rtimes C_8$ & $C_4$\\
$18T49$ & $(C_3^2\rtimes C_3)\rtimes C_4$ & $S_3\simeq C_3\rtimes C_2$\\
$18T158$ & $(C_9^2\rtimes C_3)\rtimes C_2$ & $E_9(27)\simeq C_9\rtimes C_3$\\
$18T280$ & $C_3^4\rtimes C_{16}$ & $C_3^2\rtimes C_8$\\
$18T427$ & $C_2\times (\PSL_2(\bF_8)\rtimes C_3)$ & $C_2^3\rtimes F_{21}$\\
$18T937$ & $\PSL_2(\bF_8)^2\rtimes C_6$ & $C_2^3\rtimes (C_7\rtimes (\PSL_2(\bF_8)\rtimes C_3))$\\
\end{tabular}
}
\end{center}

\begin{center}
Table 4: $G=nTm\leq S_n$ $(2\leq n\leq 19)$: not metacyclic with 
$M(G)\neq 0$ 
and $H\leq G$ 
with $[G:H]=n$, 
$[G,G]\cap H=[H,H]$, $M(H)=0$ which satisfy 
$\Sha^2_\omega(G,J_{G/H})\leq H^2(G,J_{G/H})\simeq M(G)$ 
as in Proposition \ref{prop4.1} (2)
\vspace*{2mm}\\
\renewcommand{\arraystretch}{1.05}
{\small 
\begin{tabular}{llll} 
$nTm$ & $G$ & $H$ & $H^2(G,J_{G/H})\simeq M(G)$\\\hline
$4T2$ & $C_2^2$ & $\{1\}$ & $\bZ/2\bZ$\\
$4T3$ & $D_4\simeq C_4\rtimes C_2$ & $C_2$ & $\bZ/2\bZ$\\
$4T4$ & $A_4\simeq C_2^2\rtimes S_3$ & $C_3$ & $\bZ/2\bZ$\\
$4T5$ & $S_4$ & $S_3$ & $\bZ/2\bZ$\\
$6T9$ & $S_3^2$ & $S_3$ & $\bZ/2\bZ$\\
$6T10$ & $C_3^2\rtimes C_4$ & $S_3$ & $\bZ/3\bZ$\\
$8T3$ & $C_2^3$ & $\{1\}$ & $(\bZ/2\bZ)^{\oplus 3}$\\
$8T9$ & $C_2\times D_4$ & $C_2$ & $(\bZ/2\bZ)^{\oplus 3}$\\
$8T10$ & $(C_4\times C_2)\rtimes C_2$ & $C_2$ & $(\bZ/2\bZ)^{\oplus 2}$\\
$8T11$ & $(C_4\times C_2)\rtimes C_2$ & $C_2$ & $(\bZ/2\bZ)^{\oplus 2}$\\
$8T13$ & $C_2\times A_4$ & $C_3$ & $\bZ/2\bZ$\\
$8T17$ & $C_4^2\rtimes C_2$ & $C_4$ & $\bZ/2\bZ$\\
$8T19$ & $C_2^3\rtimes C_4$ & $C_4$ & $(\bZ/2\bZ)^{\oplus 2}$\\
$8T24$ & $C_2\times S_4$ & $S_3$ & $(\bZ/2\bZ)^{\oplus 2}$\\
$9T5$ & $C_3^2\rtimes C_2$ & $C_2$ & $\bZ/3\bZ$\\
$9T7$ & $C_3^2\rtimes C_3$ & $C_3$ & $(\bZ/3\bZ)^{\oplus 2}$\\
$9T9$ & $C_3^2\rtimes C_4$ & $C_4$ & $\bZ/3\bZ$\\
$9T11$ & $C_3^2\rtimes C_6$ & $C_6$ & $\bZ/3\bZ$\\
$9T13$ & $C_3^2\rtimes C_6$ & $S_3$ & $\bZ/3\bZ$\\
$9T14$ & $C_3^2\rtimes Q_8$ & $Q_8$ & $\bZ/3\bZ$\\
$9T23$ & $(C_3^2\rtimes Q_8)\rtimes C_3$ & $\SL_2(\bF_3)\simeq Q_8\rtimes C_3$ & $\bZ/3\bZ$\\
$10T9$ & $D_5^2$ & $D_5\simeq C_5\rtimes C_2$ & $\bZ/2\bZ$\\
$10T10$ & $C_5^2\rtimes C_4$ & $D_5\simeq C_5\rtimes C_2$ & $\bZ/5\bZ$\\
$10T17$ & $C_5^2\rtimes (C_4\times C_2)$ & $F_{20}\simeq C_5\rtimes C_4$ & $\bZ/2\bZ$\\
$10T35$ & $(A_6.C_2)\rtimes C_2$ & $C_3^2\rtimes QD_8$ & $\bZ/2\bZ$\\
$12T4$ & $A_4\simeq C_2^2\rtimes S_3$ & $\{1\}$ & $\bZ/2\bZ$\\
$12T6$ & $C_2\times A_4$ & $C_2$ & $\bZ/2\bZ$\\
$12T8$ & $S_4$ & $C_2$ & $\bZ/2\bZ$\\
$12T10$ & $C_2^2\times S_3$ & $C_2$ & $(\bZ/2\bZ)^{\oplus 3}$\\
$12T13$ & $(C_6\times C_2)\rtimes C_2$ & $C_2$ & $\bZ/2\bZ$\\
$12T15$ & $(C_6\times C_2)\rtimes C_2$ & $C_2$ & $\bZ/2\bZ$\\
$12T20$ & $C_3\times A_4$ & $C_3$ & $\bZ/6\bZ$\\
$12T27$ & $A_4\rtimes C_4$ & $C_4$ & $\bZ/2\bZ$\\
$12T35$ & $S_3^2\rtimes C_2$ & $S_3$ & $\bZ/2\bZ$\\
$12T37$ & $C_2\times S_3^2$ & $S_3$ & $(\bZ/2\bZ)^{\oplus 3}$\\
$12T38$ & $(C_6\times S_3)\rtimes C_2$ & $S_3$ & $\bZ/2\bZ$\\
$12T39$ & $C_3^2\rtimes (C_4\times C_2)$ & $S_3$ & $\bZ/2\bZ$\\
$12T40$ & $C_2\times (C_3^2\rtimes C_4)$ & $S_3$ & $\bZ/6\bZ$\\
$12T41$ & $C_2\times (C_3^2\rtimes C_4)$ & $S_3$ & $\bZ/6\bZ$\\
$12T42$ & $C_3\times ((C_6\times C_2)\rtimes C_2)$ & $C_6$ & $\bZ/2\bZ$\\
$12T43$ & $A_4\times S_3$ & $C_6$ & $\bZ/2\bZ$\\
$12T44$ & $(C_3\times A_4)\rtimes C_2$ & $S_3$ & $\bZ/6\bZ$\\
$12T45$ & $C_3\times S_4$ & $S_3$ & $\bZ/2\bZ$\\
$12T76$ & $C_2\times A_5$ & $D_5\simeq C_5\rtimes C_2$ & $\bZ/2\bZ$\\
$12T112$ & $((C_4^2\rtimes C_3)\rtimes C_2)\rtimes C_2$ & $QD_8\simeq C_8\rtimes C_2$ & $(\bZ/2\bZ)^{\oplus 2}$\\
$12T121$ & $C_3\times (S_3^2\rtimes C_2)$ & $C_3\times S_3$ & $\bZ/2\bZ$\\
$12T124$ & $A_5\rtimes C_4$ & $F_{20}\simeq C_5\rtimes C_4$ & $\bZ/2\bZ$\\
$12T175$ & $((C_3^3\rtimes C_2^2)\rtimes C_3)\rtimes C_2$ & $(C_3^2\rtimes C_3)\rtimes C_2$ & $\bZ/2\bZ$\\
$12T231$ & $C_3\times (((C_3^3\rtimes C_2^2)\rtimes C_3)\rtimes C_2)$ & $(C_3^3\rtimes C_3)\rtimes C_2$ & $\bZ/2\bZ$\\
$12T233$ & $(C_3\times ((C_3^3\rtimes C_2^2)\rtimes C_3))\rtimes C_2$ & $(C_3^3\rtimes C_3)\rtimes C_2$ & $\bZ/6\bZ$

\end{tabular}
}

\newpage
Table 4 (continued): $G=nTm\leq S_n$ $(2\leq n\leq 19)$: not metacyclic with 
$M(G)\neq 0$ 
and $H\leq G$ 
with $[G:H]=n$, 
$[G,G]\cap H=[H,H]$, $M(H)=0$ which satisfy 
$\Sha^2_\omega(G,J_{G/H})\leq H^2(G,J_{G/H})\simeq M(G)$ 
as in Proposition \ref{prop4.1} (2)
\vspace*{2mm}\\
\renewcommand{\arraystretch}{1.05}
{\small 
\begin{tabular}{llll}
$nTm$ & $G$ & $H$ & $H^2(G,J_{G/H})\simeq M(G)$\\\hline
$12T295$ & $M_{12}$ & $M_{11}$ & $\bZ/2\bZ$\\
$14T12$ & $C_7^2\rtimes C_4$ & $D_7\simeq C_7\rtimes C_2$ & $\bZ/7\bZ$\\
$14T13$ & $D_7^2$ & $D_7\simeq C_7\rtimes C_2$ & $\bZ/2\bZ$\\
$14T24$ & $C_7^2\rtimes (C_6\times C_2)$ & $F_{42}\simeq C_7\rtimes C_6$ & $\bZ/2\bZ$\\
$15T12$ & $C_5^2\rtimes C_6$ & $D_5\simeq C_5\rtimes C_2$ & $\bZ/5\bZ$\\
$15T17$ & $C_5^2\rtimes (C_3\rtimes C_4)$ & $F_{20}\simeq C_5\rtimes C_4$ & $\bZ/5\bZ$\\
$16T2$ & $C_4\times C_2^2$ & $\{1\}$ & $(\bZ/2\bZ)^{\oplus 3}$\\
$16T3$ & $C_2^4$ & $\{1\}$ & $(\bZ/2\bZ)^{\oplus 6}$\\
$16T7$ & $C_2\times Q_8$ & $\{1\}$ & $(\bZ/2\bZ)^{\oplus 2}$\\
$16T9$ & $C_2\times D_4$ & $\{1\}$ & $(\bZ/2\bZ)^{\oplus 3}$\\
$16T10$ & $(C_4\times C_2)\rtimes C_2$ & $\{1\}$ & $(\bZ/2\bZ)^{\oplus 2}$\\
$16T11$ & $(C_4\times C_2)\rtimes C_2$ & $\{1\}$ & $(\bZ/2\bZ)^{\oplus 2}$\\
$16T15$ & $C_2\times (C_8\rtimes C_2)$ & $C_2$ & $(\bZ/2\bZ)^{\oplus 2}$\\
$16T16$ & $(C_8\times C_2)\rtimes C_2$ & $C_2$ & $(\bZ/2\bZ)^{\oplus 2}$\\
$16T17$ & $C_4^2\rtimes C_2$ & $C_2$ & $\bZ/2\bZ\oplus \bZ/4\bZ$\\
$16T18$ & $C_2\times ((C_4\times C_2)\rtimes C_2)$ & $C_2$ & $(\bZ/2\bZ)^{\oplus 5}$\\
$16T19$ & $C_4\times D_4$ & $C_2$ & $(\bZ/2\bZ)^{\oplus 3}$\\
$16T20$ & $(C_2\times Q_8)\rtimes C_2$ & $C_2$ & $(\bZ/2\bZ)^{\oplus 5}$\\
$16T21$ & $C_2\times ((C_4\times C_2)\rtimes C_2)$ & $C_2$ & $(\bZ/2\bZ)^{\oplus 4}$\\
$16T23$ & $C_2^3\rtimes (C_2^2)$ & $C_2$ & $(\bZ/2\bZ)^{\oplus 5}$\\
$16T24$ & $(C_8\times C_2)\rtimes C_2$ & $C_2$ & $(\bZ/2\bZ)^{\oplus 2}$\\
$16T25$ & $C_2^2\times D_4$ & $C_2$ & $(\bZ2/\bZ)^{\oplus 6}$\\
$16T26$ & $(C_8\times C_2)\rtimes C_2$ & $C_2$ & $(\bZ/2\bZ)^{\oplus 2}$\\
$16T27$ & $C_4^2\rtimes C_2$ & $C_2$ & $\bZ/2\bZ$\\
$16T28$ & $C_4^2\rtimes C_2$ & $C_2$ & $\bZ/2\bZ$\\
$16T29$ & $C_2\times D_8$ & $C_2$ & $(\bZ/2\bZ)^{\oplus 3}$\\
$16T30$ & $C_4^2\rtimes C_2$ & $C_2$ & $\bZ/2\bZ\oplus \bZ/4\bZ$\\
$16T31$ & $(C_2\times Q_8)\rtimes C_2$ & $C_2$ & $(\bZ/2\bZ)^{\oplus 2}$\\
$16T32$ & $(C_2\times Q_8)\rtimes C_2$ & $C_2$ & $(\bZ/2\bZ)^{\oplus 2}$\\
$16T33$ & $C_2^3\rtimes C_4$ & $C_2$ & $(\bZ/2\bZ)^{\oplus 2}$\\
$16T34$ & $(C_4\times C_2^2)\rtimes C_2$ & $C_2$ & $(\bZ/2\bZ)^{\oplus 3}$\\
$16T35$ & $C_8\rtimes (C_2^2)$ & $C_2$ & $(\bZ/2\bZ)^{\oplus 2}$\\
$16T37$ & $(C_4\times C_2^2)\rtimes C_2$ & $C_2$ & $(\bZ/2\bZ)^{\oplus 2}$\\
$16T38$ & $C_8\rtimes (C_2^2)$ & $C_2$ & $(\bZ/2\bZ)^{\oplus 2}$\\
$16T39$ & $C_2^4\rtimes C_2$ & $C_2$ & $(\bZ/2\bZ)^{\oplus 4}$\\
$16T41$ & $(C_8\rtimes C_2)\rtimes C_2$ & $C_2$ & $\bZ/2\bZ$\\
$16T42$ & $C_4^2\rtimes C_2$ & $C_2$ & $\bZ/2\bZ$\\
$16T43$ & $(C_4\times C_2^2)\rtimes C_2$ & $C_2$ & $(\bZ/2\bZ)^{\oplus 3}$\\
$16T44$ & $(C_8\times C_2)\rtimes C_2$ & $C_2$ & $(\bZ/2\bZ)^{\oplus 2}$\\
$16T45$ & $C_8\rtimes (C_2^2)$ & $C_2$ & $(\bZ/2\bZ)^{\oplus 2}$\\
$16T46$ & $C_2^4\rtimes C_2$ & $C_2$ & $(\bZ/2\bZ)^{\oplus 4}$\\
$16T47$ & $(C_8\times C_2)\rtimes C_2$ & $C_2$ & $(\bZ/2\bZ)^{\oplus 2}$\\
$16T48$ & $C_2\times QD_8$ & $C_2$ & $(\bZ/2\bZ)^{\oplus 2}$\\
$16T50$ & $(C_2\times Q_8)\rtimes C_2$ & $C_2$ & $(\bZ/2\bZ)^{\oplus 2}$\\
$16T51$ & $C_4^2\rtimes C_2$ & $C_2$ & $(\bZ/2\bZ)^{\oplus 2}\oplus \bZ/4\bZ$\\
$16T52$ & $C_2^3\rtimes C_4$ & $C_2$ & $(\bZ/2\bZ)^{\oplus 2}$\\
$16T54$ & $(C_4\times C_2^2)\rtimes C_2$ & $C_2$ & $(\bZ/2\bZ)^{\oplus 2}$\\
$16T57$ & $C_4\times A_4$ & $C_3$ & $\bZ/2\bZ$\\
$16T58$ & $C_2^2\times A_4$ & $C_3$ & $(\bZ/2\bZ)^{\oplus 2}$\\
$16T63$ & $C_4^2\rtimes C_3$ & $C_3$ & $\bZ/4\bZ$
\end{tabular}
}

\newpage
Table 4 (continued): $G=nTm\leq S_n$ $(2\leq n\leq 19)$: not metacyclic with 
$M(G)\neq 0$ 
and $H\leq G$ 
with $[G:H]=n$, 
$[G,G]\cap H=[H,H]$, $M(H)=0$ which satisfy 
$\Sha^2_\omega(G,J_{G/H})\leq H^2(G,J_{G/H})\simeq M(G)$ 
as in Proposition \ref{prop4.1} (2)
\vspace*{2mm}\\
\renewcommand{\arraystretch}{1.05}
{\small 
\begin{tabular}{llll}
$nTm$ & $G$ & $H$ & $H^2(G,J_{G/H})\simeq M(G)$\\\hline
$16T64$ & $C_2^4\rtimes C_3$ & $C_3$ & $(\bZ/2\bZ)^{\oplus 4}$\\
$16T74$ & $C_4^2\rtimes C_4$ & $C_4$ & $\bZ/2\bZ\oplus \bZ/4\bZ$\\
$16T76$ & $C_2\times (C_2^3\rtimes C_4)$ & $C_4$ & $(\bZ/2\bZ)^{\oplus 4}$\\
$16T96$ & $(C_4\times C_2^2)\rtimes C_4$ & $C_4$ & $(\bZ/2\bZ)^{\oplus 3}$\\
$16T107$ & $(C_4^2\rtimes C_2)\rtimes C_2$ & $C_4$ & $(\bZ/2\bZ)^{\oplus 3}$\\
$16T110$ & $(C_8\rtimes C_2)\rtimes C_4$ & $C_4$ & $(\bZ/2\bZ)^{\oplus 2}$\\
$16T111$ & $C_2\times (C_4^2\rtimes C_2)$ & $C_4$ & $(\bZ/2\bZ)^{\oplus 3}$\\
$16T113$ & $(C_4^2\rtimes C_2)\rtimes C_2$ & $C_4$ & $(\bZ/2\bZ)^{\oplus 2}$\\
$16T114$ & $(C_4^2\rtimes C_2)\rtimes C_2$ & $C_4$ & $(\bZ/2\bZ)^{\oplus 2}$\\
$16T120$ & $(C_2^3\rtimes C_4)\rtimes C_2$ & $C_4$ & $(\bZ/2\bZ)^{\oplus 2}\oplus \bZ/4\bZ$\\
$16T121$ & $C_4^2\rtimes C_4$ & $C_4$ & $(\bZ/2\bZ)^{\oplus 2}$\\
$16T143$ & $(C_2^3\rtimes C_4)\rtimes C_2$ & $C_4$ & $\bZ/2\bZ\oplus \bZ/4\bZ$\\
$16T148$ & $((C_4\times C_2)\rtimes C_2)\rtimes C_4$ & $C_4$ & $\bZ/2\bZ$\\
$16T156$ & $(C_8\rtimes C_4)\rtimes C_2$ & $C_4$ & $\bZ/2\bZ$\\
$16T161$ & $(C_2\times Q_8)\rtimes C_4$ & $C_4$ & $\bZ/2\bZ$\\
$16T163$ & $((C_4\times C_2)\rtimes C_4)\rtimes C_2$ & $C_4$ & $(\bZ/2\bZ)^{\oplus 2}$\\
$16T166$ & $(C_2^3\rtimes C_4)\rtimes C_2$ & $C_4$ & $(\bZ/2\bZ)^{\oplus 2}$\\
$16T176$ & $C_4^2\rtimes C_4$ & $C_4$ & $(\bZ/2\bZ)^{\oplus 2}$\\
$16T178$ & $C_2^4\rtimes C_5$ & $C_5$ & $(\bZ/2\bZ)^{\oplus 2}$\\
$16T179$ & $D_4\times A_4$ & $C_6$ & $(\bZ/2\bZ)^{\oplus 2}$\\
$16T180$ & $((C_2\times Q_8)\rtimes C_2)\rtimes C_3$ & $C_6$ & $\bZ/2\bZ$\\
$16T181$ & $C_4\times S_4$ & $S_3$ & $(\bZ/2\bZ)^{\oplus 2}$\\
$16T182$ & $C_2^2\times S_4$ & $S_3$ & $(\bZ/2\bZ)^{\oplus 4}$\\
$16T183$ & $(C_2^4\rtimes C_2)\rtimes C_3$ & $C_6$ & $(\bZ/2\bZ)^{\oplus 2}$\\
$16T184$ & $(C_4^2\rtimes C_2)\rtimes C_3$ & $C_6$ & $\bZ/2\bZ$\\
$16T185$ & $(C_4^2\rtimes C_2)\rtimes C_3$ & $C_6$ & $\bZ/4\bZ$\\
$16T186$ & $(A_4\rtimes C_4)\rtimes C_2$ & $S_3$ & $(\bZ/2\bZ)^{\oplus 2}$\\
$16T187$ & $\GL_2(\bF_3)\rtimes C_2$ & $S_3$ & $\bZ/2\bZ$\\
$16T188$ & $C_2\times \GL_2(\bF_3)$ & $S_3$ & $\bZ/2\bZ$\\
$16T189$ & $(\SL_2(\bF_3).C2)\rtimes C_2$ & $S_3$ & $\bZ/2\bZ$\\
$16T190$ & $\GL_2(\bF_3)\rtimes C_2$ & $S_3$ & $\bZ/2\bZ$\\
$16T191$ & $(C_2\times S_4)\rtimes C_2$ & $S_3$ & $(\bZ/2\bZ)^{\oplus 2}$\\
$16T194$ & $(C_2^4\rtimes C_3)\rtimes C_2$ & $S_3$ & $(\bZ/2\bZ)^{\oplus 3}$\\
$16T195$ & $(C_4^2\rtimes C_3)\rtimes C_2$ & $S_3$ & $\bZ/2\bZ$\\
$16T260$ & $(C_4.(C_8\times C_2))\rtimes C_2$ & $C_8$ & $\bZ/2\bZ$\\
$16T289$ & $C_8^2\rtimes C_2$ & $C_8$ & $\bZ/2\bZ$\\
$16T332$ & $(((C_2\times Q_8)\rtimes C_2)\rtimes C_2)\rtimes C_2$ & $Q_8$ & $(\bZ/2\bZ)^{\oplus 3}$\\
$16T338$ & $((C_4\times Q_8)\rtimes C_2)\rtimes C_2$ & $Q_8$ & $(\bZ/2\bZ)^{\oplus 3}$\\
$16T351$ & $(C_4.(C_8\times C_2))\rtimes C_2$ & $Q_8$ & $\bZ/2\bZ$\\
$16T357$ & $((C_4^2\rtimes C_2)\rtimes C_2)\rtimes C_2$ & $Q_8$ & $(\bZ/2\bZ)^{\oplus 3}$\\
$16T370$ & $((C_4\times C_2).(C_4\times C_2))\rtimes C_2$ & $Q_8$ & $(\bZ/2\bZ)^{\oplus 3}$\\
$16T380$ & $((C_2^3\rtimes C_4)\rtimes C_2)\rtimes C_2$ & $Q_8$ & $(\bZ/2\bZ)^{\oplus 3}$\\
$16T381$ & $((C_2^3\rtimes C_4)\rtimes C_2)\rtimes C_2$ & $Q_8$ & $(\bZ/2\bZ)^{\oplus 2}\oplus \bZ/4\bZ$\\
$16T415$ & $(C_2^4\rtimes C_5)\rtimes C_2$ & $D_5\simeq C_5\rtimes C_2$ & $(\bZ/2\bZ)^{\oplus 2}$\\
$16T430$ & $(C_4^2\rtimes C_3)\rtimes C_4$ & $Q_{12}\simeq C_3\rtimes C_4$ & $\bZ/4\bZ$\\
$16T433$ & $(C_2^4\rtimes C_3)\rtimes C_4$ & $Q_{12}\simeq C_3\rtimes C_4$ & $\bZ/2\bZ$\\
$16T504$ & $((C_4\times (C_8\rtimes C_2))\rtimes C_2)\rtimes C_2$ & $M_{16}\simeq C_8\rtimes C_2$ & $(\bZ/2\bZ)^{\oplus 2}\oplus \bZ/4\bZ$\\
$16T548$ & $(C_4^2\rtimes C_4)\rtimes C_4$ & $M_{16}\simeq C_8\rtimes C_2$ & $\bZ/2\bZ\oplus \bZ/4\bZ$\\
$16T568$ & $(C_8^2\rtimes C_2)\rtimes C_2$ & $M_{16}\simeq C_8\rtimes C_2$ & $(\bZ/2\bZ)^{\oplus 2}$\\
$16T576$ & $((C_8\rtimes C_4)\rtimes C_2)\rtimes C_4$ & $M_{16}\simeq C_8\rtimes C_2$ & $(\bZ/2\bZ)^{\oplus 2}$
\end{tabular}
}

\newpage
Table 4 (continued): $G=nTm\leq S_n$ $(2\leq n\leq 19)$: not metacyclic with 
$M(G)\neq 0$ 
and $H\leq G$ 
with $[G:H]=n$, 
$[G,G]\cap H=[H,H]$, $M(H)=0$ which satisfy 
$\Sha^2_\omega(G,J_{G/H})\leq H^2(G,J_{G/H})\simeq M(G)$ 
as in Proposition \ref{prop4.1} (2)
\vspace*{2mm}\\
\renewcommand{\arraystretch}{1.05}
{\small 
\begin{tabular}{llll}
$nTm$ & $G$ & $H$ & $H^2(G,J_{G/H})\simeq M(G)$\\\hline
$16T580$ & $(C_{16}\rtimes C_4)\rtimes C_4$ & $M_{16}\simeq C_8\rtimes C_2$ & $\bZ/2\bZ$\\
$16T601$ & $(((C_8\times C_2)\rtimes C_2)\rtimes C_2)\rtimes C_4$ & $M_{16}\simeq C_8\rtimes C_2$ & $\bZ/2\bZ\oplus \bZ/4\bZ$\\
$16T605$ & $(((C_4\rtimes C_8)\rtimes C_2)\rtimes C_2)\rtimes C_2$ & $M_{16}\simeq C_8\rtimes C_2$ & $(\bZ/2\bZ)^{\oplus 2}\oplus \bZ/4\bZ$\\
$16T619$ & $((C_2^2.(C_4\times C_2))\rtimes C_2)\rtimes C_4$ & $M_{16}\simeq C_8\rtimes C_2$ & $(\bZ/2\bZ)^{\oplus 2}$\\
$16T627$ & $(((C_8\rtimes C_4)\rtimes C_2)\rtimes C_2)\rtimes C_2$ & $M_{16}\simeq C_8\rtimes C_2$ & $(\bZ/2\bZ)^{\oplus 3}$\\
$16T655$ & $(((C_4\rtimes C_8)\rtimes C_2)\rtimes C_2)\rtimes C_2$ & $M_{16}\simeq C_8\rtimes C_2$ & $(\bZ/2\bZ)^{\oplus 3}$\\
$16T668$ & $(C_8\rtimes C_8)\rtimes C_4$ & $QD_8\simeq C_8\rtimes C_2$ & $\bZ/8\bZ$\\
$16T669$ & $(((C_2^3\rtimes C_4)\rtimes C_2)\rtimes C_2)\rtimes C_2$ & $QD_8\simeq C_8\rtimes C_2$ & $(\bZ/2\bZ)^{\oplus 3}$\\
$16T672$ & $(((C_2^2.(C_4\times C_2))\rtimes C_2)\rtimes C_2)\rtimes C_2$ & $QD_8\simeq C_8\rtimes C_2$ & $(\bZ/2\bZ)^{\oplus 3}$\\
$16T674$ & $(((C_4\times C_2).(C_4\times C_2))\rtimes C_2)\rtimes C_2$ & $QD_8\simeq C_8\rtimes C_2$ & $(\bZ/2\bZ)^{\oplus 3}$\\
$16T681$ & $(((C_2^2).(C_2^3))\rtimes C_4)\rtimes C_2$ & $QD_8\simeq C_8\rtimes C_2$ & $\bZ/2\bZ$\\
$16T683$ & $((C_4\rtimes C_8)\rtimes C_2)\rtimes C_4$ & $QD_8\simeq C_8\rtimes C_2$ & $\bZ/2\bZ$\\
$16T687$ & $((C_8\rtimes Q_8)\rtimes C_2)\rtimes C_2$ & $QD_8\simeq C_8\rtimes C_2$ & $(\bZ/2\bZ)^{\oplus 2}$\\
$16T694$ & $C_8^2\rtimes C_4$ & $QD_8\simeq C_8\rtimes C_2$ & $\bZ/4\bZ$\\
$16T704$ & $((C_4\times C_2).(C_4\times C_2))\rtimes C_4$ & $QD_8\simeq C_8\rtimes C_2$ & $\bZ/2\bZ$\\
$16T706$ & $(((C_2\times QD_8)\rtimes C_2)\rtimes C_2)\rtimes C_2$ & $QD_8\simeq C_8\rtimes C_2$ & $(\bZ/2\bZ)^{\oplus 2}$\\
$16T708$ & $((C_2^4\rtimes C_3)\rtimes C_2)\rtimes C_3$ & $C_3\times S_3$ & $\bZ/2\bZ$\\
$16T709$ & $S_4\times A_4$ & $C_3\times S_3$ & $(\bZ/2\bZ)^{\oplus 2}$\\
$16T711$ & $(C_2^4\rtimes C_5)\rtimes C_4$ & $F_{20}\simeq C_5\rtimes C_4$ & $\bZ/2\bZ$\\
$16T730$ & $(Q_8^2\rtimes C_2)\rtimes C_3$ & $\SL_2(\bF_3)\simeq Q_8\rtimes C_3$ & $\bZ/2\bZ$\\
$16T734$ & $(((C_4^2\rtimes C_2)\rtimes C_2)\rtimes C_2)\rtimes C_3$ & $\SL_2(\bF_3)\simeq Q_8\rtimes C_3$ & $\bZ/2\bZ$\\
$16T735$ & $((((C_2\times Q_8)\rtimes C_2)\rtimes C_2)\rtimes C_2)\rtimes C_3$ & $\SL_2(\bF_3)\simeq Q_8\rtimes C_3$ & $\bZ/2\bZ$\\
$16T1062$ & $((((C_4^2\rtimes C_2)\rtimes C_2)\rtimes C_3)\rtimes C_2)\rtimes C_2$ & $\GL_2(\bF_3)\simeq Q_8\rtimes S_3$ & $(\bZ/2\bZ)^{\oplus 2}$\\
$16T1067$ & $((((C_4\times C_2).(C_4\times C_2))\rtimes C_2)\rtimes C_3)\rtimes C_2$ & $\GL_2(\bF_3)\simeq Q_8\rtimes S_3$ & $(\bZ/2\bZ)^{\oplus 2}$\\
$16T1076$ & $(C_2^4.C_2^3)\rtimes C_7$ & $C_2^3\rtimes C_7$ & $\bZ/2\bZ$\\
$16T1077$ & $((C_2\times (C_2^3\rtimes C_2^2))\rtimes C_2)\rtimes C_7$ & $C_2^3\rtimes C_7$ & $(\bZ/2\bZ)^{\oplus 2}$\\
$16T1502$ & $((C_2^6\rtimes C_7)\rtimes C_2)\rtimes C_3$ & $C_2^3\rtimes F_{21}$ & $\bZ/2\bZ$\\
$18T4$ & $C_3^2\rtimes C_2$ & $\{1\}$ & $\bZ/3\bZ$\\
$18T9$ & $S_3^2$ & $C_2$ & $\bZ/2\bZ$\\
$18T10$ & $C_3^2\rtimes C_4$ & $C_2$ & $\bZ/3\bZ$\\
$18T11$ & $S_3^2$ & $C_2$ & $\bZ/2\bZ$\\
$18T12$ & $C_2\times (C_3^2\rtimes C_2)$ & $C_2$ & $\bZ/6\bZ$\\
$18T15$ & $C_2\times (C_3^2\rtimes C_3)$ & $C_3$ & $(\bZ/3\bZ)^{\oplus 2}$\\
$18T17$ & $C_3^2\times S_3$ & $C_3$ & $\bZ/3\bZ$\\
$18T21$ & $C_3^2\rtimes C_6$ & $C_3$ & $\bZ/3\bZ$\\
$18T22$ & $C_3^2\rtimes C_6$ & $C_3$ & $\bZ/3\bZ$\\
$18T23$ & $C_3\times (C_3^2\rtimes C_2)$ & $C_3$ & $\bZ/3\bZ$\\
$18T27$ & $C_2\times (C_3^2\rtimes C_4)$ & $C_4$ & $\bZ/6\bZ$\\
$18T41$ & $C_2\times (C_3^2\rtimes C_6)$ & $C_6$ & $\bZ/6\bZ$\\
$18T42$ & $C_2\times (C_3^2\rtimes C_6)$ & $S_3$ & $\bZ/6\bZ$\\
$18T43$ & $C_3\times S_3^2$ & $S_3$ & $\bZ/2\bZ$\\
$18T44$ & $C_3\times (C_3^2\rtimes C_4)$ & $S_3$ & $\bZ/3\bZ$\\
$18T46$ & $C_3\times S_3^2$ & $C_6$ & $\bZ/2\bZ$\\
$18T50$ & $D_9\times S_3$ & $S_3$ & $\bZ/2\bZ$\\
$18T51$ & $(C_3^2\rtimes C_3)\rtimes C_2^2$ & $S_3$ & $\bZ/2\bZ$\\
$18T52$ & $C_2\times ((C_3^2\rtimes C_3)\rtimes C_2)$ & $S_3$ & $\bZ/2\bZ$\\
$18T53$ & $C_3^3\rtimes C_2^2$ & $S_3$ & $\bZ/2\bZ$\\
$18T54$ & $C_3^3\rtimes C_4$ & $S_3$ & $\bZ/3\bZ$\\
$18T56$ & $(C_3^2\rtimes C_3)\rtimes C_2^2$ & $S_3$ & $\bZ/2\bZ$\\
$18T57$ & $(C_3^2\rtimes C_3)\rtimes C_2^2$ & $S_3$ & $\bZ/2\bZ$
\end{tabular}
}

\newpage
Table 4 (continued): $G=nTm\leq S_n$ $(2\leq n\leq 19)$: not metacyclic with 
$M(G)\neq 0$ 
and $H\leq G$ 
with $[G:H]=n$, 
$[G,G]\cap H=[H,H]$, $M(H)=0$ which satisfy 
$\Sha^2_\omega(G,J_{G/H})\leq H^2(G,J_{G/H})\simeq M(G)$ 
as in Proposition \ref{prop4.1} (2)
\vspace*{2mm}\\
\renewcommand{\arraystretch}{1.05}
{\small 
\begin{tabular}{llll}
$nTm$ & $G$ & $H$ & $H^2(G,J_{G/H})\simeq M(G)$\\\hline

$18T58$ & $(C_3^2\rtimes C_2)\times S_3$ & $S_3$ & $\bZ/6\bZ$\\
$18T59$ & $C_2\times (C_3^2\rtimes C_8)$ & $C_8$ & $\bZ/2\bZ$\\
$18T64$ & $C_2\times (C_3^2\rtimes Q_8)$ & $Q_8$ & $\bZ/2\bZ\oplus \bZ/6\bZ$\\
$18T110$ & $C_2\times (C_3^2\rtimes QD_8)$ & $QD_8\simeq C_8\rtimes C_2$ & $(\bZ/2\bZ)^{\oplus 2}$\\
$18T118$ & $C_3\times ((C_3^2\rtimes C_3)\rtimes C_2^2)$ & $C_3\times S_3$ & $\bZ/2\bZ$\\
$18T119$ & $C_2\times ((C_3^3\rtimes C_3)\rtimes C_2)$ & $C_3\times S_3$ & $\bZ/2\bZ$\\
$18T120$ & $C_3\times (C_3^3\rtimes C_2^2)$ & $C_3\times S_3$ & $\bZ/2\bZ$\\
$18T121$ & $(C_3^2\rtimes C_6)\times S_3$ & $C_3\times S_3$ & $\bZ/6\bZ$\\
$18T122$ & $(C_9\rtimes C_6)\times S_3$ & $C_3\times S_3$ & $\bZ/2\bZ$\\
$18T123$ & $C_3\times (C_3^3\rtimes C_4)$ & $C_3\times S_3$ & $\bZ/3\bZ$\\
$18T126$ & $C_3\times ((C_3^2\rtimes C_3)\rtimes C_2^2)$ & $C_3\times S_3$ & $\bZ/2\bZ$\\
$18T130$ & $C_9^2\rtimes C_4$ & $D_9\simeq C_9\rtimes C_2$ & $\bZ/9\bZ$\\
$18T140$ & $D_{18}^2$ & $D_{18}\simeq C_{18}\rtimes C_2$ & $\bZ/2\bZ$\\
$18T151$ & $C_2\times ((C_3^2\rtimes Q_8)\rtimes C_3)$ & $\SL_2(\bF_3)\simeq Q_8\rtimes C_3$ & $\bZ/3\bZ$\\
$18T229$ & $C_2\times (((C_3^2\rtimes Q_8)\rtimes C_3)\rtimes C_2)$ & $\GL_2(\bF_3)\simeq Q_8\rtimes S_3$ & $\bZ/2\bZ$\\
$18T233$ & $((C_9\rtimes C_9)\rtimes C_3)\rtimes C_2^2$ & $C_9\rtimes C_6$ & $\bZ/2\bZ$\\
$18T235$ & $((C_9\rtimes C_9)\rtimes C_3)\rtimes C_4$ & $C_9\rtimes C_6$ & $\bZ/3\bZ$\\
$18T244$ & $((C_3\times (C_3^2\rtimes C_3))\rtimes C_3)\rtimes C_2^2$ & $(C_3^2\rtimes C_3)\rtimes C_2$ & $\bZ/2\bZ$\\
$18T281$ & $C_3^4\rtimes (C_8\times C_2)$ & $C_3^2\rtimes C_8$ & $\bZ/2\bZ$\\
$18T383$ & $C_3^4\rtimes ((C_8\times C_2)\rtimes C_2)$ & $C_3^2\rtimes QD_8$ & $(\bZ/2\bZ)^{\oplus 2}$\\
$18T385$ & $C_3^4\rtimes (C_8\rtimes (C_2^2))$ & $C_3^2\rtimes QD_8$ & $(\bZ/2\bZ)^{\oplus 2}$\\
$18T386$ & $C_3^4\rtimes ((C_2\times Q_8)\rtimes C_2)$ & $C_3^2\rtimes QD_8$ & $\bZ/2\bZ\oplus \bZ/6\bZ$\\
$18T393$ & $C_3^4\rtimes (C_2\times QD_8)$ & $C_3^2\rtimes QD_8$ & $(\bZ/2\bZ)^{\oplus 2}$\\
$18T524$ & $(((C_3^4\rtimes Q_8)\rtimes C_3)\rtimes C_2)\rtimes C_2$ & $((C_3^2\rtimes Q_8)\rtimes C_3)\rtimes C_2$ & $\bZ/2\bZ$\\
$18T526$ & $(((C_3^4\rtimes Q_8)\rtimes C_3)\rtimes C_2)\rtimes C_2$ & $((C_3^2\rtimes Q_8)\rtimes C_3)\rtimes C_2$ & $\bZ/2\bZ$\\
$18T527$ & $((C_3^4\rtimes (C_2\times Q_8))\rtimes C_3)\rtimes C_2$ & $((C_3^2\rtimes Q_8)\rtimes C_3)\rtimes C_2$ & $\bZ/6\bZ$\\
$18T528$ & $((C_3^4\rtimes ((C_4\times C_2)\rtimes C_2))\rtimes C_3)\rtimes C_2$ & $((C_3^2\rtimes Q_8)\rtimes C_3)\rtimes C_2$	& $\bZ/2\bZ$
\end{tabular}
}
\end{center}


\section{Appendix: Examples of groups $G$ with $M(G)=0$: Galois cases with $H=\{1\}$}\label{S8}

Let $k$ be a field and $K/k$ be 
a finite Galois extension with Galois group $G={\rm Gal}(K/k)$. 
When $k$ be a global field, 
it follows from Tate's theorem (Theorem \ref{thTate}) that 
if $M(G)=0$, then the Hasse norm principle holds for $K/k$. 
It also follows from Voskresenskii's theorem (Theorem \ref{thVos70}) 
that if $M(G)=0$, then 
$\Sha^2_\omega(G,J_G)=H^2(G,J_G)\simeq M(G)=0$. 
We note that $\Sha^2_\omega(G,J_G)=0$ implies that 
$A(T)=0$, i.e. 
$T$ has the weak approximation property, 
and $\Sha(T)=0$, i.e. the Hasse norm principle holds for $K/k$ 
$($that is, Hasse principle holds for all torsors $E$ under $T$$)$ 
where $T=R^{(1)}_{K/k}(\bG_{m,K})$ is the norm one torus of $K/k$ 
$($see Section \ref{S3} and Ono's theorem $($Theorem \ref{thOno}$))$. 

\begin{example}[$G=\SL_n(\bF_q)$: the special linear group of degree $n$ over $\bF_q$ with $M(G)=0$; $G=\PSL_n(\bF_q)$: the projective special linear group of degree $n$ over $\bF_q$ with $M(G)=0$; $G=\GL_n(\bF_q)$: the general linear group of degree $n$ over $\bF_q$ with $M(G)=0$; $G=\PGL_n(\bF_q)$: the projective general linear group of degree $n$ over $\bF_q$ with $M(G)=0$]\label{exapp1}
Let $n\geq 2$ be an integer and $q=p^r\geq 2$ be prime power. 
Assume that $(n,q)\neq (2,4),(2,9),(3,2),(3,4),(4,2)$. 
Then\\
{\rm (1)}
If $G=\SL_n(\bF_q)$, 
then $\Sha^2_\omega(G,J_G)=H^2(G,J_G)\simeq M(G)=0$;\\
{\rm (2)} 
If $G=\PSL_n(\bF_q)$, 
then $\Sha^2_\omega(G,J_G)=H^2(G,J_G)\simeq M(G)\simeq \bZ/u\bZ$ 
where $u=\gcd(n,q-1)$. 
In particular, 
if $\gcd(n,q-1)=1$, then $\PSL_n(\bF_q)=\SL_n(\bF_q)$ and 
$\Sha^2_\omega(G,J_G)=H^2(G,J_G)\simeq M(G)=0$;\\
{\rm (3)}
If $G=\GL_n(\bF_q)$, 
then $\Sha^2_\omega(G,J_G)=H^2(G,J_G)\simeq M(G)=0$;\\
{\rm (4)}
If $G=\PGL_n(\bF_q)$, 
then $\Sha^2_\omega(G,J_G)=H^2(G,J_G)\simeq M(G)\simeq \bZ/u\bZ$ 
where $u=\gcd(n,q-1)$. In particular, 
if $\gcd(n,q-1)=1$, then $\PGL_n(\bF_q)=\PSL_n(\bF_q)=\SL_n(\bF_q)$ and 
$\Sha^2_\omega(G,J_G)=H^2(G,J_G)\simeq M(G)=0$. 
Note that if $q=2$, then 
$\SL_n(\bF_2)=\PSL_n(\bF_2)=\GL_n(\bF_2)=\PGL_n(\bF_2)$. 

When $k$ is a global field, $\Sha^2_\omega(G,J_G)=0$ implies that 
$A(T)=0$, i.e. 
$T$ has the weak approximation property, 
and $\Sha(T)=0$, i.e. the Hasse norm principle holds for $K/k$ 
$($that is, Hasse principle holds for all torsors $E$ under $T$$)$ 
where $T=R^{(1)}_{K/k}(\bG_{m,K})$ is the norm one torus of $K/k$  
$($see Section \ref{S3} and Ono's theorem $($Theorem \ref{thOno}$))$. 

For (1), (2), we have $M(\PSL_n(\bF_q))\simeq \bZ/u\bZ$ 
(see Gorenstein \cite[Table 4.1, page 302]{Gor82}, 
Karpilovsky \cite[Table 8.5, page 283]{Kar87} 
with Lie type $A_n(q)=\PSL_{n+1}(\bF_q)$, 
see also \cite[Section 2.1, page 70]{Gor82}) 
and a Schur cover of $\PSL_n(\bF_q)$ 
(see Karpilovsky \cite[Theorem 7.1.1, page 246]{Kar87}, 
\cite[Theorem 3.2, page 735]{Kar93}): 
\begin{align*}
1\to Z(\SL_n(\bF_q))\simeq M(\PSL_n(\bF_q))\simeq 
\bZ/u\bZ\to \SL_n(\bF_q)\to \PSL_n(\bF_q)\to 1.
\end{align*}

For (3), by $\SL_n(\bF_q)$ $\lhd$ $\GL_n(\bF_q)$, 
$\SL_n(\bF_q)^{ab}=1$ and 
$\GL_n(\bF_q)/\SL_n(\bF_q)$ $\simeq$ $\bF_q^\times\simeq C_{q-1}$; cyclic, 
we have 
$M(\GL_n(\bF_q))$ $\leq$ $M(\SL_n(\bF_q))=0$ 
(see Sonn \cite[Lemma 6, Proposition 8]{Son94}). 

For (4), by $\PSL_n(\bF_q)$ $\lhd$ $\PGL_n(\bF_q)$, 
$\PSL_n(\bF_q)^{ab}=1$ and 
$\PGL_n(\bF_q)$ $/$ $\PSL_n(\bF_q)$ $\simeq$ $\bF_q^\times/(\bF_q^\times)^n$; cyclic, 
we have 
$[\GL_n(\bF_q),\GL_n(\bF_q)]$ $\cap$ $Z(\GL_n(\bF_q))$ $=$ $Z(\SL_n(\bF_q))\simeq C_u\leq M(\PGL_n(\bF_q))\leq M(\PSL_n(\bF_q)\simeq \bZ/u\bZ$  
(see Sonn \cite[Lemma 6, Lemma 7, Proposition 8]{Son94}). 

Indeed, we get a commutative diagram with exact rows and columns: 
\begin{align*}
\xymatrix@=20pt{
 & 1 \ar[d] & 1 \ar[d] & 1 \ar[d] & \\
1 \ar[r] & Z(\SL_n(\bF_q))\simeq C_u \ar[r] \ar[d] & \SL_n(\bF_q) \ar[r] \ar[d] & \PSL_n(\bF_q) \ar[r] \ar[d] & 1\\
1 \ar[r] & Z(\GL_n(\bF_q))\simeq \bF_q^\times \ar[r] \ar[d]_{x\mapsto x^n} & \GL_n(\bF_q) \ar[r] \ar[d]_{\rm det} & \PGL_n(\bF_q) \ar[r] \ar[d]_{\overline{\rm det}} & 1\\
1 \ar[r] & (\bF_q^\times)^n=(\bF_q^\times)^u \ar[r] \ar[d] & \bF_q^\times \ar[r] \ar[d] & \bF_q^\times/(\bF_q^\times)^n \ar[r] \ar[d] & 1\\
 & 1 & 1 & 1 & 
}
\end{align*}
where $u=\gcd(n,q-1)$ 
with $Z(\SL_n(\bF_q))\simeq M(\PSL_n(\bF_q))\simeq \bZ/u\bZ$ and 
$\bF_q^\times/(\bF_q^\times)^n\simeq C_u$. 

For the $5$ exceptional cases, 
see Gorenstein \cite[Table 4.1, page 302]{Gor82} 
with Lie type $A_n(q)=\PSL_{n+1}(\bF_q)$, 
Karpilovsky \cite[Theorem 7.1.1, page 246]{Kar87}, 
see also Remark \ref{rem8.2}. 
\begin{remark}\label{rem8.2}
(1) We see that $G=\PSL_n(\bF_q)$ $(n\geq 2)$ 
is simple except for the $2$ cases $(n,q)=(2,2),(2,3)$:\\ 
(i) $G=\PSL_2(\bF_2)= \SL_2(\bF_2)\simeq S_3$ with $M(G)=0$;\\
(ii) $G=\PSL_2(\bF_3)\simeq A_4$ with $M(G)\simeq \bZ/2\bZ$.\\
(2) The $5$ exceptional cases with $(n,q)=(2,4),(2,9),(3,2),(3,4),(4,2)$ 
satisfy that\\
(i) $G=\SL_2(\bF_4)= \PSL_2(\bF_4)= \PGL_2(\bF_4)
\simeq A_5$ with $M(G)\simeq \bZ/2\bZ$, 
$G=\GL_2(\bF_4)$ with $M(G)\simeq \bZ/2\bZ$;\\
(ii) $G=\SL_2(\bF_9)$ with $M(G)\simeq \bZ/3\bZ$, 
$G=\PSL_2(\bF_9)\simeq A_6$ with $M(G)\simeq \bZ/2\bZ\oplus \bZ/3\bZ$, 
$G=\GL_2(\bF_9)$ with $M(G)=0$, 
$G=\PGL_2(\bF_9)$ with $M(G)\simeq \bZ/2\bZ$;\\
(iii) $G=\SL_3(\bF_2)= \PSL_3(\bF_2)=  
\GL_3(\bF_2)= \PGL_3(\bF_2)\simeq \PSL_2(\bF_7)$ 
with $M(G)\simeq \bZ/2\bZ$;\\
(iv) $G=\SL_3(\bF_4)$ with $M(G)\simeq (\bZ/4\bZ)^{\oplus 2}$, 
$G=\PSL_3(\bF_4)\simeq M_{21}$ with $M(G)\simeq (\bZ/4\bZ)^{\oplus 2}\oplus \bZ/3\bZ$, 
$G=\GL_3(\bF_4)$ with $M(G)=0$, 
$G=\PGL_3(\bF_4)$ with $M(G)\simeq \bZ/3\bZ$;\\
(v) $G=\SL_4(\bF_2)= \PSL_4(\bF_2)= 
\GL_4(\bF_2)= \PGL_4(\bF_2)\simeq A_8$ 
with $M(G)\simeq \bZ/2\bZ$\\
(see Gorenstein \cite[Theorem 2.13, Table 4.1]{Gor82}, 
Karpilovsky \cite[Section 8.4, Table 8.5]{Kar87} with Lie type $A_n(q)=\PSL_{n+1}(\bF_q)$ and GAP computations as in Section \ref{GAPcomp2}, see also Sonn \cite[Proposition 8, page 404]{Son94}). 
\end{remark}
\end{example}
\begin{example}[$G=\SU_n(\bF_q)$: the special unitary group of degree $n$ over $\bF_q$ with $M(G)=0$; $G=\PSU_n(\bF_q)$: the projective special unitary group of degree $n$ over $\bF_q$ with $M(G)=0$; $G={\rm GU}_n(\bF_q)$: the general unitary group of degree $n$ over $\bF_q$ with $M(G)=0$; $G=\PGU_n(\bF_q)$: the projective general unitary group of degree $n$ over $\bF_q$ with $M(G)=0$]\label{exapp3}
Let $n\geq 3$ be an integer and $q=p^r\geq 2$ be prime power. 
Assume that $(n,q)\neq (4,2),(4,3),(6,2)$. 
Then\\
{\rm (1)} 
If $G=\SU_n(\bF_q)$, 
then $\Sha^2_\omega(G,J_G)=H^2(G,J_G)\simeq M(G)=0$;\\ 
{\rm (2)} 
If $G=\PSU_n(\bF_q)$, 
then $\Sha^2_\omega(G,J_G)=H^2(G,J_G)\simeq M(G)\simeq \bZ/v\bZ$ 
where $v=\gcd(n,q+1)$. 
In particular, 
if $\gcd(n,q+1)=1$, then $\PSU_n(\bF_q)=\SU_n(\bF_q)$ and 
$\Sha^2_\omega(G,J_G)=H^2(G,J_G)\simeq M(G)=0$;\\
{\rm (3)}
If $G={\rm GU}_n(\bF_q)$, 
then $\Sha^2_\omega(G,J_G)=H^2(G,J_G)\simeq M(G)=0$;\\
{\rm (4)}
If $G=\PGU_n(\bF_q)$, 
then $\Sha^2_\omega(G,J_G)=H^2(G,J_G)\simeq M(G)\simeq \bZ/v\bZ$ 
where $v=\gcd(n,q+1)$. 
In particular, 
if $\gcd(n,q+1)=1$, then $\PGU_n(\bF_q)=\PSU_n(\bF_q)=\SU_n(\bF_q)$ and 
$\Sha^2_\omega(G,J_G)=H^2(G,J_G)\simeq M(G)=0$. 

When $k$ is a global field, $\Sha^2_\omega(G,J_G)=0$ implies that 
$A(T)=0$, i.e. 
$T$ has the weak approximation property, 
and $\Sha(T)=0$, i.e. the Hasse norm principle holds for $K/k$ 
$($that is, Hasse principle holds for all torsors $E$ under $T$$)$ 
where $T=R^{(1)}_{K/k}(\bG_{m,K})$ is the norm one torus of $K/k$ 
$($see Section \ref{S3} and Ono's theorem $($Theorem \ref{thOno}$))$. 

For (1), (2), we have $M(\PSU_n(\bF_q))\simeq \bZ/v\bZ$ 
(see Gorenstein \cite[Table 4.1, page 302]{Gor82}, 
Karpilovsky \cite[Table 8.5, page 283]{Kar87} 
with Lie type ${}^2\!A_{n}(q)=\PSU_{n+1}(\bF_q)$, 
see also \cite[Section 2.1, page 73]{Gor82}) 
and a Schur cover of $\PSU_n(\bF_q)$: 
\begin{align*}
1\to Z(\SU_n(\bF_q))\simeq M(\PSU_n(\bF_q))\simeq \bZ/v\bZ\to 
\SU_n(\bF_q)\to \PSU_n(\bF_q)\to 1.
\end{align*}


For (3), by $\SU_n(\bF_q)$ $\lhd$ $\GU_n(\bF_q)$, 
$\SU_n(\bF_q)^{ab}=1$ and 
$\GU_n(\bF_q)/\SU_n(\bF_q)\simeq (\bF_{q^2}^\times)^{(1)}\simeq C_{q+1}$; 
cyclic, we get 
$M(\GU_n(\bF_q))$ $\leq$ $M(\SU_n(\bF_q))=0$ 
(see Sonn \cite[Lemma 6, Proposition 8]{Son94}). 

For (4), by $\PSU_n(\bF_q)$ $\lhd$ $\PGU_n(\bF_q)$, 
$\PSU_n(\bF_q)^{ab}=1$ and 
$\PGU_n(\bF_q)/\PSU_n(\bF_q)$ $\simeq$ 
$(\bF_{q^2}^\times)^{(1)}/((\bF_{q^2}^\times)^{(1)})^n\simeq C_v$; cyclic, 
we get 
$[\GU_n(\bF_q),\GU_n(\bF_q)]$ $\cap$ $Z(\GU_n(\bF_q))=Z(\SU_n(\bF_q))\simeq C_v
\leq M(\PGU_n(\bF_q))\leq M(\PSU_n(\bF_q)\simeq \bZ/v\bZ$  
(see Sonn \cite[Lemma 6, Lemma 7, Proposition 8]{Son94}). 

Indeed, we get a commutative diagram with exact rows and columns: 
\begin{align*}
\xymatrix@=20pt{
 & 1 \ar[d] & 1 \ar[d] & 1 \ar[d] & \\
1 \ar[r] & Z(\SU_n(\bF_q))\simeq C_v \ar[r] \ar[d] & \SU_n(\bF_q) \ar[r] \ar[d] & \PSU_n(\bF_q) \ar[r] \ar[d] & 1\\
1 \ar[r] & Z(\GU_n(\bF_q))\simeq (\bF_{q^2}^\times)^{(1)} \ar[r] \ar[d]_{x\mapsto x^n} & \GU_n(\bF_q) \ar[r] \ar[d]_{\rm det} & \PGU_n(\bF_q) \ar[r] \ar[d]_{\overline{\rm det}} & 1\\
1 \ar[r] & ((\bF_{q^2}^\times)^{(1)})^n=((\bF_{q^2}^\times)^{(1)})^v \ar[r] \ar[d] & (\bF_{q^2}^\times)^{(1)} \ar[r] \ar[d] & (\bF_{q^2}^\times)^{(1)}/((\bF_{q^2}^\times)^{(1)})^n \ar[r] \ar[d] & 1\\
 & 1 & 1 & 1 & 
}
\end{align*}
where $(\bF_{q^2}^\times)^{(1)}=
{\rm Ker}\{N_{\bF_{\!q^2}/\bF_q}:\bF_{q^2}^\times\to \bF_{q}^\times\}$ 
and $v=\gcd(n,q+1)$ 
with $Z(\SU_n(\bF_q))\simeq M(\PSU_n(\bF_q))\simeq \bZ/v\bZ$, 
$(\bF_{q^2}^\times)^{(1)}\simeq C_{q+1}$ and 
$(\bF_{q^2}^\times)^{(1)}/((\bF_{q^2}^\times)^{(1)})^n\simeq C_v$. 

For the $3$ exceptional cases, 
see Gorenstein \cite[Table 4.1, page 302]{Gor82} 
with Lie type ${}^2\!A_{n}(q)=\PSU_{n+1}(\bF_q)$, 
see also Remark \ref{rem8.4}. 
\begin{remark}\label{rem8.4}
(1) We see that $G=\PSU_n(\bF_q)$ $(n\geq 3)$ 
is simple except for the $1$ case $(n,q)=(3,2)$:\\ 
(i) $G=\PSU_3(\bF_2)\simeq (C_3)^2\rtimes Q_8$ with 
$M(G)\simeq \bZ/3\bZ$.\\
(2) The $3$ exceptional cases with $(n,q)=(4,2),(4,3),(6,2)$ satisfy that\\
(i) $G=\SU_4(\bF_2)=\PSU_4(\bF_2)=\PGU(\bF_2)$ 
with $M(G)\simeq \bZ/2\bZ$, 
$G=\GU_4(\bF_2)$ with $M(G)\simeq \bZ/2\bZ$;\\
(ii) 
$G=\SU_4(\bF_3)$ with $M(G)\simeq (\bZ/3\bZ)^{\oplus 2}$, 
$G=\PSU_4(\bF_3)$ with $M(G)\simeq \bZ/4\bZ\oplus (\bZ/3\bZ)^{\oplus 2}$, 
$G=\GU_4(\bF_3)$ with $M(G)=0$, 
$G=\PGU_4(\bF_3)$ with $M(G)\simeq \bZ/4\bZ$;\\
(iii) 
$G=\SU_6(\bF_2)$ with $M(G)\simeq (\bZ/2\bZ)^{\oplus 2}$, 
$G=\PSU_6(\bF_2)$ with $M(G)\simeq (\bZ/2\bZ)^{\oplus 2}\oplus \bZ/3\bZ$, 
$G=\GU_6(\bF_2)$ with $M(G)=0$, 
$G=\PGU_6(\bF_2)$ with $M(G)\simeq \bZ/3\bZ$\\
(see Gorenstein \cite[Theorem 2.13, Table 4.1]{Gor82}, 
Karpilovsky \cite[Section 8.4, Table 8.5]{Kar87} with Lie type ${}^2\!A_{n}(q)=\PSU_{n+1}(\bF_q)$ and GAP computations as in Section \ref{GAPcomp2}). 
We also note that 
$\SU_2(\bF_q)\simeq \SL_2(\bF_q)$, 
$\PSU_2(\bF_q)\simeq \PSL_2(\bF_q)$, 
$\PGU_2(\bF_q)\simeq \PGL_2(\bF_q)$. 
\end{remark}
\end{example}
%
\begin{example}[$G=\Sp_{2n}(\bF_q)$: the symplectic group of degree $2n$ over $\bF_q$ with $M(G)=0$; $G=\PSp_{2n}(\bF_q)$: the projective symplectic group of degree $2n$ over $\bF_q$ with $M(G)=0$]\label{exapp1}
Let $n\geq 2$ be an integer and $q=p^r\geq 2$ be prime power. 
Assume that $(n,q)\neq (2,2),(3,2)$. 
Then\\
{\rm (1)}
If $G=\Sp_{2n}(\bF_q)$, 
then $\Sha^2_\omega(G,J_G)=H^2(G,J_G)\simeq M(G)=0$;\\
{\rm (2)} 
If $G=\PSp_{2n}(\bF_q)$, 
then $\Sha^2_\omega(G,J_G)=H^2(G,J_G)\simeq M(G)\simeq \bZ/w\bZ$ 
where $w=\gcd(2,q-1)$. In particular, 
if $\gcd(2,q-1)=1$, i.e. $q=2^r$, 
then $\PSp_{2n}(\bF_q)={\rm Sp}_{2n}(\bF_q)$ and 
$\Sha^2_\omega(G,J_G)=H^2(G,J_G)\simeq M(G)=0$. 

When $k$ is a global field, $\Sha^2_\omega(G,J_G)=0$ implies that 
$A(T)=0$, i.e. 
$T$ has the weak approximation property, 
and $\Sha(T)=0$, i.e. the Hasse norm principle holds for $K/k$ 
$($that is, Hasse principle holds for all torsors $E$ under $T$$)$ 
where $T=R^{(1)}_{K/k}(\bG_{m,K})$ is the norm one torus of $K/k$  
$($see Section \ref{S3} and Ono's theorem $($Theorem \ref{thOno}$))$. 

Indeed, we have $M(\PSp_{2n}(\bF_q))\simeq \bZ/w\bZ$ 
(see Gorenstein \cite[Table 4.1, page 302]{Gor82}, 
Karpilovsky \cite[Table 8.5, page 283]{Kar87} 
with Lie type $C_{n}(q)=\PSp_{2n}(\bF_q)$, 
see also Sonn \cite[Proposition 8]{Son94}) 
and a Schur cover of $\PSp_{2n}(\bF_q)$: 
\begin{align*}
1\to Z({\rm Sp}_{2n}(\bF_q))\simeq M(\PSp_{2n}(\bF_q)\simeq 
\bZ/w\bZ\to {\rm Sp}_{2n}(\bF_q)\to \PSp_{2n}(\bF_q)\to 1.
\end{align*}
For the $2$ exceptional cases, 
see Gorenstein \cite[Table 4.1, page 302]{Gor82} 
with Lie type $C_{n}(q)=\PSp_{2n}(\bF_q)$, 
see also Remark \ref{rem8.6}. 
\begin{remark}\label{rem8.6}
(1) We see that $G=\PSp_{2n}(\bF_q)$ $(n\geq 2)$ 
is simple except for the $1$ case $(n,q)=(2,2)$:\\ 
(i) $G=\PSp_4(\bF_2)\simeq S_6$ with $M(G)\simeq \bZ/2\bZ$.\\
(2) The $2$ exceptional cases with $(n,q)=(2,2),(3,2)$ satisfy that\\
(i) $G=\Sp_4(\bF_2)=\PSp_4(\bF_2)\simeq S_6$ 
with $M(G)\simeq \bZ/2\bZ$;\\
(ii) $G=\Sp_6(\bF_2)=\PSp_6(\bF_2)$ 
with $M(G)\simeq \bZ/2\bZ$\\
(see Gorenstein \cite[Theorem 2.13, Table 4.1]{Gor82}, 
Karpilovsky \cite[Section 8.4, Table 8.5]{Kar87} with Lie type 
$C_{n}(q)=\PSp_{2n}(\bF_q)$ and GAP computations as in Section \ref{GAPcomp2}, 
see also Sonn \cite[Proposition 8, page 404]{Son94}). 
We note that $\Sp_2(\bF_q)\simeq \SL_2(\bF_q)$, 
$\PSp_2(\bF_q)\simeq \PSL_2(\bF_q)$. 
\end{remark}
\end{example}
%

\begin{example}[The $13$ sporadic finite simple groups $G$ with $M(G)=0$]\label{exapp3}
If $G$ is one of the following $13$ sporadic finite simple groups: 
Mathieu groups $M_{11}$, $M_{23}$, $M_{24}$; 
Janko groups $J_1$, $J_4$; 
Conway groups $Co_2$, $Co_3$; 
Fischer group $Fi_{23}$; 
Held group $He$; 
Harada-Norton group $HN$; 
Lyons group $Ly$; 
Thompson group $Th$; 
Monster group $M$, 
then $\Sha^2_\omega(G,J_G)=H^2(G,J_G)\simeq M(G)=0$. 

When $k$ is a global field, $\Sha^2_\omega(G,J_G)=0$ implies that 
$A(T)=0$, i.e. 
$T$ has the weak approximation property, 
and $\Sha(T)=0$, i.e. the Hasse norm principle holds for $K/k$ 
$($that is, Hasse principle holds for all torsors $E$ under $T$$)$ 
where $T=R^{(1)}_{K/k}(\bG_{m,K})$ is the norm one torus of $K/k$ 
$($see Section \ref{S3} and Ono's theorem $($Theorem \ref{thOno}$))$. 

For Schur multipliers $M(G)$ of the $26$ sporadic finite simple groups $G$, 
see Gorenstein \cite[Table 4.1, page 203]{Gor82}, 
Karpilovsky \cite[Section 8.4, Table 8.5, page 283]{Kar87}. 
\end{example}

\section{GAP computations: Schur multipliers $M(G)$ and $M(H)$ as in Table 1, Table 2, Table 3 and Table 4}\label{GAPcomp}

The GAP algorithms and related ones can be available as {\tt HNP.gap} 
in \cite{Norm1ToriHNP}. 

\begin{verbatim}
gap> Read("HNP.gap");
gap> for n in [2..4] do
> for m in [1..NrTransitiveGroups(n)] do
> G:=TransitiveGroup(n,m); # G=nTm
> H:=Stabilizer(G,1); # H with [G:H]=n
> if GroupCohomology(G,3)=[] and Intersection(DerivedSubgroup(G),H)=DerivedSubgroup(H) then
> Print([n,m],"\t",StructureDescription(G:short),"\t",StructureDescription(H:short),"\n");
> fi;od;od;
[ 2, 1 ]	2	1
[ 3, 1 ]	3	1
[ 3, 2 ]	S3	2
[ 4, 1 ]	4	1
\end{verbatim}

\begin{verbatim}
gap> Read("HNP.gap");
gap> for n in [2..4] do
> for m in [1..NrTransitiveGroups(n)] do
> G:=TransitiveGroup(n,m); # G=nTm
> H:=Stabilizer(G,1); # H with [G:H]=n
> if GroupCohomology(G,3)<>[] and GroupCohomology(H,3)=[] and 
> Intersection(DerivedSubgroup(G),H)=DerivedSubgroup(H) then
> Print([n,m],"\t",StructureDescription(G:short),"\t",
> StructureDescription(H:short),"\t",GroupCohomology(G,3),"\n");
> fi;od;od;
[ 4, 2 ]	2^2	1	[ 2 ]
[ 4, 3 ]	D8	2	[ 2 ]
[ 4, 4 ]	A4	3	[ 2 ]
[ 4, 5 ]	S4	S3	[ 2 ]
\end{verbatim}

\begin{verbatim}
gap> Read("HNP.gap");
gap> for n in [5..30] do
> for m in [1..NrTransitiveGroups(n)-2] do # An and Sn can be omitted
> G:=TransitiveGroup(n,m); # G=nTm
> H:=Stabilizer(G,1); # H with [G:H]=n
> if IsMetacyclic(G)=true and GroupCohomology(G,3)=[] and 
> Intersection(DerivedSubgroup(G),H)=DerivedSubgroup(H) then
> Print([n,m],"\t",StructureDescription(G:short),"\t",StructureDescription(H:short),"\n");
> fi;od;od;
[ 5, 1 ]	5	1
[ 5, 2 ]	D10	2
[ 5, 3 ]	5:4	4
[ 6, 1 ]	6	1
[ 6, 2 ]	S3	1
[ 6, 5 ]	3xS3	3
[ 7, 1 ]	7	1
[ 7, 2 ]	D14	2
[ 7, 3 ]	7:3	3
[ 7, 4 ]	7:6	6
[ 8, 1 ]	8	1
[ 8, 5 ]	Q8	1
[ 8, 7 ]	8:2	2
[ 8, 8 ]	QD16	2
[ 9, 1 ]	9	1
[ 9, 3 ]	D18	2
[ 9, 4 ]	3xS3	2
[ 9, 6 ]	9:3	3
[ 9, 10 ]	9:6	6
[ 10, 1 ]	10	1
[ 10, 2 ]	D10	1
[ 10, 4 ]	5:4	2
[ 10, 6 ]	5xD10	5
[ 11, 1 ]	11	1
[ 11, 2 ]	D22	2
[ 11, 3 ]	11:5	5
[ 11, 4 ]	11:10	10
[ 12, 1 ]	12	1
[ 12, 5 ]	3:4	1
[ 12, 19 ]	3x(3:4)	3
[ 13, 1 ]	13	1
[ 13, 2 ]	D26	2
[ 13, 3 ]	13:3	3
[ 13, 4 ]	13:4	4
[ 13, 5 ]	13:6	6
[ 13, 6 ]	13:12	12
[ 14, 1 ]	14	1
[ 14, 2 ]	D14	1
[ 14, 4 ]	7:6	3
[ 14, 5 ]	2x(7:3)	3
[ 14, 8 ]	7xD14	7
[ 15, 1 ]	15	1
[ 15, 2 ]	D30	2
[ 15, 3 ]	3xD10	2
[ 15, 4 ]	5xS3	2
[ 15, 6 ]	15:4	4
[ 15, 8 ]	3x(5:4)	4
[ 16, 1 ]	16	1
[ 16, 6 ]	8:2	1
[ 16, 12 ]	QD16	1
[ 16, 14 ]	Q16	1
[ 16, 22 ]	16:2	2
[ 16, 49 ]	4.D8=4.(4x2)	2
[ 16, 55 ]	QD32	2
[ 16, 124 ]	8.D8 = 4.(8x2)	4
[ 16, 125 ]	16:4	4
[ 16, 136 ]	16:4	4
[ 17, 1 ]	17	1
[ 17, 2 ]	D34	2
[ 17, 3 ]	17:4	4
[ 17, 4 ]	17:8	8
[ 17, 5 ]	17:16	16
[ 18, 1 ]	18	1
[ 18, 3 ]	3xS3	1
[ 18, 5 ]	D18	1
[ 18, 14 ]	2x(9:3)	3
[ 18, 16 ]	9xS3	3
[ 18, 18 ]	9:6	3
[ 18, 19 ]	3xD18	3
[ 18, 74 ]	9xD18	9
[ 18, 80 ]	(9:9):2	9
[ 19, 1 ]	19	1
[ 19, 2 ]	D38	2
[ 19, 3 ]	19:3	3
[ 19, 4 ]	19:6	6
[ 19, 5 ]	19:9	9
[ 19, 6 ]	19:18	18
[ 20, 1 ]	20	1
[ 20, 2 ]	5:4	1
[ 20, 5 ]	5:4	1
[ 20, 25 ]	5x(5:4)	5
[ 20, 29 ]	5x(5:4)	5
[ 21, 1 ]	21	1
[ 21, 2 ]	7:3	1
[ 21, 3 ]	3xD14	2
[ 21, 4 ]	7:6	2
[ 21, 5 ]	D42	2
[ 21, 6 ]	7xS3	2
[ 21, 10 ]	7:(3xS3)	6
[ 21, 11 ]	(7:3)xS3	6
[ 21, 13 ]	7x(7:3)	7
[ 22, 1 ]	22	1
[ 22, 2 ]	D22	1
[ 22, 4 ]	11:10	5
[ 22, 5 ]	2x(11:5)	5
[ 22, 7 ]	11xD22	11
[ 23, 1 ]	23	1
[ 23, 2 ]	D46	2
[ 23, 3 ]	23:11	11
[ 23, 4 ]	23:22	22
[ 24, 1 ]	24	1
[ 24, 4 ]	3xQ8	1
[ 24, 5 ]	3:Q8	1
[ 24, 8 ]	3:8	1
[ 24, 16 ]	3x(8:2)	2
[ 24, 20 ]	(3:8):2	2
[ 24, 31 ]	24:2	2
[ 24, 35 ]	24:2	2
[ 24, 41 ]	3xQD16	2
[ 24, 64 ]	3x(3:Q8)	3
[ 24, 69 ]	3x(3:8)	3
[ 24, 211 ]	3x((3:8):2)	6
[ 25, 1 ]	25	1
[ 25, 3 ]	5xD10	2
[ 25, 4 ]	D50	2
[ 25, 7 ]	5x(5:4)	4
[ 25, 8 ]	25:4	4
[ 25, 13 ]	25:5	5
[ 25, 25 ]	(25:5):2	10
[ 25, 40 ]	(25:5):4	20
[ 26, 1 ]	26	1
[ 26, 2 ]	D26	1
[ 26, 4 ]	13:4	2
[ 26, 5 ]	2x(13:3)	3
[ 26, 6 ]	13:6	3
[ 26, 8 ]	13:12	6
[ 26, 11 ]	13xD26	13
[ 27, 1 ]	27	1
[ 27, 5 ]	9:3	1
[ 27, 8 ]	D54	2
[ 27, 9 ]	3xD18	2
[ 27, 12 ]	9xS3	2
[ 27, 14 ]	9:6	2
[ 27, 22 ]	27:3	3
[ 27, 55 ]	(27:3):2	6
[ 27, 107 ]	27:9	9
[ 27, 176 ]	(27:9):2	18
[ 28, 1 ]	28	1
[ 28, 3 ]	7:4	1
[ 28, 12 ]	7:12	3
[ 28, 13 ]	4x(7:3)	3
[ 28, 33 ]	7x(7:4)	7
[ 29, 1 ]	29	1
[ 29, 2 ]	D58	2
[ 29, 3 ]	29:4	4
[ 29, 4 ]	29:7	7
[ 29, 5 ]	29:14	14
[ 29, 6 ]	29:28	28
[ 30, 1 ]	30	1
[ 30, 2 ]	5xS3	1
[ 30, 3 ]	D30	1
[ 30, 4 ]	3xD10	1
[ 30, 6 ]	15:4	2
[ 30, 7 ]	3x(5:4)	2
[ 30, 15 ]	15xS3	3
[ 30, 16 ]	3xD30	3
[ 30, 36 ]	5xD30	5
[ 30, 39 ]	15xD10	5
[ 30, 47 ]	3x(15:4)	6
[ 30, 104 ]	15xD30	15
\end{verbatim}

\begin{verbatim}
gap> Read("HNP.gap");
gap> for n in [5..30] do
> for m in [1..NrTransitiveGroups(n)-2] do # An and Sn can be omitted
> G:=TransitiveGroup(n,m); # G=nTm
> H:=Stabilizer(G,1); # H with [G:H]=n
> if IsMetacyclic(G)=true and GroupCohomology(G,3)<>[] and GroupCohomology(H,3)=[] and 
> Intersection(DerivedSubgroup(G),H)=DerivedSubgroup(H) then
> Print([n,m],"\t",StructureDescription(G:short),"\t",
> StructureDescription(H:short),"\t",GroupCohomology(G,3),"\n");
> fi;od;od;
[ 6, 3 ]	D12	2	[ 2 ]
[ 8, 2 ]	4x2	1	[ 2 ]
[ 8, 4 ]	D8	1	[ 2 ]
[ 8, 6 ]	D16	2	[ 2 ]
[ 9, 2 ]	3^2	1	[ 3 ]
[ 10, 3 ]	D20	2	[ 2 ]
[ 10, 5 ]	2x(5:4)	4	[ 2 ]
[ 12, 2 ]	6x2	1	[ 2 ]
[ 12, 3 ]	D12	1	[ 2 ]
[ 12, 11 ]	4xS3	2	[ 2 ]
[ 12, 12 ]	D24	2	[ 2 ]
[ 12, 14 ]	3xD8	2	[ 2 ]
[ 12, 18 ]	6xS3	3	[ 2 ]
[ 14, 3 ]	D28	2	[ 2 ]
[ 14, 7 ]	2x(7:6)	6	[ 2 ]
[ 16, 4 ]	4^2	1	[ 4 ]
[ 16, 5 ]	8x2	1	[ 2 ]
[ 16, 8 ]	4:4	1	[ 2 ]
[ 16, 13 ]	D16	1	[ 2 ]
[ 16, 56 ]	D32	2	[ 2 ]
[ 18, 2 ]	6x3	1	[ 3 ]
[ 18, 6 ]	6xS3	2	[ 2 ]
[ 18, 13 ]	D36	2	[ 2 ]
[ 18, 45 ]	2x(9:6)	6	[ 2 ]
[ 20, 3 ]	10x2	1	[ 2 ]
[ 20, 4 ]	D20	1	[ 2 ]
[ 20, 6 ]	4xD10	2	[ 2 ]
[ 20, 9 ]	2x(5:4)	2	[ 2 ]
[ 20, 10 ]	D40	2	[ 2 ]
[ 20, 12 ]	5xD8	2	[ 2 ]
[ 20, 13 ]	2x(5:4)	2	[ 2 ]
[ 20, 18 ]	20:4	4	[ 2 ]
[ 20, 20 ]	4x(5:4)	4	[ 4 ]
[ 20, 24 ]	10xD10	5	[ 2 ]
[ 21, 7 ]	3x(7:3)	3	[ 3 ]
[ 21, 9 ]	3x(7:6)	6	[ 3 ]
[ 22, 3 ]	D44	2	[ 2 ]
[ 22, 6 ]	2x(11:10)	10	[ 2 ]
[ 24, 2 ]	12x2	1	[ 2 ]
[ 24, 6 ]	2x(3:4)	1	[ 2 ]
[ 24, 12 ]	4xS3	1	[ 2 ]
[ 24, 13 ]	D24	1	[ 2 ]
[ 24, 15 ]	3xD8	1	[ 2 ]
[ 24, 32 ]	8xS3	2	[ 2 ]
[ 24, 34 ]	D48	2	[ 2 ]
[ 24, 40 ]	3xD16	2	[ 2 ]
[ 24, 65 ]	12xS3	3	[ 2 ]
[ 24, 66 ]	6x(3:4)	3	[ 2 ]
[ 24, 67 ]	3xD24	3	[ 2 ]
[ 25, 2 ]	5^2	1	[ 5 ]
[ 26, 3 ]	D52	2	[ 2 ]
[ 26, 7 ]	2x(13:4)	4	[ 2 ]
[ 26, 9 ]	2x(13:6)	6	[ 2 ]
[ 26, 10 ]	2x(13:12)	12	[ 2 ]
[ 27, 2 ]	9x3	1	[ 3 ]
[ 28, 2 ]	14x2	1	[ 2 ]
[ 28, 4 ]	D28	1	[ 2 ]
[ 28, 5 ]	7xD8	2	[ 2 ]
[ 28, 8 ]	4xD14	2	[ 2 ]
[ 28, 10 ]	D56	2	[ 2 ]
[ 28, 14 ]	2^2x(7:3)	3	[ 2 ]
[ 28, 15 ]	2x(7:6)	3	[ 2 ]
[ 28, 22 ]	D8x(7:3)	6	[ 2 ]
[ 28, 23 ]	7:(3xD8)	6	[ 2 ]
[ 28, 26 ]	4x(7:6)	6	[ 2 ]
[ 28, 34 ]	14xD14	7	[ 2 ]
[ 30, 5 ]	6xD10	2	[ 2 ]
[ 30, 12 ]	10xS3	2	[ 2 ]
[ 30, 14 ]	D60	2	[ 2 ]
[ 30, 17 ]	2x(15:4)	4	[ 2 ]
[ 30, 26 ]	6x(5:4)	4	[ 2 ]
\end{verbatim}

\begin{verbatim}
gap> Read("HNP.gap");
gap> for n in [5..19] do
> for m in [1..NrTransitiveGroups(n)-2] do # An and Sn can be omitted
> G:=TransitiveGroup(n,m); # G=nTm
> H:=Stabilizer(G,1); # H with [G:H]=n
> if IsMetacyclic(G)=false and GroupCohomology(G,3)=[] and 
> Intersection(DerivedSubgroup(G),H)=DerivedSubgroup(H) then
> Print([n,m],"\t",StructureDescription(G:short),"\t",StructureDescription(H:short),"\n");
> fi;od;od;
[ 8, 12 ]	SL(2,3)	3
[ 8, 23 ]	GL(2,3)	S3
[ 8, 25 ]	(2^3):7	7
[ 8, 36 ]	2^3:(7:3)	7:3
[ 9, 12 ]	((3^2):3):2	S3
[ 9, 15 ]	(3^2):8	8
[ 9, 19 ]	3^2:QD16	QD16
[ 9, 20 ]	((3^3):3):2	3xS3
[ 9, 26 ]	(((3^2):Q8):3):2	GL(2,3)
[ 9, 32 ]	PSL(2,8):3	2^3:(7:3)
[ 10, 18 ]	5^2:8	5:4
[ 12, 46 ]	(3^2):8	S3
[ 12, 272 ]	M11	PSL(2,11)
[ 14, 11 ]	2^3:(7:3)	A4
[ 14, 14 ]	7^2:6	7:3
[ 14, 18 ]	2x(2^3:(7:3))	2xA4
[ 14, 23 ]	7^2:12	7:6
[ 15, 13 ]	5^2:S3	D10
[ 15, 19 ]	5^2:12	5:4
[ 15, 32 ]	5x(5^2:S3)	5xD10
[ 15, 38 ]	5^3:12	(5^2):4
[ 15, 41 ]	3^4:(5:4)	3^3:4
[ 15, 56 ]	3x(3^4:(5:4))	3x(3^3:4)
[ 16, 59 ]	2xSL(2,3)	3
[ 16, 60 ]	((4x2):2):3	3
[ 16, 196 ]	2x((2^3):7)	7
[ 16, 439 ]	(4.4^2):3	A4
[ 16, 447 ]	2^4:15	15
[ 16, 712 ]	2x(2^3:(7:3))	7:3
[ 16, 728 ]	((4.4^2):2):3	2xA4
[ 16, 732 ]	((4.4^2):2):3	2xA4
[ 16, 773 ]	((4.4^2):3):2	S4
[ 16, 777 ]	(((2^4):5):2):3	3xD10
[ 16, 1064 ]	((4.4^2):3):4	A4:4
[ 16, 1075 ]	((2x((2^4):2)):2):7	(2^3):7
[ 16, 1079 ]	(2^4:15):4	15:4
[ 16, 1501 ]	((2x((2^4):2)):2):(7:3)	2^3:(7:3)
[ 16, 1503 ]	((2x((2^3):(2^2))):2):(7:3)	2^3:(7:3)
[ 16, 1798 ]	(((2^6:7):2):7):3	2^3:(7^2:3)
[ 17, 8 ]	PSL(2,16):4	(2^4:15):4
[ 18, 28 ]	(3^2):8	4
[ 18, 49 ]	((3^2):3):4	S3
[ 18, 158 ]	(9^2:3):2	9:3
[ 18, 280 ]	3^4:16	(3^2):8
[ 18, 427 ]	2x(PSL(2,8):3)	2^3:(7:3)
[ 18, 937 ]	(PSL(2,8)xPSL(2,8)):6	2^3:(7:(PSL(2,8):3))
\end{verbatim}

\begin{verbatim}
gap> Read("HNP.gap");
gap> for n in [5..19] do
> for m in [1..NrTransitiveGroups(n)-2] do # An and Sn can be omitted
> G:=TransitiveGroup(n,m); # G=nTm
> H:=Stabilizer(G,1); # H with [G:H]=n
> if IsMetacyclic(G)=false and GroupCohomology(G,3)<>[] and GroupCohomology(H,3)=[] and
> Intersection(DerivedSubgroup(G),H)=DerivedSubgroup(H) then
> Print([n,m],"\t",StructureDescription(G:short),"\t",
> StructureDescription(H:short),"\t",GroupCohomology(G,3),"\n");
> fi;od;od;
[ 6, 9 ]	S3xS3	S3	[ 2 ]
[ 6, 10 ]	(3^2):4	S3	[ 3 ]
[ 8, 3 ]	2^3	1	[ 2, 2, 2 ]
[ 8, 9 ]	2xD8	2	[ 2, 2, 2 ]
[ 8, 10 ]	(4x2):2	2	[ 2, 2 ]
[ 8, 11 ]	(4x2):2	2	[ 2, 2 ]
[ 8, 13 ]	2xA4	3	[ 2 ]
[ 8, 17 ]	(4^2):2	4	[ 2 ]
[ 8, 19 ]	(2^3):4	4	[ 2, 2 ]
[ 8, 24 ]	2xS4	S3	[ 2, 2 ]
[ 9, 5 ]	(3^2):2	2	[ 3 ]
[ 9, 7 ]	(3^2):3	3	[ 3, 3 ]
[ 9, 9 ]	(3^2):4	4	[ 3 ]
[ 9, 11 ]	(3^2):6	6	[ 3 ]
[ 9, 13 ]	(3^2):6	S3	[ 3 ]
[ 9, 14 ]	(3^2):Q8	Q8	[ 3 ]
[ 9, 23 ]	((3^2):Q8):3	SL(2,3)	[ 3 ]
[ 10, 9 ]	D10xD10	D10	[ 2 ]
[ 10, 10 ]	(5^2):4	D10	[ 5 ]
[ 10, 17 ]	5^2:(4x2)	5:4	[ 2 ]
[ 10, 35 ]	(A6.2):2	3^2:QD16	[ 2 ]
[ 12, 4 ]	A4	1	[ 2 ]
[ 12, 6 ]	2xA4	2	[ 2 ]
[ 12, 8 ]	S4	2	[ 2 ]
[ 12, 10 ]	2^2xS3	2	[ 2, 2, 2 ]
[ 12, 13 ]	(6x2):2	2	[ 2 ]
[ 12, 15 ]	(6x2):2	2	[ 2 ]
[ 12, 20 ]	3xA4	3	[ 2, 3 ]
[ 12, 27 ]	A4:4	4	[ 2 ]
[ 12, 35 ]	(S3xS3):2	S3	[ 2 ]
[ 12, 37 ]	2xS3xS3	S3	[ 2, 2, 2 ]
[ 12, 38 ]	(6xS3):2	S3	[ 2 ]
[ 12, 39 ]	(3^2):(4x2)	S3	[ 2 ]
[ 12, 40 ]	2x((3^2):4)	S3	[ 2, 3 ]
[ 12, 41 ]	2x((3^2):4)	S3	[ 2, 3 ]
[ 12, 42 ]	3x((6x2):2)	6	[ 2 ]
[ 12, 43 ]	A4xS3	6	[ 2 ]
[ 12, 44 ]	(3xA4):2	S3	[ 2, 3 ]
[ 12, 45 ]	3xS4	S3	[ 2 ]
[ 12, 76 ]	2xA5	D10	[ 2 ]
[ 12, 112 ]	(((4^2):3):2):2	QD16	[ 2, 2 ]
[ 12, 121 ]	3x((S3xS3):2)	3xS3	[ 2 ]
[ 12, 124 ]	A5:4	5:4	[ 2 ]
[ 12, 175 ]	((3^3:2^2):3):2	((3^2):3):2	[ 2 ]
[ 12, 231 ]	3x(((3^3:2^2):3):2)	((3^3):3):2	[ 2 ]
[ 12, 233 ]	(3x((3^3:2^2):3)):2	((3^3):3):2	[ 2, 3 ]
[ 12, 295 ]	M12	M11	[ 2 ]
[ 14, 12 ]	7^2:4	D14	[ 7 ]
[ 14, 13 ]	D14xD14	D14	[ 2 ]
[ 14, 24 ]	7^2:(6x2)	7:6	[ 2 ]
[ 15, 12 ]	5^2:6	D10	[ 5 ]
[ 15, 17 ]	5^2:(3:4)	5:4	[ 5 ]
[ 16, 2 ]	4x2^2	1	[ 2, 2, 2 ]
[ 16, 3 ]	2^4	1	[ 2, 2, 2, 2, 2, 2 ]
[ 16, 7 ]	2xQ8	1	[ 2, 2 ]
[ 16, 9 ]	2xD8	1	[ 2, 2, 2 ]
[ 16, 10 ]	(4x2):2	1	[ 2, 2 ]
[ 16, 11 ]	(4x2):2	1	[ 2, 2 ]
[ 16, 15 ]	2x(8:2)	2	[ 2, 2 ]
[ 16, 16 ]	(8x2):2	2	[ 2, 2 ]
[ 16, 17 ]	(4^2):2	2	[ 2, 4 ]
[ 16, 18 ]	2x((4x2):2)	2	[ 2, 2, 2, 2, 2 ]
[ 16, 19 ]	4xD8	2	[ 2, 2, 2 ]
[ 16, 20 ]	(2xQ8):2	2	[ 2, 2, 2, 2, 2 ]
[ 16, 21 ]	2x((4x2):2)	2	[ 2, 2, 2, 2 ]
[ 16, 23 ]	(2^3):(2^2)	2	[ 2, 2, 2, 2, 2 ]
[ 16, 24 ]	(8x2):2	2	[ 2, 2 ]
[ 16, 25 ]	2^2xD8	2	[ 2, 2, 2, 2, 2, 2 ]
[ 16, 26 ]	(8x2):2	2	[ 2, 2 ]
[ 16, 27 ]	(4^2):2	2	[ 2 ]
[ 16, 28 ]	(4^2):2	2	[ 2 ]
[ 16, 29 ]	2xD16	2	[ 2, 2, 2 ]
[ 16, 30 ]	(4^2):2	2	[ 2, 4 ]
[ 16, 31 ]	(2xQ8):2	2	[ 2, 2 ]
[ 16, 32 ]	(2xQ8):2	2	[ 2, 2 ]
[ 16, 33 ]	(2^3):4	2	[ 2, 2 ]
[ 16, 34 ]	(4x2^2):2	2	[ 2, 2, 2 ]
[ 16, 35 ]	8:(2^2)	2	[ 2, 2 ]
[ 16, 37 ]	(4x2^2):2	2	[ 2, 2 ]
[ 16, 38 ]	8:(2^2)	2	[ 2, 2 ]
[ 16, 39 ]	(2^4):2	2	[ 2, 2, 2, 2 ]
[ 16, 41 ]	(8:2):2	2	[ 2 ]
[ 16, 42 ]	(4^2):2	2	[ 2 ]
[ 16, 43 ]	(4x2^2):2	2	[ 2, 2, 2 ]
[ 16, 44 ]	(8x2):2	2	[ 2, 2 ]
[ 16, 45 ]	8:(2^2)	2	[ 2, 2 ]
[ 16, 46 ]	(2^4):2	2	[ 2, 2, 2, 2 ]
[ 16, 47 ]	(8x2):2	2	[ 2, 2 ]
[ 16, 48 ]	2xQD16	2	[ 2, 2 ]
[ 16, 50 ]	(2xQ8):2	2	[ 2, 2 ]
[ 16, 51 ]	(4^2):2	2	[ 2, 2, 4 ]
[ 16, 52 ]	(2^3):4	2	[ 2, 2 ]
[ 16, 54 ]	(4x2^2):2	2	[ 2, 2 ]
[ 16, 57 ]	4xA4	3	[ 2 ]
[ 16, 58 ]	2^2xA4	3	[ 2, 2 ]
[ 16, 63 ]	(4^2):3	3	[ 4 ]
[ 16, 64 ]	(2^4):3	3	[ 2, 2, 2, 2 ]
[ 16, 74 ]	4^2:4	4	[ 2, 4 ]
[ 16, 76 ]	2x((2^3):4)	4	[ 2, 2, 2, 2 ]
[ 16, 96 ]	(4x2^2):4	4	[ 2, 2, 2 ]
[ 16, 107 ]	((4^2):2):2	4	[ 2, 2, 2 ]
[ 16, 110 ]	(8:2):4	4	[ 2, 2 ]
[ 16, 111 ]	2x((4^2):2)	4	[ 2, 2, 2 ]
[ 16, 113 ]	((4^2):2):2	4	[ 2, 2 ]
[ 16, 114 ]	((4^2):2):2	4	[ 2, 2 ]
[ 16, 120 ]	((2^3):4):2	4	[ 2, 2, 4 ]
[ 16, 121 ]	4^2:4	4	[ 2, 2 ]
[ 16, 143 ]	((2^3):4):2	4	[ 2, 4 ]
[ 16, 148 ]	((4x2):2):4	4	[ 2 ]
[ 16, 156 ]	(8:4):2	4	[ 2 ]
[ 16, 161 ]	(2xQ8):4	4	[ 2 ]
[ 16, 163 ]	((4x2):4):2	4	[ 2, 2 ]
[ 16, 166 ]	((2^3):4):2	4	[ 2, 2 ]
[ 16, 176 ]	4^2:4	4	[ 2, 2 ]
[ 16, 178 ]	(2^4):5	5	[ 2, 2 ]
[ 16, 179 ]	D8xA4	6	[ 2, 2 ]
[ 16, 180 ]	((2xQ8):2):3	6	[ 2 ]
[ 16, 181 ]	4xS4	S3	[ 2, 2 ]
[ 16, 182 ]	2^2xS4	S3	[ 2, 2, 2, 2 ]
[ 16, 183 ]	((2^4):2):3	6	[ 2, 2 ]
[ 16, 184 ]	((4^2):2):3	6	[ 2 ]
[ 16, 185 ]	((4^2):2):3	6	[ 4 ]
[ 16, 186 ]	(A4:4):2	S3	[ 2, 2 ]
[ 16, 187 ]	GL(2,3):2	S3	[ 2 ]
[ 16, 188 ]	2xGL(2,3)	S3	[ 2 ]
[ 16, 189 ]	(2.S4=SL(2,3).2):2	S3	[ 2 ]
[ 16, 190 ]	GL(2,3):2	S3	[ 2 ]
[ 16, 191 ]	(2xS4):2	S3	[ 2, 2 ]
[ 16, 194 ]	((2^4):3):2	S3	[ 2, 2, 2 ]
[ 16, 195 ]	((4^2):3):2	S3	[ 2 ]
[ 16, 260 ]	(8.D8=4.(8x2)):2	8	[ 2 ]
[ 16, 289 ]	8^2:2	8	[ 2 ]
[ 16, 332 ]	(((2xQ8):2):2):2	Q8	[ 2, 2, 2 ]
[ 16, 338 ]	((4xQ8):2):2	Q8	[ 2, 2, 2 ]
[ 16, 351 ]	(8.D8=4.(8x2)):2	Q8	[ 2 ]
[ 16, 357 ]	(((4^2):2):2):2	Q8	[ 2, 2, 2 ]
[ 16, 370 ]	(2.((2^3):4)=(4x2).(4x2)):2	Q8	[ 2, 2, 2 ]
[ 16, 380 ]	(((2^3):4):2):2	Q8	[ 2, 2, 2 ]
[ 16, 381 ]	(((2^3):4):2):2	Q8	[ 2, 2, 4 ]
[ 16, 415 ]	((2^4):5):2	D10	[ 2, 2 ]
[ 16, 430 ]	((4^2):3):4	3:4	[ 4 ]
[ 16, 433 ]	((2^4):3):4	3:4	[ 2 ]
[ 16, 504 ]	((4x(8:2)):2):2	8:2	[ 2, 2, 4 ]
[ 16, 548 ]	(4^2:4):4	8:2	[ 2, 4 ]
[ 16, 568 ]	(8^2:2):2	8:2	[ 2, 2 ]
[ 16, 576 ]	((8:4):2):4	8:2	[ 2, 2 ]
[ 16, 580 ]	(16:4):4	8:2	[ 2 ]
[ 16, 601 ]	(((8x2):2):2):4	8:2	[ 2, 4 ]
[ 16, 605 ]	(((4:8):2):2):2	8:2	[ 2, 2, 4 ]
[ 16, 619 ]	((2.((4x2):2)=(2^2).(4x2)):2):4	8:2	[ 2, 2 ]
[ 16, 627 ]	(((8:4):2):2):2	8:2	[ 2, 2, 2 ]
[ 16, 655 ]	(((4:8):2):2):2	8:2	[ 2, 2, 2 ]
[ 16, 668 ]	(8:8):4	QD16	[ 8 ]
[ 16, 669 ]	((((2^3):4):2):2):2	QD16	[ 2, 2, 2 ]
[ 16, 672 ]	(((2.((4x2):2)=(2^2).(4x2)):2):2):2	QD16	[ 2, 2, 2 ]
[ 16, 674 ]	((2.((2^3):4)=(4x2).(4x2)):2):2	QD16	[ 2, 2, 2 ]
[ 16, 681 ]	(((2^2).(2^3)):4):2	QD16	[ 2 ]
[ 16, 683 ]	((4:8):2):4	QD16	[ 2 ]
[ 16, 687 ]	((8:Q8):2):2	QD16	[ 2, 2 ]
[ 16, 694 ]	8^2:4	QD16	[ 4 ]
[ 16, 704 ]	(2^2.((4x2):2)=(4x2).(4x2)):4	QD16	[ 2 ]
[ 16, 706 ]	(((2xQD16):2):2):2	QD16	[ 2, 2 ]
[ 16, 708 ]	(((2^4):3):2):3	3xS3	[ 2 ]
[ 16, 709 ]	S4xA4	3xS3	[ 2, 2 ]
[ 16, 711 ]	((2^4):5):4	5:4	[ 2 ]
[ 16, 730 ]	((Q8xQ8):2):3	SL(2,3)	[ 2 ]
[ 16, 734 ]	((((4^2):2):2):2):3	SL(2,3)	[ 2 ]
[ 16, 735 ]	((((2xQ8):2):2):2):3	SL(2,3)	[ 2 ]
[ 16, 1062 ]	(((((4^2):2):2):3):2):2	GL(2,3)	[ 2, 2 ]
[ 16, 1067 ]	(((2.((2^3):4)=(4x2).(4x2)):2):3):2	GL(2,3)	[ 2, 2 ]
[ 16, 1076 ]	(2.(2^3.2^3)=2^4.2^3):7	(2^3):7	[ 2 ]
[ 16, 1077 ]	((2x((2^3):(2^2))):2):7	(2^3):7	[ 2, 2 ]
[ 16, 1502 ]	((2^6:7):2):3	2^3:(7:3)	[ 2 ]
[ 18, 4 ]	(3^2):2	1	[ 3 ]
[ 18, 9 ]	S3xS3	2	[ 2 ]
[ 18, 10 ]	(3^2):4	2	[ 3 ]
[ 18, 11 ]	S3xS3	2	[ 2 ]
[ 18, 12 ]	2x((3^2):2)	2	[ 2, 3 ]
[ 18, 15 ]	2x((3^2):3)	3	[ 3, 3 ]
[ 18, 17 ]	3^2xS3	3	[ 3 ]
[ 18, 21 ]	(3^2):6	3	[ 3 ]
[ 18, 22 ]	(3^2):6	3	[ 3 ]
[ 18, 23 ]	3x((3^2):2)	3	[ 3 ]
[ 18, 27 ]	2x((3^2):4)	4	[ 2, 3 ]
[ 18, 41 ]	2x((3^2):6)	6	[ 2, 3 ]
[ 18, 42 ]	2x((3^2):6)	S3	[ 2, 3 ]
[ 18, 43 ]	3xS3xS3	S3	[ 2 ]
[ 18, 44 ]	3x((3^2):4)	S3	[ 3 ]
[ 18, 46 ]	3xS3xS3	6	[ 2 ]
[ 18, 50 ]	D18xS3	S3	[ 2 ]
[ 18, 51 ]	((3^2):3):2^2	S3	[ 2 ]
[ 18, 52 ]	2x(((3^2):3):2)	S3	[ 2 ]
[ 18, 53 ]	3^3:2^2	S3	[ 2 ]
[ 18, 54 ]	3^3:4	S3	[ 3 ]
[ 18, 56 ]	((3^2):3):2^2	S3	[ 2 ]
[ 18, 57 ]	((3^2):3):2^2	S3	[ 2 ]
[ 18, 58 ]	((3^2):2)xS3	S3	[ 2, 3 ]
[ 18, 59 ]	2x((3^2):8)	8	[ 2 ]
[ 18, 64 ]	2x((3^2):Q8)	Q8	[ 2, 2, 3 ]
[ 18, 110 ]	2x(3^2:QD16)	QD16	[ 2, 2 ]
[ 18, 118 ]	3x(((3^2):3):2^2)	3xS3	[ 2 ]
[ 18, 119 ]	2x(((3^3):3):2)	3xS3	[ 2 ]
[ 18, 120 ]	3x(3^3:2^2)	3xS3	[ 2 ]
[ 18, 121 ]	((3^2):6)xS3	3xS3	[ 2, 3 ]
[ 18, 122 ]	(9:6)xS3	3xS3	[ 2 ]
[ 18, 123 ]	3x(3^3:4)	3xS3	[ 3 ]
[ 18, 126 ]	3x(((3^2):3):2^2)	3xS3	[ 2 ]
[ 18, 130 ]	9^2:4	D18	[ 9 ]
[ 18, 140 ]	D18xD18	D18	[ 2 ]
[ 18, 151 ]	2x(((3^2):Q8):3)	SL(2,3)	[ 3 ]
[ 18, 229 ]	2x((((3^2):Q8):3):2)	GL(2,3)	[ 2 ]
[ 18, 233 ]	((9:9):3):2^2	9:6	[ 2 ]
[ 18, 235 ]	((9:9):3):4	9:6	[ 3 ]
[ 18, 244 ]	((3x((3^2):3)):3):2^2	((3^2):3):2	[ 2 ]
[ 18, 281 ]	3^4:(8x2)	(3^2):8	[ 2 ]
[ 18, 383 ]	3^4:((8x2):2)	3^2:QD16	[ 2, 2 ]
[ 18, 385 ]	3^4:(8:(2^2))	3^2:QD16	[ 2, 2 ]
[ 18, 386 ]	3^4:((2xQ8):2)	3^2:QD16	[ 2, 2, 3 ]
[ 18, 393 ]	3^4:(2xQD16)	3^2:QD16	[ 2, 2 ]
[ 18, 524 ]	(((3^4:Q8):3):2):2	(((3^2):Q8):3):2	[ 2 ]
[ 18, 526 ]	(((3^4:Q8):3):2):2	(((3^2):Q8):3):2	[ 2 ]
[ 18, 527 ]	((3^4:(2xQ8)):3):2	(((3^2):Q8):3):2	[ 2, 3 ]
[ 18, 528 ]	((3^4:((4x2):2)):3):2	(((3^2):Q8):3):2	[ 2 ]
\end{verbatim}

\section{GAP computations: Schur multipliers $M(G)$ as in Remark \ref{rem8.2}, Remark \ref{rem8.4} and Remark \ref{rem8.6}}\label{GAPcomp2}

\begin{verbatim}
gap> LoadPackage("HAP");
true
gap> LoadPackage("Sonata");
true
gap> IsIsomorphicGroup(SL(2,4),PSL(2,4));
true
gap> IsIsomorphicGroup(PGL(2,4),PSL(2,4));
true
gap> GroupCohomology(PSL(2,4),3);
[ 2 ]
gap> GroupCohomology(GL(2,4),3);
[ 2 ]
gap> GroupCohomology(SL(2,9),3);
[ 3 ]
gap> GroupCohomology(PSL(2,9),3);
[ 2, 3 ]
gap> GroupCohomology(GL(2,9),3);
[  ]
gap> GroupCohomology(PGL(2,9),3);
[ 2 ]
gap> IsIsomorphicGroup(SL(3,2),PSL(3,2));
true
gap> IsIsomorphicGroup(GL(3,2),PSL(3,2));
true
gap> IsIsomorphicGroup(PGL(3,2),PSL(3,2));
true
gap> GroupCohomology(PSL(3,2),3);
[ 2 ]
gap> GroupCohomology(SL(3,4),3);
[ 4, 4 ]
gap> GroupCohomology(PSL(3,4),3);
[ 4, 4, 3 ]
gap> GroupCohomology(GL(3,4),3);
[  ]
gap> GroupCohomology(PGL(3,4),3);
[ 3 ]
gap> IsIsomorphicGroup(SL(4,2),PSL(4,2));
true
gap> IsIsomorphicGroup(GL(4,2),PSL(4,2));
true
gap> IsIsomorphicGroup(PGL(4,2),PSL(4,2));
true
gap> GroupCohomology(PSL(4,2),3);
[ 2 ]

gap> IsIsomorphicGroup(SU(4,2),PSU(4,2));
true
gap> IsIsomorphicGroup(PGU(4,2),PSU(4,2));
true
gap> GroupCohomology(PSU(4,2),3);
[ 2 ]
gap> GroupCohomology(GU(4,2),3);
[ 2 ]
gap> GroupCohomology(SU(4,3),3);
[ 3, 3 ]
gap> GroupCohomology(PSU(4,3),3);
[ 4, 3, 3 ]
gap> GroupCohomology(GU(4,3),3);
[  ]
gap> GroupCohomology(PGU(4,3),3);
[ 4 ]
gap> GroupCohomology(SU(6,2),3);
[ 2, 2 ]
gap> GroupCohomology(PSU(6,2),3);
[ 2, 2, 3 ]
gap> GroupCohomology(GU(6,2),3);
[  ]
gap> GroupCohomology(PGU(6,2),3);
[ 3 ]

gap> StructureDescription(PSp(4,2));
"S6"
gap> IsIsomorphicGroup(Sp(4,2),PSp(4,2));
true
gap> GroupCohomology(PSp(4,2),3);
[ 2 ]
gap> IsIsomorphicGroup(Sp(6,2),PSp(6,2));
true
gap> GroupCohomology(PSp(6,2),3);
[ 2 ]
\end{verbatim}

%
%
%


\begin{thebibliography}{Norm1ToriHNP}
\bibitem[Bar81a]{Bar81a} 
H.-J. Bartels, 
{\it Zur Arithmetik von Konjugationsklassen in algebraischen Gruppen}, 
J. Algebra \textbf{70} (1981) 179--199. 
\bibitem[Bar81b]{Bar81b} 
H.-J. Bartels, 
{\it Zur Arithmetik von Diedergruppenerweiterungen}, 
Math. Ann. \textbf{256} (1981) 465--473.
\bibitem[BP20]{BP20} 
E. Bayer-Fluckiger, R. Parimala, 
{\it On unramified Brauer groups of torsors over tori}, 
Doc. Math. \textbf{25} (2020) 1263--1284. 
\bibitem[Bey73]{Bey73} 
F. R. Beyl, 
{\it The Schur multiplicator of metacyclic groups}, 
Proc. Amer. Math. Soc. \textbf{40} (1973) 413--418.
\bibitem[BJ84]{BJ84} 
F. R. Beyl, M. R. Jones, 
{\it Addendum: ``The Schur multiplicator of metacyclic groups'' (Proc. Amer. Math. Soc. 40 (1973), 413--418) by F. R. Beyl}, 
Proc. Amer. Math. Soc. \textbf{43} (1974) 251--252.
\bibitem[BT82]{BT82} 
F. R. Beyl, J. Tappe, 
{\it Group extensions, representations, and the Schur multiplicator}, 
Lecture Notes in Mathematics, 958, Springer-Verlag, Berlin-New York, 1982. 
\bibitem[Bro82]{Bro82}
K. S. Brown, 
{\it Cohomology of groups}, 
Graduate Texts in Mathematics, 87, Springer-Verlag, New York-Berlin, 
1982, x+306 pp. 
\bibitem[BM83]{BM83}
G. Butler, J. McKay, 
{\it The transitive groups of degree up to eleven}, 
Comm. Algebra \textbf{11} (1983) 863--911.
\bibitem[BN16]{BN16}
T. D. Browning, R. Newton, 
{\it The proportion of failures of the Hasse norm principle}, 
Mathematika \textbf{62} (2016) 337--347. 
\bibitem[CTHS05]{CTHS05} 
J.-L. Colliot-Th\'{e}l\`{e}ne, D. Harari, A. N. Skorobogatov, 
{\it Compactification \'equivariante d'un tore 
(d'apr\`es Brylinski et K\"unnemann)}, 
Expo. Math. \textbf{23} (2005) 161--170.
\bibitem[CTS77]{CTS77} 
J.-L. Colliot-Th\'{e}l\`{e}ne, J.-J. Sansuc, 
{\it La R-\'{e}quivalence sur les tores}, 
Ann. Sci. \'{E}cole Norm. Sup. (4) \textbf{10} (1977) 175--229. 
\bibitem[CTS87]{CTS87} 
J.-L. Colliot-Th\'{e}l\`{e}ne, J.-J. Sansuc, 
{\it Principal homogeneous spaces under flasque tori: Applications}, 
J. Algebra \textbf{106} (1987) 148--205. 
\bibitem[CTS07]{CTS07} 
J.-L. Colliot-Th\'{e}l\`{e}ne, J.-J. Sansuc, 
{\it The rationality problem for fields of invariants under linear 
algebraic groups (with special regards to the Brauer group)}, 
Algebraic groups and homogeneous spaces, 113--186, 
Tata Inst. Fund. Res. Stud. Math., 19, Tata Inst. Fund. Res., Mumbai, 2007.
\bibitem[CTS21]{CTS21} 
J.-L. Colliot-Th\'{e}l\`{e}ne, A. N. Skorobogatov, 
{\it The Brauer-Grothendieck group}, 
Ergeb. Math. Grenzgeb. (3) 71,  
Springer, Cham, 2021, xv+453 pp.
\bibitem[CK00]{CK00} 
A. Cortella, B. Kunyavskii, 
{\it Rationality problem for generic tori in simple groups}, 
J. Algebra \textbf{225} (2000) 771--793. 
\bibitem[DP87]{DP87} 
Yu. A. Drakokhrust, V. P. Platonov, 
{\it The Hasse norm principle for algebraic number fields}, (Russian)
Izv. Akad. Nauk SSSR Ser. Mat. \textbf{50} (1986) 946--968; 
translation in Math. USSR-Izv. \textbf{29} (1987) 299--322. 
\bibitem[End11]{End11} 
S. Endo, 
{\it The rationality problem for norm one tori}, 
Nagoya Math. J. \textbf{202} (2011) 83--106. 
\bibitem[EM73]{EM73} 
S. Endo, T. Miyata, 
{\it Invariants of finite abelian groups}, 
J. Math. Soc. Japan \textbf{25} (1973) 7--26. 
\bibitem[EM75]{EM75} 
S. Endo, T. Miyata, 
{\it On a classification of the function fields of algebraic tori}, 
Nagoya Math. J. \textbf{56} (1975) 85--104. 
\bibitem[Flo]{Flo} 
M. Florence, 
{\it Non rationality of some norm-one tori}, preprint (2006).
\bibitem[FLN18]{FLN18}
C. Frei, D. Loughran, R. Newton, 
{\it The Hasse norm principle for abelian extensions}, 
Amer. J. Math. \textbf{140} (2018) 1639--1685. 
\bibitem[GAP]{GAP} 
The GAP Group, GAP -- Groups, Algorithms, and Programming, 
Version 4.9.3; 2018. (\url{http://www.gap-system.org}).
\bibitem[Ger77]{Ger77} 
F. Gerth III, 
{\it The Hasse norm principle in metacyclic extensions of number fields}, 
J. London Math. Soc. (2) \textbf{16} (1977) 203--208. 
\bibitem[Ger78]{Ger78} 
F. Gerth III, 
{\it The Hasse norm principle in cyclotomic number fields}, 
J. Reine Angew. Math. \textbf{303/304} (1978) 249--252. 
\bibitem[Gor80]{Gor80}
D. Gorenstein, {\it Finite groups}, Second edition, 
Chelsea Publishing Co., New York, 1980, xvii+519 pp.
\bibitem[Gor82]{Gor82}
D. Gorenstein, 
{\it Finite simple groups. An introduction to their classification}, 
Univ. Ser. Math., Plenum Publishing Corp., New York, 1982. x+333 pp. 
\bibitem[Gur78a]{Gur78a} 
S. Gurak, {\it On the Hasse norm principle}, 
J. Reine Angew. Math. \textbf{299/300} (1978) 16--27.
\bibitem[Gur78b]{Gur78b}
S. Gurak, 
{\it The Hasse norm principle in non-abelian extensions}, 
J. Reine Angew. Math. \textbf{303/304} (1978) 314--318.
\bibitem[Gur80]{Gur80} 
S. Gurak, 
{\it The Hasse norm principle in a compositum of radical extensions}, 
J. London Math. Soc. (2) \textbf{22} (1980) 385--397.
\bibitem[Hal59]{Hal59} 
M. Hall, 
{\it The theory of groups}, 
The Macmillan Company, New York, 1959, xiii+434 pp. 
\bibitem[Har20]{Har20} 
D. Harari, {\it Galois cohomology and class field theory}, 
Translated from the 2017 French original by Andrei Yafaev, 
Universitext, Springer, Cham, 2020, xiv+338 pp. 
\bibitem[HHY20]{HHY20} 
S. Hasegawa, A. Hoshi, A. Yamasaki, 
{\it Rationality problem for norm one tori in small dimensions}, 
Math. Comp. \textbf{89} (2020) 923--940. 
\bibitem[Has31]{Has31} 
H. Hasse, 
{\it Beweis eines Satzes und Wiederlegung einer Vermutung \"uber das allgemeine Normenrestsymbol}, 
Nachrichten von der Gesellschaft der Wissenschaften zu G\"ottingen, Mathematisch-Physikalische Klasse 
(1931) 64--69. 
\bibitem[HKY22]{HKY22}
A. Hoshi, K. Kanai, A. Yamasaki, 
{\it Norm one tori and Hasse norm principle}, 
Math. Comp. \textbf{91} (2022) 2431--2458. 
\bibitem[HKY23]{HKY23}
A. Hoshi, K. Kanai, A. Yamasaki, 
{\it Norm one tori and Hasse norm principle, II: Degree $12$ case}, 
J. Number Theory \textbf{244} (2023) 84--110. 
\bibitem[HKY25]{HKY25}
A. Hoshi, K. Kanai, A. Yamasaki, 
{\it Norm one tori and Hasse norm principle, III: Degree $16$ case}, 
J. Algebra \textbf{666} (2025) 794--820. 
\bibitem[HKY]{HKY} 
A. Hoshi, K. Kanai, A. Yamasaki, 
{\it Hasse norm principle for $M_{11}$ and $J_1$ extensions}, 
arXiv:2210.09119.
\bibitem[HY17]{HY17} 
A. Hoshi, A. Yamasaki, 
{\it Rationality problem for algebraic tori}, 
Mem. Amer. Math. Soc. \textbf{248} (2017) no. 1176, v+215 pp. 
\bibitem[HY21]{HY21} 
A. Hoshi, A. Yamasaki, 
{\it Rationality problem for norm one tori}, 
Israel J. Math. \textbf{241} (2021) 849--867. 
\bibitem[HY24]{HY24}
A. Hoshi, A. Yamasaki, 
{\it Rationality problem for norm one tori for dihedral extensions}, 
J. Algebra \textbf{640} (2024) 368--384. 
\bibitem[HY1]{HY1} 
A. Hoshi, A. Yamasaki, 
{\it Birational classification for algebraic tori}, 
arXiv:2112.02280.  
\bibitem[HY2]{HY2} 
A. Hoshi, A. Yamasaki, 
{\it Rationality problem for norm one tori for $A_5$ and ${\rm PSL}_2(\mathbbm{F}_8)$ extensions}, 
arXiv:2309.16187. 
\bibitem[H\"{u}r84]{Hur84} 
W. H\"{u}rlimann, 
{\it On algebraic tori of norm type}, Comment. Math. Helv. 
\textbf{59} (1984) 539--549.
\bibitem[Kar87]{Kar87} 
G. Karpilovsky, 
{\it The Schur multiplier}, 
London Mathematical Society Monographs. 
New Series, 2. The Clarendon Press, Oxford University Press, New York, 1987.
\bibitem[Kar93]{Kar93}
G. Karpilovsky, 
{\it Group representations. Vol. 2},  
North-Holland Math. Stud., 177, North-Holland Publishing Co., Amsterdam, 1993. xvi+902 pp. 
\bibitem[Kun84]{Kun84} 
B. E. Kunyavskii, 
{\it Arithmetic properties of three-dimensional algebraic tori}, (Russian)
Integral lattices and finite linear groups, 
Zap. Nauchn. Sem. Leningrad. Otdel. Mat. Inst. Steklov. (LOMI) \textbf{116} (1982) 102--107, 163; 
translation in J. Soviet Math. \textbf{26} (1984) 1898--1901. 
\bibitem[LL00]{LL00} 
N. Lemire, M. Lorenz, 
{\it On certain lattices associated with generic division algebras}, 
J. Group Theory \textbf{3} (2000) 385--405.
\bibitem[LeB95]{LeB95} 
L. Le Bruyn, {\it Generic norm one tori},
Nieuw Arch. Wisk. (4) \textbf{13} (1995) 401--407. 
\bibitem[Mac20]{Mac20} 
A. Macedo, 
{\it The Hasse norm principle for $A_n$-extensions}, 
J. Number Theory \textbf{211} (2020) 500--512.
\bibitem[MN22]{MN22} 
A. Macedo, R. Newton, 
{\it Explicit methods for the Hasse norm principle and applications to $A_n$ and $S_n$ extensions}, 
Math. Proc. Cambridge Philos. Soc. \textbf{172} (2022) 489--529. 
\bibitem[Man74]{Man74} 
Yu. I. Manin, 
{\it Cubic forms: algebra, geometry, arithmetic}, 
Translated from the Russian by M. Hazewinkel. 
North-Holland Mathematical Library, Vol. 4. 
North-Holland Publishing Co., Amsterdam-London; 
American Elsevier Publishing Co., New York, 1974, vii+292 pp.
North-Holland Mathematical Library 4, North-Holland, Amsterdam, 1974.
\bibitem[MT86]{MT86}
Yu. I. Manin, M. A. Tsfasman, 
{\it Rational varieties: algebra, geometry, arithmetic}, (Russian)
Uspekhi Mat. Nauk \textbf{41} (1986) 43--94; 
translation in Russian Math. Surveys \textbf{41} (1986) 51--116. 
\bibitem[Mil86]{Mil86}
Milne, J. S.
{\it Arithmetic duality theorems}, 
Perspect. Math., 1, 
Academic Press, Inc., Boston, MA, 1986, x+421 pp.
\bibitem[NSW00]{NSW00}
J. Neukirch, A. Schmidt, K. Wingberg, 
{\it Cohomology of number fields}, 
Grundlehren Math. Wiss., 323, Springer-Verlag, Berlin, 2000, xvi+699 pp.
\bibitem[Neu55]{Neu55} 
B. H. Neumann, 
{\it On some finite groups with trivial multiplicator}, 
Publ. Math. Debrecen \textbf{4} (1956) 190--194.
\bibitem[Norm1ToriHNP]{Norm1ToriHNP} 
A. Hoshi, A. Yamasaki, 
Norm1ToriHNP for GAP 4 ver.2025.7.18, 
available at \url{https://doi.org/10.57723/kds594279}, 
\url{http://mathweb.sc.niigata-u.ac.jp/~hoshi/Algorithm/Norm1ToriHNP/}, 
\url{https://www.math.kyoto-u.ac.jp/~yamasaki/Algorithm/Norm1ToriHNP/}. 
\bibitem[Ono61]{Ono61} 
T. Ono, 
{\it Arithmetic of algebraic tori}, 
Ann. of Math. (2) \textbf{74} (1961) 101--139. 
\bibitem[Ono63]{Ono63} 
T. Ono, 
{\it On the Tamagawa number of algebraic tori}, 
Ann. of Math. (2) \textbf{78} (1963) 47--73.
\bibitem[Ono65]{Ono65} 
T. Ono, 
{\it On the relative theory of Tamagawa numbers}, 
Ann. of Math. (2) \textbf{82} (1965) 88--111.
\bibitem[Pla82]{Pla82} 
V. P. Platonov, 
{\it Arithmetic theory of algebraic groups}, (Russian)
Uspekhi Mat. Nauk \textbf{37} (1982) 3--54; 
translation in Russian Math. Surveys \textbf{37} (1982) 1--62. 
\bibitem[PR94]{PR94} 
V. P. Platonov, A. Rapinchuk, 
{\it Algebraic groups and number theory}, 
Translated from the 1991 Russian original by Rachel Rowen, 
Pure and applied mathematics, 139, Academic Press, 1994. 
\bibitem[Rob96]{Rob96} 
D. J. S. Robinson, 
{\it A course in the theory of groups}, Second edition, 
Graduate Texts in Mathematics, 80, Springer-Verlag, New York, 1996. 
\bibitem[Sal99]{Sal99} 
D. J. Saltman, 
{\it Lectures on division algebras}, 
CBMS Regional Conference Series in Mathematics, 94, Published by American Mathematical Society, 
Providence, RI; on behalf of Conference Board of the Mathematical Sciences, Washington, DC, 1999, viii+120 pp. 
\bibitem[San81]{San81} 
J.-J. Sansuc, 
{\it Groupe de Brauer et arithm\'etique des groupes alg\'ebriques lin\'eaires sur un corps de nombres}, (French) 
J. Reine Angew. Math. \textbf{327} (1981) 12--80.
\bibitem[Ser79]{Ser79} 
J-P. Serre, 
{\it Local fields}, 
Graduate Texts in Mathematics, 67, Springer-Verlag, Berlin, 1979. 
\bibitem[Son94]{Son94}
J. Sonn, 
{\it Rigidity and embedding problems over $\mathbb{Q}_{ab}(t)$}, 
J. Number Theory \textbf{47} (1994) 398--404. 
\bibitem[Suz55]{Suz55} 
M. Suzuki, 
{\it On finite groups with cyclic Sylow subgroups for all odd primes}, 
Amer. J. Math. \textbf{77} (1955) 657--691.
\bibitem[Suz82]{Suz82}
M. Suzuki, {\it Group theory. I}, 
Translated from the Japanese by the author
Grundlehren der Mathematischen Wissenschaften, 247, 
Springer-Verlag, Berlin-New York, 1982, xiv+434 pp. 
\bibitem[Suz86]{Suz86} 
M. Suzuki, {\it Group theory. II}, 
Translated from the Japanese
rundlehren der Mathematischen Wissenschaften, 248, 
Springer-Verlag, New York, 1986, x+621 pp.
\bibitem[Tat67]{Tat67} 
J. Tate, 
{\it Global class field theory}, 
Algebraic Number Theory (Proc. Instructional Conf., Brighton, 1965), 
162--203, Thompson, Washington, D.C., 1967.  
\bibitem[Vos69]{Vos69} 
V. E. Voskresenskii, 
{\it The birational equivalence of linear algebraic groups}, (Russian) 
Dokl. Akad. Nauk SSSR \textbf{188} (1969) 978--981; 
erratum, ibid. 191 1969 nos., 1, 2, 3, vii; 
translation in Soviet Math. Dokl. \textbf{10} (1969) 1212--1215. 
\bibitem[Vos70]{Vos70} 
V. E. Voskresenskii, 
{\it Birational properties of linear algebraic groups}, (Russian) 
Izv. Akad. Nauk SSSR Ser. Mat. \textbf{34} (1970) 3--19; 
translation in Math. USSR-Izv. \textbf{4} (1970) 1--17.
\bibitem[Vos98]{Vos98} 
V. E. Voskresenskii, 
{\it Algebraic groups and their birational invariants}, 
Translated from the Russian manuscript by Boris Kunyavskii, 
Translations of Mathematical Monographs, 179. 
American Mathematical Society, Providence, RI, 1998.
\bibitem[VK84]{VK84} 
V. E. Voskresenskii,  B. E. Kunyavskii, 
{\it Maximal tori in semisimple algebraic groups}, 
Kuibyshev State Inst., Kuibyshev (1984). 
Deposited in VINITI March 5, 1984, No. 1269-84 Dep. 
(Ref. Zh. Mat. (1984), 7A405 Dep.).
H. Zassenhaus, 
{\it \"Uber endliche Fastk\"orper}, (German) 
Abh. Math. Sem. Univ. Hamburg \textbf{11} (1935) 187--220.
\end{thebibliography}
\end{document}